\documentclass[11pt]{amsart}

\usepackage{adjustbox,epsfig}

\usepackage{mathtools}
\usepackage{enumerate}
\usepackage{color}
 \usepackage[bookmarks=true,
bookmarksnumbered=true, breaklinks=true,
pdfstartview=FitH, hyperfigures=false,
plainpages=false, naturalnames=true,
colorlinks=false,
pdfpagelabels]{hyperref}

\makeatletter
\newtheorem*{rep@theorem}{\rep@title}
\newcommand{\newreptheorem}[2]{%
\newenvironment{rep#1}[1]{%
 \def\rep@title{#2 \ref{##1}}%
 \begin{rep@theorem}}%
 {\end{rep@theorem}}}
\makeatother
\newreptheorem{theorem}{Theorem}
\newreptheorem{corollary}{Corollary}

\usepackage{color}
\usepackage{verbatim}

\newcommand{\R}{\mathbb{R}} 

\newcommand{\Rn}{\mathbb{R}^n} 
\def\LC{\operatorname{LC}_n}
\newcommand \B[1][n]{B_2^{#1}}
\newcommand{\N}{\mathbb{N}}
\newcommand{\vol}[1][n]{\operatorname{vol}_{#1}}
\newcommand{\Vol}{\operatorname{vol}_{nm}}
\def \s{\mathbb{S}^{n-1}}

\def \M{\mathcal{M}}
\newcommand {\conbode}[1][n] {\mathcal{K}^{#1}_e}
\newcommand {\conbodI}[1][n] {\mathcal{K}^{#1}_I}
\newcommand {\conbodo}[1][nm] {\mathcal{K}^{#1}_o}
\newcommand {\conbodio}[1][n] {\mathcal{K}^{#1}_{(o)}}

\newcommand {\Lu}[1]{\mathcal{L}^{#1}}
\newcommand {\Ld}[1]{\mathcal{L}_{#1}}

\mathtoolsset{showonlyrefs}

\newtheorem{theorem}{Theorem}[section]
\newtheorem{definition}[theorem]{Definition}
\newtheorem{lemma}[theorem]{Lemma}
\newtheorem{proposition}[theorem]{Proposition}
\newtheorem{corollary}[theorem]{Corollary}

\newtheorem{rem}[theorem]{Remark}

\title[Cost-Santal\'o inequalities for Polish measure spaces]{On cost-induced Santal\'o-type inequalities \\ in Polish measure spaces}

\author[Langharst]{Dylan Langharst}
\address{Department of Mathematical Sciences, Carnegie Mellon University, Pittsburgh, PA 15213, USA}
\email{dlanghar@andrew.cmu.edu}
\thanks{D.L. is supported by US NSF grant DMS-2502744}

\author[Malliaris]{Andreas Malliaris}
\address{Institut de Math\'ematiques de Toulouse (UMR 5219). University of Toulouse \& CNRS. UPS, F-31062
Toulouse Cedex 09, France.}
\email{andreas.malliaris@math.univ-toulouse.fr}
\thanks{A.M. received support from the Graduate School EUR-MINT (State support managed by the
National Research Agency for Future Investments program bearing the reference ANR-18-EURE-0023)}

\author[Roysdon]{Michael Roysdon}
\address{Department of Mathematical Sciences, University of Cincinnati, Cincinnati, OH, USA and Department of Mathematics, The Ohio State University, Columbus OH, USA}
\email{roysdon.3@osu.edu}
\thanks{M.R. is supported by US NSF grant DMS-2548742.}


\begin{document}
\begin{abstract}
    We introduce a framework for establishing Blaschke-Santal\'o-type inequalities on $m$-tuples of Polish measure spaces coupled together by a continuous cost function. Central to our approach is a transference principle, which provides a mechanism to lift geometric weighted inequalities involving cost-polar sets into functional integral inequalities of Santal\'o-type. We call these equivalent inequalities \textit{cost-Santal\'o inequalities}. This definition expands and includes previous notions in the literature.
    
    We apply this principle to deduce several new versions of functional Santal\'o inequalities, including on the space of rectangular matrices and a functional sine Santal\'o inequality. A surprising development is that probability spaces with log-concave isoperimetric functions fit into our framework, for example, Gauss space and spherical space, leading to new functional Santal\'o inequalities in these settings. In particular, we obtain results for $\operatorname{RCD}(K,\infty)$ spaces. As a discrete application, we obtain an inequality for the Hamming cube. Finally, we explore applications to optimal transport, utilizing our functional framework to establish generalized transport-entropy inequalities on arbitrary Polish spaces satisfying a cost-Santal\'o inequality, which we explicitly instantiate for matrix spaces.
\end{abstract}
\maketitle
\tableofcontents
\section{Introduction}
\subsection{Motivation}

The concept of duality is pervasive throughout mathematics. A historical example is the Legendre transform on $\R^n
$: for a measurable function $\psi:\R^n\to\R\cup\{+\infty\}$ its Legendre transform $\mathcal{L}\psi:\R^n\to\R\cup\{+\infty\}$ is defined, for $v\in\R^n$, by
\begin{equation}\mathcal{L}\psi(v)=\sup_{u\in\Rn}[\langle v,u \rangle-\psi(u)].
\label{eq:legendre}
\end{equation}
Here, $\R^n$ is equipped with the Euclidean inner product $\langle \cdot,\cdot\rangle:\R^n\times \R^n\to\R$. The Legendre transform preserves the class $\operatorname{Cvx}(\R^n)$ of lower semicontinuous, convex functions on $\R^n$ with values in $\R\cup\{+\infty\}$, and serves as a duality on that space.

Moreover, this transform admits a natural analogue in the log-concave setting. The polar of a function $f:\R^n\to\R_+$ is defined by
\begin{equation}
\label{eq:classical_polarity}
f^\circ(v):=\exp\left[-\mathcal{L}(-\log f)(v)\right]=\inf_{u\in\Rn}\left[\frac{e^{-\langle v,u\rangle}}{f(u)}\right],\end{equation}
and is a log-concave function. Recall that a (Borel) measurable function $f:\R^n\to\R_+$ is log-concave if, for every $\lambda\in (0,1)$ and $x,y\in\R^n$, it holds
\[
f((1-\lambda)x+\lambda y)\geq f(x)^{1-\lambda}f(y)^\lambda.
\]
The polarity \eqref{eq:classical_polarity} preserves the class of functions (we write u.s.c for upper semicontinuous)
\begin{align*}
\LC:=\left\{f:\R^n\rightarrow \R_+: f\text{ is log-concave, u.s.c, and} \; \; 0\!<\!\int_{\R^n} f(x) d x \!< \infty\right\}.
\end{align*}

In his PhD thesis, Ball \cite{Ball_thesis} established that, for an even function $f:\R^n\to\R_+$ satisfying $0<\int_{\R^n}f(x)dx <\infty$, one has
\begin{equation}
\label{eq:ball_bS}
\left(\int_{\R^n}f(x)dx\right)\cdot \left(\int_{\R^n} f^\circ(x) dx\right) \leq (2\pi)^n.
\end{equation}
The works of Artstein, Klartag and Milman \cite{AAKM04}, Fradelizi and Meyer \cite{FM07}, and Lehec \cite{JL09} culminate in the following expansion of \eqref{eq:ball_bS}.
\begin{proposition}[Functional Blaschke-Santal\'o]
\label{prop:Ball_Fradelizi_Meyer}
Let $f:\R^n\to\R_+$ be a measurable function such that $0<\int f <+\infty$. Then, there exists $z\in \R^n$ with the following property: 
\\
for every measurable function $g:\R^n\to\R_+$ and measurable function $\Omega:\R_{+}\to\R_{+}$ satisfying
$$f(x+z)g(y)\le \Omega^{2}(\langle x,y\rangle ) \quad \text{for every } x,y\in \R^n \text{ such that } \langle x,y\rangle >0,$$
        one has
         \begin{equation}
         \label{eq:BS_fun}
         \int_{\R^n} f(x)dx\int_{\R^n} g(y)dy\le
        \left( \int_{\R^n} \Omega({|x|^{2}})dx\right)^2.\end{equation}
\end{proposition}
\noindent The case when $\Omega(t)=e^{-\frac{t}{2}}$ corresponds to the functional polarity given by \eqref{eq:classical_polarity}, as $g=f^\circ$ satisfies the hypotheses of Proposition~\ref{prop:Ball_Fradelizi_Meyer}.

This type of duality has a geometric counterpart. The polar of a compact set $K\subset \R^n$ with positive volume is
$$K^\circ =\{x\in \R^n: \langle x,y\rangle \leq 1,\, \forall \, y \in K\}.$$
Denote by $\conbodo[n]$ the space of \textit{convex bodies} in $\R^n$, where a convex body is a compact, convex set which contains the origin and has non-empty interior. Then, one has $\conbode\subset\conbodio\subset \conbodo[n]$; moreover, $\conbodio$, the subspace of convex bodies that have the origin in their interiors, and $\conbode$, the subspace of origin-symmetric convex bodies, are two spaces fixed by polarity. Recall that a set $A\subset \R^n$ is said to be origin-symmetric if $A=-A$.

In turn, the classical Blaschke-Santal\'o inequality  is precisely the geometric counterpart to Proposition~\ref{prop:Ball_Fradelizi_Meyer}. We recommend the works \cite{MP90,FMZ23,CG06,LZ97} for a rich overview of this inequality; usually, it is stated for convex bodies. We state it here for compact sets, which was proven recently in \cite{FGSZ23}: for $K\subset \R^n$ a compact set with positive volume such that $K$ or $K^\circ$ has center of mass at the origin, it holds
\begin{equation}
\label{eq:BS}
\vol(K)\vol(K^\circ) \leq \vol(\B)^2,
\end{equation}
with equality if and only if $K$ is a centered ellipsoid which may have had null sets removed.

In fact, the inequalities \eqref{eq:BS} and \eqref{eq:BS_fun} are formally equivalent. Indeed, recall that the Minkowski functional of a convex body $K$ is $\|v\|_{K}=\inf\{t>0:v\in tK\}$. If we define $\psi(v)=\frac{1}{2}\|v\|_K^2$, then $\mathcal{L}(\psi(v))=\frac{1}{2}\|v\|_{K^\circ}^2$, and, therefore,
\begin{equation}\exp\left[-\frac{\|v\|_K^2}{2}\right]^\circ=\exp\left[-\frac{\|v\|_{K^\circ}^2}{2}\right].
\label{eq:duality_exp}
\end{equation}
Consequently, using the choice $\Omega(t)=e^{-\frac{t}{2}}$, \eqref{eq:duality_exp} and the fact that
\begin{equation}\vol(K)=\frac{\vol(\B)}{(2\pi)^{\frac{n}{2}}}\int_{\Rn}e^{-\frac{\|v\|_K^2}{2}}dv,
\label{eq:norm_gauss_int}
\end{equation}
produces \eqref{eq:BS} from \eqref{eq:BS_fun} in the convex case. The converse direction in the convex case is not as immediate, and there are a few methods; for example, the method of K. Ball's bodies utilized in \cite{FM07}. 

One can try to deduce the compact case from the convex case using the fact that $K\subset \operatorname{conv}\left(K\cup\{o\}\right)=:L$ and $K^\circ=L^\circ$. Here, $\operatorname{conv} (M)$ is the smallest convex set (with respect to set inclusion) containing any non-empty $M\subset \R^n$. However, the hypothesis of \eqref{eq:BS} does not follow; this reduction does not capture the validity of \eqref{eq:BS} when $K$ itself has center of mass at the origin.

Recently, the second-named and third-named author together with Melbourne and Roberto \cite{MMRR26} introduced the notion of \textit{generalized sup-convolution} and showed that this concept lifts certain geometric inequalities to functional inequalities. We utilize this method and obtain the equivalence between \eqref{eq:BS} and \eqref{eq:BS_fun} in the compact case of \eqref{eq:BS} and any non-increasing $\Omega$ in \eqref{eq:BS_fun}; however, we must pay by placing a centering condition on the function $f$ in Proposition~\ref{prop:Ball_Fradelizi_Meyer}. In this way, we bypass the troublesome vector $z$. 

To give an appetizer, we recall the Pr\'ekopa-Leindler inequality \cite{PreL1,PreL2}, which states that for a given $\lambda\in(0,1)$, if $f,g,h: \R^n\to\R_+$ are a triple of functions satisfying
\begin{equation} h\left((1-\lambda)x+\lambda y\right)\geq f(x)^{1-\lambda}g(y)^{\lambda}, \quad \text{ for all } x,y \in \R^n,
\label{eq:PL_hypo}
\end{equation}
then it holds that
\begin{equation}
    \label{eq:PL}
    \int_{\R^n}h(x)dx \geq \left(\int_{\R^n}f(x)dx\right)^{1-\lambda}\left(\int_{\R^n}g(y)dy\right)^{\lambda}.
\end{equation}
The geometric version of the Pr\'ekopa-Leindler inequality is precisely the Brunn-Minkowski inequality (see, e.g., \cite{G20}). The associated sup-convolution is the classical one: for non-negative, measurable functions $f$ and $g$, one has, for a fixed $\lambda \in (0,1)$,
\begin{equation}
\label{eq:sup_con}
\left(f \square g\right)(z)=\sup_{\substack{x,y\in \R^n: \\ (1-\lambda) x + \lambda y = z}}f(x)^{1-\lambda}g(y)^{\lambda}.
\end{equation}
Notice that, if $h$ satisfies \eqref{eq:PL_hypo}, then $h \geq \left(f \square g\right)$ pointwise. Consequently, one is able to state the Pr\'ekopa-Leindler inequality using just $f, g$ and $\left(f \square g\right)$.

\subsection{Cost functions and the transference principle}
The Legendre transform, and consequently polarity, naturally extends to any real Hilbert space. In general, for any set $X$, it is common to introduce a symmetric function $c:X\times X\to (-\infty,\infty)$ to replace the role of the inner product in \eqref{eq:legendre} and \eqref{eq:classical_polarity}. Recently, a systematic characterization of such transforms was given by Artstein-Avidan, Sadovsky and Wyczesany \cite{AASW23_2}. In this work, we are interested in non-symmetric cost functions and the dualities they define.

We recall that a metric space $(X,d)$ is said to be a \text{Polish space} if it is complete and separable. If $\mu$ is a Borel measure on $X$, we call $(X,d,\mu)$ a Polish measure space. Throughout this work, every set or function called ``measurable'' is meant in the sense of Borel.

Let $(X_1,d_1,\mu_1),\dots,(X_m,d_m,\mu_m)$ be Polish measure spaces. We say an upper semi-continuous function $c:X_1\times \cdots \times X_m\to (-\infty,\infty)$ that is measurable with respect to the product Borel $\sigma$-algebra is a \textit{cost function}. There is a deep relationship  between cost functions and duality; as we recall below, the classical dualities \eqref{eq:legendre} and \eqref{eq:classical_polarity} are when $m=2$, $X_1=X_2=\R^n$ and $c(x,y)=-\langle x,y \rangle$. Sometimes in the literature, the cost function in the classical case is taken to be $\frac{1}{2}|x-y|^2=\frac{1}{2}|x|^2-\langle x,y \rangle +\frac{1}{2}|y|^2$; but the two are equivalent by re-normalizing by the quadratic terms $\frac{|\,\cdot\,|^2}{2}$.

Our main tool is the following transference principle. It pushes Santal\'o-type inequalities from subsets of a metric space to functions via a cost function relationship, and it is in the same spirit as K. Ball's proof of \eqref{eq:ball_bS}. We first mention some notation for its statement. 

Let $(X,\mu)$ be a Polish measure space. Let $\mathcal{C}$ be a collection of $\mu$-measurable subsets of $X$. Then, we denote the class of non-negative functions on $X$ which have almost all of their superlevel sets in $\mathcal{C}$ by
\[
\mathcal{F}\left(\mathcal{C}\right)=\left\{\mu\text{-measurable } f:X\to \R_+\,|\,\{f\geq r\}\in \mathcal{C}\,\, \text{for almost all } r \in (0,\|f\|_\infty)\right\}.
\]
Recall that every continuous, non-decreasing $F:(a,b]\to \R_+$ induces a positive, locally finite Lebesgue-Stieltjes measure $dF(r)$ on the Borel $\sigma$-algebra of $(a,b]$. If $F(a^+)=\lim_{r\to a^+}F(r)=0,$ then $F(t)=\int_a^tdF(r)$ for every $t\in (a,b)$. If $F$ is absolutely continuous, then $dF(r)=F^\prime(r)dr$ by the Radon–Nikodym theorem (see, e.g., \cite{Stein}).
\begin{lemma}[The transference principle]
\label{l:transfer}
For $m\in\N$, let $(X_1,\mu_{X_1}),\dots, (X_m,\mu_{X_m})$ be Polish measure spaces. Let $c \colon X_1 \times\cdots \times X_m \to (-\infty,\infty)$ be a continuous cost function. Fix $\lambda_1,\dots,\lambda_m \in (0,1)$ such that $\sum_{i=1}^m\lambda_i=1$.

Consider collections, for $i=1,\dots,m$, $\mathcal{C}_{X_i}$ of $\mu_{X_i}$-measurable subsets of $X_i$. Let $a,b\in \R\cup\{\pm\infty\}$ be such that for all non-empty $A_i\in \mathcal{C}_{X_i}$, $\sup_{\substack{x_i\in A_i \\ 1\leq i \leq m}} -c(x_1,\dots,x_m) \in [a,b]$ and consider a continuous, non-decreasing function $F \colon (a,b]\to [0,\infty)$ satisfying $F(a^+)=0$.

Then, the following are equivalent:\begin{enumerate}
    \item For every non-empty $A_1\in \mathcal{C}_{X_1},\dots,A_m\in\mathcal{C}_{X_m}$, one has
    \[
\prod_{i=1}^m\mu_{X_i}(A_i)^{\lambda_i} \leq F\left(\sup_{\substack{x_i\in A_i \\ 1\leq i \leq m}} -c(x_1,\dots,x_m) \right).
\]
\item For every $(m+1)$-tuple of measurable functions $f_1 \in \mathcal{F}(\mathcal{C}_{X_1}),\dots, f_m\in \mathcal{F}\left(\mathcal{C}_{X_m}\right)$ and $\Omega \colon (a,b]  \to [0,\infty)$, where $\Omega$ is non-increasing, that satisfy
\begin{align*}
\prod_{i=1}^mf_i(x_i)^{\lambda_i} &\leq \Omega(-c(x_1,\dots,x_m)) \text{ for all } 
\\
&(x_1,\dots,x_m)\in X_1\times\cdots\times  X_m, \text{ with } -c(x_1,\dots,x_m)\in (a,b],
\end{align*}
one has
\[
\prod_{i=1}^m\left(\int_{X_i} f_i(x) d\mu_{X_i}(x) \right)^{{\lambda_i}} \leq \int_{a}^b \Omega(r) dF(r).
\]
\end{enumerate}
\end{lemma}
Here, when $b=\infty$, the interval $(a,b]$ means $(a,\infty)$; symmetrically, when $a=-\infty$, the interval $[a,b)$ means $(-\infty,b)$. The same convention applies to $[a,b]$. We also use the convention that $0\times \infty=0$. Motivated by Lemma~\ref{l:transfer}, we introduce the following definition, which we may use from time-to-time.
\begin{definition}
\label{def:cost_sant}
    Let $m\in \N$. Let $a,b\in \R\cup\{\pm \infty\}$, with $a<b$, and let $F:(a,b]\to \R_+$ be a continuous, non-decreasing function such that $\lim_{r\to a^+}F(r)=0$. Let $\lambda_i\in (0,1)$ satisfy $\sum_{i=1}^m\lambda_i=1$ and set $\lambda=(\lambda_1,\dots,\lambda_m)$. 
    
    We say Polish measure spaces $(X_1,\mu_1),\dots,(X_m,\mu_m)$ satisfy the cost-Santal\'o inequality associated with $(\lambda,c,F)$, with domain $(\mathcal{C}_{x_1},\dots, \mathcal{C}_{X_m})$, where $\mathcal{C}_{X_i}\subset 2^{X_i}$, if they satisfy the equivalent inequalities in Lemma~\ref{l:transfer} with these parameters.
    
    If a single space $(X,\mu)$ is repeated $m$-times, we mention this space just once; the amount of occurrences is dictated by the length of $\lambda$. If $m=2$, we may list $\lambda$ just once, since $\lambda_1=\lambda$ and $\lambda_2=1-\lambda$ in this case.
\end{definition}

We denote for $K\subset \R^n$ a compact set with positive volume, $b(K)=\int_Kx\frac{dx}{\vol(K)}$ its center of mass. We define
\[
\mathcal{C}^n=\left\{K\subset\R^n|\,b(K)=0, \, K\text{ is compact with positive volume}\right\}
\]
and define the following class of non-negative, measurable functions
$$\mathcal{F}_n=\mathcal{F}\left(\mathcal{C}^n\right):=\left\{f:\R^n\to\R_+\bigg|\,0<\int_{\R^n}f<\infty\,\text{and, for a.e. }r\in (0,\|f\|_\infty), \{f\geq r\}\in\mathcal{C}^n  \, \right\},$$
under the convention that $b\left(\emptyset\right)=o$. We denote by $\mathcal{B}\left(X\right)$ the Borel $\sigma$-algebra on a Polish measure space $(X,d,\mu)$.

We will derive the equivalence between the functional \eqref{eq:BS_fun} and geometric \eqref{eq:BS} Blaschke-Santal\'o inequalities for all $\Omega$ when $f\in \mathcal{F}_n$ using the transference principle (Lemma~\ref{l:transfer}). In the case of sets, this corresponds to when $K$ has center of mass at the origin. We present here a sketch, which we will complete in Section~\ref{sec:prelim}. Consider $K \in \mathcal{C}^n$ and let $L \in \mathcal{B}(\R^n)$ and define $M(K,L) = \sup_{x\in K, y\in L} \langle x, y \rangle$. A direct application of polarity and the Blaschke–Santal\'o inequality \eqref{eq:BS} yields the sharp relation
\[
\vol(K)^{\frac{1}{2}} \vol(L)^{\frac{1}{2}} \leq \vol(\B) M(K,L)^{\frac{n}{2}}.
\]
When sent through the lens of Lemma~\ref{l:transfer} for $m=2$ under the cost $c(x,y)=-\langle x,y\rangle$, this inequality satisfies the hypothesis required with the function $F(r)=\vol(\B) r^{\frac{n}{2}}$. 

Consequently, for any function $f\in\mathcal{F}_n$ and any measurable function $g:\R^n\to\R_+$ satisfying $f(x)g(y)\leq \Omega^2(\langle x,y\rangle)$ whenever $\langle x,y\rangle>0$, Lemma~\ref{l:transfer} automatically produces
\[
\left(\int_{\R^n} f\right)^{\frac{1}{2}} \left(\int_{\R^n} g\right)^{\frac{1}{2}} \leq \frac{n}{2} \vol(\B) \int_0^\infty \Omega(r) r^{\frac{n}{2}-1}\,dr.
\] 
A change of variables and an integration in polar coordinates identifies the right-hand side as $\int_{\R^n}\Omega(|x|^2)\,dx$, recovering the functional Santal\'o inequality \eqref{eq:BS_fun}. While this introductory showcase highlights the efficiency of our transference mechanism, one minor limitation is that this approach does not preserve the characterization of the equality cases.

Interest in Santal\'o-type inequalities induced by a cost function of more than two variables was recently sparked by
Kolesnikov and Werner \cite{KW22}, who considered the cost function
$c:X_1\times\cdots\times X_m\to \mathbb R$ given by
\begin{equation}
    c(x_1,\dots,x_m)=\sum_{1\le i<j\le m}|x_i-x_j|^2
\end{equation}
and then conjectured a generalization of
Proposition~\ref{prop:Ball_Fradelizi_Meyer}; a recent work~\cite{NT24_2} by Nakamura and Tsuji, building on their
previous techniques~\cite{NT24,NT22}, resolved this conjecture in arguably the most important choice of $\Omega$, the exponential function.

\subsection{New Examples}
Having detailed how the transference principle bridges geometric and functional Blaschke-Santal\'o inequalities, we provide further examples. As far as we are aware, the resulting functional inequalities are new. 

\subsubsection{New Euclidean functional Santal\'o inequalities} In his PhD thesis, K. Ball \cite{Ball_thesis} proposed the following variant of the Blaschke-Santal\'o inequality: for every $K\in\conbode[n]$,
\begin{equation}
\label{eq:Ball_san}
\int_{K}\int_{K^\circ}\langle x,y\rangle^2 dx dy \leq \int_{\B}\int_{\B}\langle x,y\rangle^2 dx dy,
\end{equation}
with equality if and only if $K$ is an ellipsoid. Clearly, using again the involution property $K\subseteq K^{\circ\circ}=\operatorname{conv}(K\cup\{o\})$, the convex case immediately implies that \eqref{eq:Ball_san} holds for any origin-symmetric compact set with positive volume.

This Ball-Santal\'o inequality remained unresolved until 2026, when it was solved by B\"or\"oczky, Patsalos and Saroglou \cite{BPS26}. In the intermediate decades, Huang and Li \cite{HL17} showed that the Ball-Santal\'o inequality \eqref{eq:Ball_san} is equivalent to the following functional statement: given $\Omega: \mathbb{R} \rightarrow \mathbb{R}_{+}$ and $f, g: \mathbb{R}^n \rightarrow \mathbb{R}_{+}$  even, measurable functions satisfying $f(x) g(y) \leq \Omega^2(\langle x, y\rangle)$, for all $x, y \in \mathbb{R}^n$ such that $\langle x, y\rangle>0$, it holds that
\begin{equation}
\label{eq:fun_ball}
\int_{\mathbb{R}^n} \int_{\mathbb{R}^n}\langle x, y\rangle^2 f(x) g(y) d x d y \leq \frac{1}{n}\left(\int_{\mathbb{R}^n}|u|^2 \Omega\left(|u|^2\right) d u\right)^2 .
\end{equation}

We would like to mention that this type of functional inequality, where there is a kernel density $K(x,y)$ on $X \times X$, falls outside of our framework. Similar inequalities, with kernels of the form $|\langle x,y\rangle|^p$ and $\langle x,y\rangle_+^p$, were discovered in \cite{LYZ00,HS09,NVH19}. Nevertheless, the method used in \cite{BPS26} to prove \eqref{eq:Ball_san} was establishing the following stronger formula, which we call the isotropic Ball-Santal\'o inequality:
\begin{equation}
\label{eq:better_ball_san}
\int_K\langle x, u\rangle^2 d x \int_{K^{\circ}}\langle x, u\rangle^2 d x \leq\left(\int_{B_2^n}\langle x, u\rangle^2 d x\right)^2, \, \, u\in \R^n.
\end{equation}
The formula \eqref{eq:better_ball_san} holds only for \textit{isotropic} convex bodies $K$, which are defined as those origin-symmetric convex bodies satisfying
\begin{equation}
\label{eq:iso_def}
\int_K\langle x, u\rangle^2 d x=\frac{1}{n} \int_K|x|^2 d x \quad \forall \,\, u \in \mathbb{S}^{n-1}.
\end{equation}
In particular, this means if $i\neq j$, then $\int_{K}x_ix_jdx=0$. To get \eqref{eq:Ball_san} from \eqref{eq:better_ball_san}, merely open the quadratic form in the integral and use this property. 

Notice that \eqref{eq:better_ball_san} is in our framework. We let
\[
\conbodI=\{K\in \conbode: K \,\,  \text{ is isotropic}\}
\]
and establish the following.

\begin{theorem}[The Functional, isotropic Ball-Santal\'o inequality]
\label{thm:func_ball_san}
    Let $u \in \R^n$. Let $f \in \mathcal{F}(\conbodI)$. Then, for every measurable function $g:\R^n\to\R_+$ and measurable, non-increasing function $\Omega : (0, \infty) \to (0, \infty)$ such that 
    \[
    f(x)g(y) \leq \Omega^2(\langle x, y \rangle), \qquad \forall \,\,x, y \in \R^n \text{ with } \langle x,y \rangle>0,
    \] 
    it holds
    \begin{equation}
    \label{eq:func_better_ball_san_equiv}
        \left(\int_{\R^n} f(x) \langle x,u \rangle^2 dx \right) \left(\int_{\R^n} g(y) \langle y,u \rangle^2 dy \right) \leq \left(\int_{\R^n} \Omega(|z|^2) \langle z,u \rangle^2 dz \right)^2.
    \end{equation}
\end{theorem}

Another example, whose history is outlined in Section~\ref{sec:sine}, is the so-called \textit{sine duality} introduced in \cite{HLXY22,AASW23_2,AASW23}. Define 
$$[x,y]=\begin{cases}
|x||y|\sqrt{1-\left\langle \frac{x}{|x|},\frac{y}{|y|} \right\rangle^2}& \text{ if }x,y \in \R^n\setminus\{o\},
\\
0& \text{otherwise}.
\end{cases}
$$
We will utilize the cost function $c(x,y)=-[x,y]$. What is particularly interesting about this cost function is that $-c$ is not convex in general. Then, we prove the following.
\begin{corollary}[The Functional Sine-Santal\'o inequality]
\label{cor:sine_fun}
    Let $n \geq 2$. For any non-negative measurable functions $f, g: \R^n \to \R_+$ and any non-increasing measurable function $\Omega: (0, \infty) \to (0, \infty)$ satisfying the pointwise bound:
    $$f(x)g(y) \leq \Omega^2([x,y]) \quad \text{for all } x, y \in \R^n \text{ with } [x,y] > 0,$$
one has the inequality:
    $$\left(\int_{\R^n} f(x) dx \right) \left(\int_{\R^n} g(y) dy \right) \leq \left(\int_{\R^n} \Omega(|z|^2) dz \right)^2.$$
\end{corollary}

Other examples that follow from Blaschke-Santal\'o inequalities on $\R^n$ can be found in Section~\ref{sec:euclidean}. A particularly noteworthy example is that for unconditional, log-concave measures; see Corollary~\ref{c:unconditional_functional_BS}.

\subsubsection{Isoperimetry and Enlargement} Given a Polish space $(X,d)$ and a non-empty $A\in \mathcal{B}(X)$, we write $A_\delta=\{x\in X:d(x,A)\leq \delta\}$ for the closed $\delta$-neighborhood of $A$, with $\delta>0$.  

Notice that the $n$-dimensional sphere $\mathbb{S}^n$ can be realized as a Polish probability space equipped with its geodesic distance $d(\cdot,\cdot) = \arccos (\langle \cdot,\cdot \rangle)$, where $\langle \cdot, \cdot \rangle$ is the usual scalar product, and the Haar probability measure $\sigma_n$. Our instinct was to then use the spherical Blaschke-Santal\'o inequality shown by Gao, Hug and Schneider \cite[Corollary]{GHS01} to obtain a functional Santal\'o inequality. As it turns out, that inequality is too weak, and we show in Proposition~\ref{p:cap_BS} that it is a consequence of the spherical isoperimetric inequality with $\delta=\frac{\pi}{2}$. 

Let $A\subset \mathbb S^n$ be a non-empty, measurable set. Let $C_A\subset \mathbb S^n$ be any spherical cap with the same Haar measure as $A$. Then, recall that the spherical isoperimetric inequality states
\begin{equation}
\label{eq:sph_iso}
\sigma_n(A_\delta)\geq \sigma_n((C_A)_\delta), \qquad \delta>0,
\end{equation}
with equality if and only if $A$ is a spherical cap, up to removal of null sets. We utilize \eqref{eq:sph_iso} itself to deduce cost Santal\'o inequalities. The geometric case follows from Lemma~\ref{l:general_enlargement_santalo} below, and we outline the (equivalent) functional case here. If $C_n(r)$ is any spherical cap of radius $r$, then we denote its Haar measure by
\[
\Sigma_n(r):=\sigma_n(C_n(r))= \frac{\int_0^r \sin^{n-1}(\theta) \, d\theta}{\int_0^\pi \sin^{n-1}(\theta) \, d\theta}.
\]

\begin{corollary}
\label{cor:spherical_functional}
Define the function $G: [-1, 1] \to [0, \frac{1}{2}]$ by $G(r) = \Sigma_n\left(\frac{\pi - \arccos r}{2}\right)$.

For any measurable functions $f, g : \mathbb{S}^n \to \R_+$ and any measurable, non-increasing function $\Omega: (-1,1] \to (0, \infty)$ satisfying $$f(x)g(y) \leq \Omega^2(\langle x, y \rangle) \text{ for all }x, y \in \mathbb{S}^n,\, x\neq - y,$$ one has
\begin{equation}
\label{eq:sph_func_BS_general}
    \left(\int_{\mathbb{S}^n} f(x) \,d\sigma_n(x)\right)^{\frac{1}{2}} \left(\int_{\mathbb{S}^n} g(y) \,d\sigma_n(y)\right)^{\frac{1}{2}} \leq 2\int_{-1}^1 \Omega(r) \,dG(r).
\end{equation}
In particular, if $f$ and $g$ satisfy the constraint $f(x)g(y) = 0$ for all $\langle x, y \rangle > 0$, we obtain
\begin{equation}
\label{eq:sph_func_polar_bound}
    \left(\int_{\mathbb{S}^n} f(x) \,d\sigma_n(x)\right)^{\frac{1}{2}} \left(\int_{\mathbb{S}^n} g(y) \,d\sigma_n(y)\right)^{\frac{1}{2}} \leq \int_{-1}^1 \Omega(r) \,dG(r).
\end{equation}
\end{corollary}

Motivated by this example, we consider Polish probability spaces satisfying isoperimetric inequalities (in the sense of enlargement, as in the spherical case) with enlargement functions that are log-concave and symmetric about their mean in Section~\ref{sec:iso}; see Corollary~\ref{cor:differential_isoperimetry_santalo} and Proposition~\ref{p:symmetric_isop_enlargement} below. We derive Corollary~\ref{cor:spherical_functional} from this directly. We also apply this framework to obtain a Santal\'o-type inequality for every even log-concave probability measure on the real line in Corollary~\ref{cor:functional_1d_symmetric}. We mention here another illuminating example. 

Recall that the Gaussian measure $\gamma_n$ on $\R^n$ is given by $\gamma=\gamma_1$ and $$d\gamma_n(x)=\frac{1}{(2\pi)^{\frac{n}{2}}}e^{-\frac{|x|^2}{2}}\,dx, \qquad n\in \N.$$ Let $\Psi_\gamma(t)=\gamma((-\infty,t))$ be the cumulative distribution function of $\gamma$. The Gaussian isoperimetric inequality states, in its enlargement form, if $A\subset \R^n$ is a measurable set with $\gamma_n(A)=\Psi_\gamma(a)\in (0,1)$, then
\begin{equation}
\label{eq:gamma_iso}
\gamma_n(A_\delta)\geq \Psi_\gamma(a+\delta), \qquad \delta>0,
\end{equation}
where $A_\delta$ is defined using the Euclidean distance $d(x,y)=|x-y|$. 

Recall also that a probability measure $\mu$ on $\R^n$ is said to be log-concave with respect to $\gamma_n$ if there exists a log-concave function $\varphi$ such that $d\mu(x)=\varphi(x)\,d\gamma_n(x)$; this class includes $\gamma_n$ itself by taking $\varphi\equiv 1$. Then, Bobkov \cite[Theorem 1]{SB02} showed that \eqref{eq:gamma_iso} holds with $\gamma_n$ replaced by all such $\mu$. The following corollary immediately follows from Corollary~\ref{cor:differential_isoperimetry_santalo} and Proposition~\ref{p:symmetric_isop_enlargement}.

\begin{corollary}
\label{cor:functional_gaussian_iso}
Let $\mu$ be a probability measure on $\R^n$ that is log-concave with respect to $\gamma_n$. Then, for any pair of measurable functions $f, g:\R^n\to \R_+$ and any measurable, non-decreasing function $W: \R_+ \to (0, \infty)$ satisfying the pointwise relation
\begin{equation}
\label{eq:pointwise_cond_gaussian}
f(x)^{\frac{1}{2}} g(y)^{\frac{1}{2}} \leq W\left(\frac{|x-y|}{2}\right) \quad \text{for all } x, y \in \R^n,
\end{equation}
one has the functional inequality
\begin{equation}
\label{eq:functional_ineq_gaussian}
\left( \int_{\R^n} f(x) \, d\mu(x) \right)^{\frac{1}{2}} \left( \int_{\R^n} g(y) \, d\mu(y) \right)^{\frac{1}{2}} \leq  \int_{\R} W(|r|) \, d\gamma(r).
\end{equation}
\end{corollary}

Following groundbreaking work by Talagrand \cite{TM91}, Maurey \cite{BM91} introduced the so-called Property $(\tau)$ to study Gaussian concentration. Later, Artstein, Klartag and Milman \cite{AAKM04} introduced Property (even $\tau$) and demonstrated that Property (even $\tau$) is intimately connected to, and in fact equivalent to, Ball's functional Santal\'o inequality \eqref{eq:ball_bS}. The methodology leading to Corollary~\ref{cor:functional_gaussian_iso} is reminiscent of that in Maurey's work.

In Section~\ref{sec:RCD}, we obtain results for $\mathrm{RCD}(K, \infty)$-spaces, culminating in Corollary~\ref{cor:RCD}. In that section, we introduce this framework properly. In that terminology, the weighted Euclidean space $(\R^n,|\cdot|,\mu)$, where $\mu$ is log-concave with respect to the Gaussian, satisfies the $\operatorname{RCD}(1,\infty)$ condition. Therefore, Corollary~\ref{cor:functional_gaussian_iso} also follows from Corollary~\ref{cor:RCD} by choosing the appropriate parameters. We obtain in Section~\ref{sec:uncon} a different functional inequality for $\gamma_n$ using the geometric Blaschke-Santal\'o inequalities for unconditional, log-concave measures in \cite[Theorem 7]{FMZ23}.

Our investigations into isoperimetry hold for every Polish probability space satisfying an isoperimetric inequality (in the usual Minkowski content sense; see Section~\ref{sec:iso_gen}). However, the integrals on the right-hand side of the above theorems do not simplify cleanly if the associated enlargement function is not log-concave and symmetric about its mean. We obtain in Section~\ref{sec:hcube} a functional Santal\'o inequality on the discrete Hamming cube.

\subsection{Cost functions and Duality}
We now discuss how cost functions impose dualities on sets. We consider two Polish measure spaces $X$ and $Y$ and a continuous cost function $c:X\times Y\to (-\infty,\infty)$. Given a cost function, it induces $c$-transforms, which are dualities on functions that switch the domains. These are usually expressed as infimum convolutions (see, e.g., \cite[Appendix A]{AASW23} and \cite[Definition 1.2.4]{AGA2}). In our work, we consider $c$-Legendre transforms, which can be derived from these $c$-transforms:
\begin{align}
    &\psi:X\to (-\infty,\infty] \quad \text{has $c$-Legendre transform} \quad \mathcal{L}_c\psi(y) = \sup_{x\in X}(-c(x,y)-\psi(x)),
    \\
    &\text{and} \nonumber
    \\
    &\varphi:Y\to (-\infty,\infty] \quad \text{has $c$-Legendre transform} \quad \mathcal{L}^c\varphi(x)= \sup_{y\in Y}(-c(x,y)-\varphi(y)).
\end{align}
In the field of convex analysis, for example in \cite{AASW23_2,AASW22}, one often considers symmetric costs, i.e., when $X=Y$ and $c$ satisfies $c(x,y)=c(y,x)$. In our situation, $X$ will not equal $Y$ in general, and our costs will be non-symmetric. In either case, we have by the definition of the supremum that
\begin{equation} \mathcal{L}_c \psi(y) + \psi(x)\geq -c(x,y) \quad \text{and} \quad \mathcal{L}^c \varphi(x) + \varphi(y)\geq -c(x,y),\quad \forall x\in X,\,\,y\in Y.
\label{eq:sup_ineq}
\end{equation}

To move closer to our considerations, we define, for measurable functions $f:X\to \R_+$ and $g:Y\to \R_+$, their \textit{c-polarities}: 
\begin{align*}
f^{\circ, c}(y) &= \exp[-\mathcal{L}_c(-\log f)(y)]=\inf_{x\in X}\left[\frac{e^{c(x,y)}}{f(x)}\right], \, \text{where }\, f^{\circ, c}:Y\to \R_+,
\\
&\text{and}
\\
g^{\star, c}(x) &= \exp[-\mathcal{L}^c(-\log g)(x)]=\inf_{y\in Y}\left[\frac{e^{c(x,y)}}{g(y)}\right], \, \text{where }\, g^{\star, c}:X\to \R_+.
\end{align*}
Each operation is order-reversing: focusing on just the functions $\varphi$ and $g$ on $Y$, as the functions $\psi$ and $f$ on $X$ are analogous, we have
\begin{equation}
\label{eq:duality}
\varphi \geq \mathcal{L}_c(\mathcal{L}^c\varphi),\; \mathcal{L}^c\varphi = \mathcal{L}^c(\mathcal{L}_c(\mathcal{L}^c\varphi)) \text{ and } g \leq (g^{\star, c})^{\circ, c},\; g^{\star, c} = ((g^{\star, c})^{\circ, c})^{\star, c}.\end{equation}
We apply polarity to the case when $X$ and $Y$ are Polish geodesic spaces and $-c$ is geodesically convex in each variable. Then, regardless of the choice of $\psi$ and $\varphi$, $\mathcal{L}_c \psi$ and $\mathcal{L}^c \varphi$ are convex and $\psi^{\circ, c}$ and $\varphi^{\star, c}$ are log-concave (in the geodesic sense).

\subsection{Duality on Matrix Spaces}
The final setting for functional Blaschke-Santal\'o inequalities for $m$ different functions is the matrix space, which represents an example of a non-symmetric space and a non-symmetric cost. We consider two parameters concerning dimension, $m,n\in\N$. We set the space $Y=\R^n$ and $X=(\R^n)^m$. Of course, $(\R^n)^m=\R^{nm}$, but the product structure will be of vital importance. 

In fact, to explicitly use the product structure, we identify $\R^{nm}$ with the space of $n\times m$ matrices with real entries, which we write as $M_{n,m}(\R)$. We will, as usual, write elements of $M_{n,m}(\R)$ in capital letters. Then, for $d\in\{n,m\}$, $\R^d$ is the space of $d\times 1$ (column) vectors. Integration in $M_{n,m}(\R)$ is merely integration on $\R^{nm}$ with respect to the Lebesgue measure. 

Let $h_K(v)=\sup_{u \in K}\langle v,u \rangle$ denote the support function of a convex body. For our results, we will be concerned with the cost function
\begin{equation}c(A,v)=-h_Q(A^tv),
\label{eq:our_costs}
\end{equation}
where $Q\in\conbodo[m]$, $v\in\R^n$, $A\in M_{n,m}(\R)$, and $A^t\in M_{m,n}(\R)$ is the transpose of $A$. If one chooses $Q=[-1,1]\subset \R$, so that $A=a$ for some $a\in\R^n$, then $h_Q(A^tv)=|\langle a,v \rangle|$. We recall that if $x\in\R^n$ and $y\in\R^m$, then their tensor product $x\otimes y \in M_{n,m}(\R)$ is right matrix multiplication of $x$ by the transpose of $y$.

We will use some special notation that is specific to this cost function. We remark only once that $-c$ is a convex function, and thus all pertinent Legendre transforms are convex and all polarities are log-concave. We first have the definitions corresponding to the $c$-polarities. Given a function $\psi:M_{n,m}(\R)\to\R$ and a $Q\in\conbodo[m]$, we define the lowered $Q$-Legendre transform of $\psi$ as the function $\Ld Q\psi:\R^{n}\to\R$ given by
\begin{equation}
\label{eq:Q_leg_low}
    \Ld Q\psi(v)=\sup_{A\in M_{n,m}(\R)}\left(h_Q(A^tv)-\psi(A)\right).
\end{equation}
Given a function $f:M_{n,m}(\R)\to\R_+$, we define its \textit{lower $Q$-polar function} $f^{\circ,Q}:\R^{n}\to\R_+$ by 
\begin{equation}
\label{eq:lower_Q_polar}
f^{\circ,Q}(v):=\exp\left[-\Ld Q(-\log f)(v)\right]=\inf_{A\in M_{n,m}(\R)}\left[\frac{e^{-h_Q(A^tv)}}{f(A)}\right].\end{equation}
It is easy to see that $f^{\circ,Q}$ is even when either $f$ or $h_Q$ is even. Similarly, the following definitions correspond to the star $c$-polarities for the cost function given by \eqref{eq:our_costs}. For a measurable function $\varphi:\R^{n}\to\R$, we define the raised $Q$-Legendre transform of $\varphi$ as the function $\Lu Q\varphi:M_{n,m}(\R)\to\R$ given by
\begin{equation}
\label{eq:Q_leg_up}
    \Lu Q\varphi(A)=\sup_{v\in\R^{n}}\left(h_Q(A^tv)-\varphi(v)\right).
\end{equation}
Many of the properties enjoyed by $\Ld Q$ are enjoyed by $\Lu Q$, but with a subtle difference in the variables. Given a measurable function $g:\R^{n}\to\R_+$, we define its \textit{upper $Q$-polar function} $g^{\star,Q}:M_{n,m}(\R)\to\R_+$ by 
\begin{equation}
\label{eq:upper_Q_polar}
g^{\star,Q}(A):=\exp\left[-\Lu Q(-\log g)(A)\right]=\inf_{v\in\R^{n}}\left[\frac{e^{-h_Q(A^tv)}}{g(v)}\right].\end{equation}
The polarity ${}^{\star,Q}$ on functions on $\R^{n}$ has properties analogous to the polarity ${}^{\circ,Q}$ on functions on $M_{n,m}(\R)$. Furthermore, if $g$ or $h_Q$ is even, then $g^{\star,Q}$ is even.

To motivate this choice of cost function, we first introduce two types of polarities on convex sets. Fixing again $Q\in\conbodo[m]$, for a compact set $L\subset M_{n,m}(\R)$, its polar in $\R^n$ is the set \begin{equation} L_{\circ,Q}:=\bigcap_{A\in L}(AQ)^\circ\subset \R^n.
\label{eq:lowered_polarity_set}
\end{equation} The gauge of this set is given by 
\[
\|x\|_{L_{\circ,Q}}=\max_{A\in L}h_{AQ}(x)=\max_{A\in L}h_Q(A^tx).
\]
Similarly, for a compact set $K\subset \R^n$, its polar in $M_{n,m}(\R)$ is the set \begin{equation}
\label{eq:raised_polarity_set}
K^{\star,Q}:=\{A\in M_{n,m}(\R):AQ\subseteq K^\circ\}\subset M_{n,m}(\R),\end{equation} 
and its gauge is given by \begin{equation}
\label{eq:a_duality}
\|\Theta\|_{K^{\star,Q}}=\max_{\xi \in \s}\left(\frac{h_Q(\Theta^t\xi)}{h_{K^\circ}(\xi)}\right)=\max_{\xi \in K}h_Q(\Theta^t\xi).\end{equation} We isolate the case when $K=\B$:
\begin{equation}
\label{eq:polar_ball}
\|\Theta\|_{(\B)^{\star, Q}}=\max_{\xi\in\s}h_Q(\Theta^t \xi)=\max_{\xi\in\s}h_{\xi\otimes Q}(\Theta)=\max_{\xi\in\s}h_{\Theta Q}(\xi).\end{equation}

We can now list the following theorems.
\begin{theorem}[The $m$th-order Santal\'o inequality on $M_{n,m}(\R)$]
\label{t:hiBS}
  Fix $Q\in\conbodo[m].$ Consider a compact set $L\subset M_{n,m}(\R)$ with positive volume and non-empty interior. If either $L_{\circ,Q}$ or $(L_{\circ,Q})^\circ$ has center of mass at the origin, then
  $$\Vol(L)\vol(L_{\circ,Q})^m \leq \vol(\B)^m \Vol((B_2^n)^{\star,Q}),
$$
with equality if and only if $L=E^{\star,Q}$ for some centered ellipsoid $E\subset\R^n$ up to removal of null sets.
\end{theorem}
The above theorem was originally shown in \cite[Theorem 4.8]{HLPRY25_2} without equality conditions via a limiting process.  We show in Lemma~\ref{l:equiv} that Theorem~\ref{t:hiBS} is entirely equivalent to:

\begin{theorem}[The $m$th-order Santal\'o inequality on $\R^n$]
\label{t:hiBS_2}
  Fix $Q\in\conbodo[m].$ Let $K\subset \R^n$ be a compact set of positive volume such that $K$ or $K^\circ$ has center of mass at the origin. Then,
  \begin{equation}
  \Vol(K^{\star,Q})\vol(K)^m \leq \vol(\B)^m \Vol((B_2^n)^{\star,Q}).
\label{eq:BS_onRn}
\end{equation}
There is equality if and only if $K=E$ for some centered ellipsoid $E\subset\R^n$ up to removal of null sets.
\end{theorem}
We supply a proof of this theorem by using Proposition~\ref{p:PPIinfty} below, which is an inequality that relates the volumes of $K^\circ$ and $K^{\star,Q}$. This approach yields the sought-after equality conditions. We now present our main theorem in the matrix setting.
\begin{theorem}
\label{t:main}
     Fix $m\in\N$, $Q\in\conbodo[m]$. Let $f\in \mathcal{F}_n$. Then, for every measurable $g:M_{n,m}(\R)\to\R_+$ and measurable, non-increasing function $\Omega:\R_{+}\to\R_{+}$ such that
    $$f(v) g(A)\leq \Omega^{m+1}\left(h_Q(A^tv)\right) \quad \text{for all } (A,v)\in M_{n,m}(\R)\times \R^n,$$ 
    it holds
    \[
    \left(\int_{M_{n,m}(\R)}g(A)dA\right)\Bigg(\int_{\R^n} f(v)^\frac{1}{m}dv\Bigg)^m
    \!\!\!\!\leq\! \left(\int_{M_{n,m}(\R)}\!\!\!\!\!\!\!\!\!\!\!\!\Omega(\|A\|_{(B_2^n)^{\star,Q}}^{m+1})dA\right)\!\left(\int_{\R^n}\Omega\left(|v|^\frac{m+1}{m}\right)dv\right)^{m}.
\]
    There is equality when $\Omega$ is log-concave and $f = \frac{1}{C} \cdot \Omega\left( |\cdot|^{\frac{m+1}{m}}\right)^m$ and $g = C\cdot \Omega\left( \|\,\cdot\,\|_{(\B)^{\star, Q}}^{m+1}\right)$ for some constant $C>0$.
\end{theorem}

The power $\frac{1}{m}$ on the terms over $\R^n$ is necessary to ensure both sides are invariant under dilation of the functions. Notice that we may take $\Omega(t)=e^{-\frac{t}{m+1}}$ and $f$ as the dual of $g$, or vice versa, since, if $f$ is a function on $\R^n$ and $g$ is a function on $M_{n,m}(\R)$,
\[
\max\{f^{\star,Q}(A)f(v),g(A)g^{\circ,Q}(v)\} \leq e^{-h_Q(A^tv)}, \quad \forall \;(A,v)\in M_{n,m}(\R)\times \R^n.
\]
With these choices of functions, we recover the origin-symmetric case of Theorems~\ref{t:hiBS} and \ref{t:hiBS_2} via \eqref{eq:dirichlet} below and the following identities, which will be shown in Proposition~\ref{p:polar_gauge}:
\begin{align*}
\label{eq:functional_duals}
\exp\left[-\frac{\|\cdot\|_L^{m+1}}{m+1}\right]^{\circ,Q}\!\!\!(v)&=\exp\left[-\frac{m}{m+1}\|v\|_{L_{\circ,Q}}^\frac{m+1}{m}\right] 
\\ &\text{and} 
\\
\exp\left[-\frac{m}{m+1}\|\cdot\|_K^\frac{m+1}{m}\right]^{\star,Q}\!\!\!\!(A)&=\exp\left[-\frac{\|A\|_{K^{\star,Q}}^{m+1}}{m+1}\right].
\end{align*}

\subsection{Connections to Optimal Transport}
Our final primary objective is to demonstrate that the framework of cost-Santal\'o inequalities yields immediate consequences in the theory of optimal transport. In his seminal work \cite{TM96}, Talagrand established a fundamental transport-entropy inequality for the standard Gaussian measure, bounding the quadratic transportation cost by the relative entropy. 

We show that the transference principle in Lemma~\ref{l:transfer} produces a wide class of new Kantorovich-Talagrand-type inequalities on Polish probability spaces. We first need some definitions. Let $(X,d,\mu)$ be a Polish probability space. We denote by $\mathcal{P}(X)$ the set of all Borel probability measures on $X$. Then, the relative entropy of $\nu \in \mathcal{P}(X)$ with respect to $\mu$ is defined by 
\[
H(\nu|\mu)=\begin{cases}
    \int_{X}f(x)\log f(x) \, d\mu(x), & \text{if } d\nu=f d\mu,
    \\
    +\infty, & \text{otherwise}.
\end{cases}
\]
By Jensen's inequality, $H(\nu|\mu)$ is non-negative. In the introduction, we aim to provide only a representative illustration of our results; we state a consequence for matrix spaces, which arises as a special case of a more general formulation.

Let $\Omega:\R_+\to(0,\infty) $ be a continuous, non-increasing, log-concave function satisfying the integrability conditions
\[
0 < \int_0^\infty \Omega\left(t^{\frac{m+1}{m}}\right)t^{n-1}dt < +\infty \quad \text{and} \quad 0 < \int_0^\infty \Omega\left(t^{m+1}\right)t^{nm-1}dt < +\infty.
\]
Given a body $Q \in \conbodo[m]$, we introduce the probability measures $\eta_{\Omega,m}$ on $\R^n$ and $\mu_{\Omega,m}$ on $M_{n,m}(\R)$ defined by:
\[
d\eta_{\Omega,m}(v) = Z_{1}^{-1} \Omega\left(|v|^{\frac{m+1}{m}}\right) dv \quad \text{and} \quad d\mu_{\Omega,m}(A) = Z_{2}^{-1} \Omega\left(\|A\|^{m+1}_{(\B)^{\star,Q}}\right) dA,
\]
where $Z_{1}$ and $Z_{2}$ denote the corresponding normalizing constants.

\begin{corollary}
\label{c:mass_trans}
    Fix $m,n\in\N$ and let $Q\in\conbodo[m]$ contain the origin. Let $\Omega$ be as above. Consider the cost-transport function on $\R^n \times M_{n,m}(\R)$ given by:
    \[
    c_{\Omega,m}(v,A) = \log\left(\frac{\Omega\left(h_Q(A^tv)\right)^{m+1}}{\Omega\left(|v|^{\frac{m+1}{m}}\right)^m \Omega\left(\|A\|^{m+1}_{(\B)^{\star,Q}}\right)}\right).
    \]
    Let $\nu_1\in\mathcal{P}(\R^n)$ and $\nu_2\in\mathcal{P}(M_{n,m}(\R))$ possess densities with respect to $\eta_{\Omega,m}$ and $\mu_{\Omega,m}$, respectively. Then, the following assertions hold:
    \begin{enumerate}
        \item The restricted Kantorovich optimal transport cost satisfies
        \[
        \mathcal{T}_{c_{\Omega,m}}^{\mathcal{F}}(\nu_1,\nu_2) \leq m H(\nu_1|\eta_{\Omega,m}) + H(\nu_2|\mu_{\Omega,m}),
        \]
        where $\mathcal{T}_{c_{\Omega,m}}^{\mathcal{F}}$ is defined by taking the supremum over pairs of potentials $(\varphi, \psi)$ with the restriction that $e^{\frac{\varphi(\cdot)}{m}} \Omega(|\cdot|^{\frac{m+1}{m}}) \in \mathcal{F}_n$.
        
        \item Moreover, if $Q$ is origin-symmetric and the measures $\nu_1, \nu_2$ are even, then the true optimal transport cost satisfies
        \[
        \mathcal{T}_{c_{\Omega,m}}(\nu_1,\nu_2) \leq m H(\nu_1|\eta_{\Omega,m}) + H(\nu_2|\mu_{\Omega,m}).
        \]
    \end{enumerate}
\end{corollary}

This paper is organized as follows. In Section~\ref{sec:prelim}, we collect the foundational concepts needed for this work. This includes, in Sections~\ref{sec:set_prop} and \ref{sec:properties}, abstract duality with respect to a cost function. In Section~\ref{sec:main}, we prove the transference principle (Lemma~\ref{l:transfer}).

In Section~\ref{sec:euclidean}, we mention new examples in Euclidean space of cost-Santal\'o inequalities. In Section~\ref{sec:matrix}, we prove cost-Santal\'o inequalities for the space of rectangular matrices. In Section~\ref{sec:iso}, we provide examples of new functional Santal\'o inequalities arising from isoperimetry.

Finally, in Section~\ref{sec:opt_trans}, we detail the full development of optimal transport inequalities, which arise as a natural consequence of our framework, along with the necessary background on optimal transport plans and cost functions.

\section{Preliminaries}
\label{sec:prelim}
Throughout this work, we will sometimes denote $\omega_n = \vol(\B)$. We recall that a set $L \subset \R^n$ is a star body if it is compact, star-shaped with respect to the origin (meaning $o \in \operatorname{int}(L)$, and $x \in L$ implies $[o,x] \subset L$), and its Minkowski functional $\|\cdot\|_L$ is continuous on $\R^n$.
In this work, we will utilize a generalization of \eqref{eq:norm_gauss_int}. For $p,c>0$ and any star-shaped set $K$, Fubini's theorem yields the relation:
\begin{equation}
\label{eq:dirichlet}
\vol(K) = \frac{1}{c^{\frac{n}{p}}\Gamma\left(1+\frac{n}{p}\right)}\int_{\R^n} e^{-\frac{\|v\|_K^p}{c}} \,dv.
\end{equation}
Moreover, this identity can be extended to an arbitrary $\Omega$. To do so, we first recall that for any star body $L$ in $\R^n$, the volume can be computed via radial integration:
\begin{equation}
\label{eq:polar_radial}
\vol(L) = \frac{1}{n}\int_{\mathbb{S}^{n-1}} \|u\|_L^{-n} \,du.
\end{equation}

\begin{proposition}
    Let $d\in \N$ and $\alpha>0$. Let $\Omega:\R_+\to \R_+$ be a measurable function such that 
    \[
    0 < \int_{0}^{\infty}\Omega\left(t^\alpha\right)t^{d-1}\,dt < +\infty.
    \]
    Then, for every star body $L\subset \R^d$, it holds that
    \begin{equation}
    \label{eq:set_m_2}
    \vol[d](L) = \frac{1}{d} \frac{\int_{\R^d}\Omega\left(\|x\|_{L}^\alpha\right)\,dx}{\int_0^\infty \Omega\left(t^\alpha\right)t^{d-1}\,dt}.
    \end{equation}
 \end{proposition}

\begin{proof} 
Integrating in polar coordinates, we observe that
\begin{align*}
    \int_{\R^d}\Omega\left(\|x\|_{L}^\alpha\right)\,dx 
    &= \int_{\mathbb{S}^{d-1}}\int_{0}^{+\infty}\Omega\left(t^\alpha\|\theta\|_{L}^\alpha\right)t^{d-1}\,dt\,d\theta \\
    &= \int_{\mathbb{S}^{d-1}}\|\theta\|_{L}^{-d} \int_{0}^{+\infty}\Omega\left(t^\alpha\right)t^{d-1}\,dt\,d\theta \\
    &= d \vol[d](L) \int_0^\infty \Omega\left(t^\alpha\right)t^{d-1}\,dt.
\end{align*}
Rearranging this equation yields the claim.
\end{proof}

For a Polish space $(X,d)$, we say a curve $\gamma:[0,1] \to X$ is a \textit{geodesic} if 
\begin{align*}
    d(\gamma_s, \gamma_t)=|t-s| d(\gamma_0, \gamma_1), \quad \forall s, t \in[0,1] .
\end{align*}
We say that $(X,d)$ is a \textit{geodesic space} if for every pair of points $x, y \in X$, there exists at least one geodesic $\gamma$ such that $\gamma_0 = x$ and $\gamma_1 = y$. A set $K\subset X$ is \textit{geodesically convex} if for every two points $x, y \in K$, \textbf{every} geodesic $\gamma$ connecting them satisfies $\gamma_t \in K$ for all $t \in [0,1]$. Furthermore, a function $h: X \to \R \cup \{+\infty\}$ is said to be \textit{geodesically convex} if for any geodesic $\gamma: [0,1] \to X$ and any $t \in [0,1]$, $h$ satisfies
\begin{equation}
    h(\gamma_t) \leq (1-t)h(\gamma_0) + t h(\gamma_1).
\end{equation}

\subsection{The Blaschke-Santal\'o inequality}
We now complete the sketch from the introduction, demonstrating how Lemma~\ref{l:transfer} establishes the equivalence between the geometric Blaschke-Santal\'o inequality \eqref{eq:BS} and Proposition~\ref{prop:Ball_Fradelizi_Meyer}. We begin with the geometric setup:

\begin{proposition}
\label{p:max_cost}
    Let $K \in \mathcal{C}^n$ and let $L \in \mathcal{B}(\R^n)$ be non-empty. Define $M(K,L) = \sup_{x\in K, y\in L} \langle x, y \rangle$. Then,
    \begin{enumerate}
        \item If $L$ is unbounded, then $M(K,L)=\infty$.
        \item If $L$ is bounded, then $0\leq M(K,L)<\infty$ and $L\subseteq M(K,L)K^\circ$. Moreover, one has $M(K,L)=0$ if and only if $L=\{o\}$.
    \end{enumerate}
\end{proposition}

\begin{proof}
We first study $K$. We will utilize the support function of $K$, which is given by $h_K(u) = \max_{x \in K} \langle x, u \rangle$. The maximum is achieved for every $u\in \s$ because $K$ is compact. We claim that $h_K>0$ on $\s$. Indeed, since $K$ has center of mass at the origin,
\[
\int_K \langle x, u \rangle \, dx = \left\langle \int_K x \, dx, u \right\rangle =\left\langle o, u \right\rangle = 0, \,\, u\in \s.
\]
Since the function $x \mapsto \langle x, u \rangle$ is continuous and integrates to zero over $K$, which has positive volume, it must change signs on $K$. Therefore, there must exist a point $x \in K$ such that $\langle x, u \rangle > 0$.  Thus, $h_K>0$. Now, by the Extreme Value Theorem, there exists $\epsilon>0$ such that
\[
\epsilon \coloneqq \min_{u \in \s} h_K(u) > 0.
\]
This implies that for any $u \in \s$, there exists an element $x \in K$ satisfying $\langle x, u \rangle \geq \epsilon$.

We now break into our cases; the first to consider is when $L$ is unbounded. Then, we can choose a sequence of vectors $\{y_k\}_{k=1}^{\infty} \subset L$ such that:
\[
\lim_{k \to \infty} |y_k| = \infty.
\]
For each $k$, we can write $y_k = |y_k| u_k$, where $u_k = \frac{y_k}{|y_k|}\in \s$. Moreover, from our analysis on $K$, we have that, for each direction $u_k$, there exists an $x_k \in K$ such that $\langle x_k, u_k \rangle \geq \epsilon.$ It follows that
\[
M(K,L)=\sup_{x \in K, y \in L} \langle x, y \rangle \geq \langle x_k, y_k \rangle = \langle x_k, |y_k| u_k \rangle = |y_k| \langle x_k, u_k \rangle \geq |y_k| \epsilon.
\]
Taking the limit as $k \to \infty$ on the right-hand side gives:
\[
\lim_{k \to \infty} |y_k| \epsilon = \infty, \quad \text{ergo,} \quad M(K,L)=\infty.
\]

The next case is when $L=\{o\}$. Then,
\[
M(K,L)=\sup_{x \in K, y \in L} \langle x, y \rangle = \sup_{x \in K} \langle x, o \rangle=0.
\]
It is immediate that $L=M(K,L)K^\circ=\{o\}$ in this case. We now consider our final case, when $L$ is bounded and there exists $y^\prime\neq o$ such that $y^\prime\in L$. Then, since $K$ is also bounded, there exists $R_K,R_L>0$ such that $|x|\leq R_K, |y|\leq R_L$ for every $x\in K$ and $y\in L$. Thus,
\[
M(K,L) = \sup_{x \in K, y \in L} \langle x, y \rangle \leq R_K R_L <\infty.
\]
To get $M(K,L)>0$, we take our non-zero $y^\prime\in L$ and use that there exists $x^\prime \in K$ such that $\langle x^\prime,y^\prime\rangle \geq |y^\prime|\epsilon$ to get
\[
M(K,L) = \sup_{x \in K, y \in L} \langle x, y \rangle \geq \left\langle x^\prime, y^\prime \right\rangle  \geq |y^\prime|\epsilon >0.
\]
Finally, in this case, if $y\in L$, then, for every $x\in K,$ $\langle x,y \rangle \leq M(K,L)$, and, therefore,  $y/M(K,L)\in K^\circ$. The set inclusion follows.
\end{proof}

\begin{proposition}
\label{p:BS_two}
    Let $K \in \mathcal{C}^n$ and $L \in \mathcal{B}(\R^n)$. Define $M(K,L) = \sup_{x\in K, y\in L} \langle x, y \rangle$. Then,
    \begin{equation}
    \vol(K)^{\frac{1}{2}} \vol(L)^{\frac{1}{2}} \leq \vol(\B) M(K,L)^{\frac{n}{2}}.
    \end{equation}
    If $0<\vol(L)<\infty$, then there is equality if and only if one has, up to removal of null sets, that $K$ is a centered ellipsoid and $L=M(K,L)K^\circ$.
\end{proposition}

\begin{proof}
    If $\vol(L)=0$, the inequality is trivially satisfied, as the left-hand side is zero. Similarly, if $L$ is unbounded, then Proposition~\ref{p:max_cost} states that $\vol(L)\leq \infty=M(K,L)$, and so the inequality is true. Otherwise, taking volume throughout Proposition~\ref{p:max_cost} produces $\vol(L) \leq M(K,L)^n \vol(K^\circ)$. Applying the classical geometric Blaschke-Santal\'o inequality \eqref{eq:BS} yields the claim.
\end{proof}

\begin{proof}[Proof of Proposition~\ref{prop:Ball_Fradelizi_Meyer} when $f\in \mathcal{F}_n$]
In Lemma~\ref{l:transfer}, we set $m=2$, $X_1=X_2=\R^n$, and $\lambda_1=\lambda_2=\frac{1}{2}$. We choose the cost function
\[
    c(x,y) = -\langle x,y\rangle,
\]
which implies $-c(x,y) = \langle x,y\rangle$. Thus, $a=0$ and $b=\infty$. By Proposition~\ref{p:BS_two}, the geometric hypothesis $(1)$ of Lemma~\ref{l:transfer} is satisfied for the classes of sets $\mathcal{C}^n$ and $\mathcal{B}(\R^n)$ and
\[
    F(r) = \vol(\B) r^{\frac{n}{2}}. \quad \text{Consequently,} \quad dF(r) = \frac{n}{2} \vol(\B)r^{\frac{n}{2}-1}\,dr.
\]
Now suppose that $f\in\mathcal{F}_n$, $g:\R^n\to\R_+$ is a measurable function, and $\Omega:(0,\infty)\to[0,\infty)$ is a measurable, non-increasing function such that
\[
    f(x)g(y) \leq \Omega^2(\langle x,y\rangle) \qquad \text{whenever } \langle x,y\rangle > 0.
\]
Applying Lemma~\ref{l:transfer}, we directly obtain
\[
    \left(\int_{\R^n} f(x)\,dx\right)^{\frac{1}{2}} \left(\int_{\R^n} g(y)\,dy\right)^{\frac{1}{2}} 
    \leq \frac{n}{2} \vol(\B) \int_0^\infty \Omega(r) r^{\frac{n}{2}-1}\,dr.
\] 
This is precisely formula \eqref{eq:BS_fun}. Indeed, by making the change of variables $r=u^2$ and converting to polar coordinates, we see that
\begin{align*}
\frac{n}{2} \vol(\B) \int_0^\infty \Omega(r) r^{\frac{n}{2}-1}\,dr
&= n\vol(\B) \int_0^\infty \Omega(u^2)u^{n-1}\,du = \int_{\R^n}\Omega(|x|^2)\,dx.
\end{align*}
We conclude.
\end{proof}

\noindent {\bf Remark.} A recent update of \eqref{eq:BS} was established in \cite{FGSZ24}. Let $K\subset \R^n$ be a compact set with positive volume such that $o\in \operatorname{int}(\operatorname{conv} (K))$. Then,
\begin{equation}
\label{eq:BS_updated}
\vol(K)\vol(K^\circ) \leq \vol(\B)^2 p(K),
\end{equation}
with equality if and only if $K$ is a centered ellipsoid up to removal of null sets. The penalty term $p(K)$ is
\[
p(K) = \min\{c(K),c(K^\circ)\}, \quad \text{where} \quad c(K) = \left(1-\langle b_s(K^\circ),b(K) \rangle\right)^{n+1}.
\]
The barycenter $b(K)$ of $K$ (respectively, the Santal\'o point $b_s(K)$ of $K$) is the unique vector $x\in\R^n$ such that $K-x$ (respectively, $(K-x)^\circ$) has its center of mass at the origin. Note that $p(K) \geq 1.$ 
Throughout this work, every application of \eqref{eq:BS} could be replaced with applications of \eqref{eq:BS_updated} with weaker assumptions on $K$. In particular, this implies stronger versions of Theorems~\ref{t:hiBS} and \ref{t:hiBS_2} and Lemma~\ref{l:general_homogeneous_santalo}. We leave the details to the reader to fill in. \hfill \qedsymbol

\subsection{Duality for Sets}
\label{sec:set_prop}
For two Polish spaces $X$ and $Y$ and a continuous cost function $c: X \times Y \to (-\infty, \infty)$, the cost induces a natural polarity between the subsets of $X$ and $Y$. For a subset $A \subset X$, we define its $c$-polar in $Y$ by:
\begin{equation}
\label{eq:abstract_lowered_polar}
A^{\circ, c} :=\{ y \in Y : -c(x,y) \le 1, \forall x \in A \}.
\end{equation}
Similarly, for a subset $B \subset Y$, we define its star $c$-polar in $X$ by:
\begin{equation}
\label{eq:abstract_raised_polar}
B^{\star, c} := \{ x \in X : -c(x,y) \le 1, \forall y \in B \}.
\end{equation}

In the next proposition, we list facts about the two dualities. For one of the properties, we want to state that polars of sets are convex. We will use geodesic convexity; recall that the geodesic convex hull $\operatorname{conv}(A)$ of set $A$ in a geodesic space is defined as the intersection of all geodesically convex sets containing $A$. Taking $X$ and $Y$ to be Polish geodesic spaces, we say the dualities ${}^{\circ, c}$ and ${}^{\star, c}$ are \textit{convex dualities} if for every $x \in X$ and $y \in Y$, the sublevel sets $\{y' \in Y : -c(x, y') \le 1\}$ and $\{x' \in X : -c(x', y) \le 1\}$ are geodesically convex.

\begin{proposition}
\label{p:dualities}
    Let $X$ and $Y$ be Polish spaces and $c: X \times Y \to (-\infty, \infty)$ be a continuous cost function. Then the following properties hold:
    \begin{enumerate}
        \item Let $A_1, A_2 \subset X$. If $A_1 \subseteq A_2$, then $(A_1)^{\circ, c} \supseteq (A_2)^{\circ, c}$.
        \item Let $B_1, B_2 \subset Y$. If $B_1 \subseteq B_2$, then $(B_1)^{\star, c} \supseteq (B_2)^{\star, c}$.
        \item For any subsets $A \subset X$ and $B \subset Y$, one has $A \subseteq (A^{\circ, c})^{\star, c}$ and $B \subseteq (B^{\star, c})^{\circ, c}$.
        \item For any subsets $A \subset X$ and $B \subset Y$, one has $((A^{\circ, c})^{\star, c})^{\circ, c} = A^{\circ, c}$ and $((B^{\star, c})^{\circ, c})^{\star, c} = B^{\star, c}$.
        \item The collections of $c$-polar sets defined by $\mathcal{S}_X := \{B^{\star, c} : B \subset Y\} \subset 2^X$ and $\mathcal{S}_Y := \{A^{\circ, c} : A \subset X\} \subset 2^Y$ are invariant under the composition of the polarities. That is, if $A \in \mathcal{S}_X$ and $B \in \mathcal{S}_Y$, then $(A^{\circ, c})^{\star, c} = A$ and $(B^{\star, c})^{\circ, c} = B$.
    \end{enumerate}
    Suppose further that $X$ and $Y$ are geodesic spaces, and that the dualities are convex. Then:
    \begin{enumerate}
        \setcounter{enumi}{5}
        \item For any sets $A \subset X$ and $B \subset Y$, one has $A^{\circ, c} = (\operatorname{conv}(A))^{\circ, c}$ and $B^{\star, c} = (\operatorname{conv}(B))^{\star, c}$, where $\operatorname{conv}$ denotes the geodesic convex hull. Moreover, if there exists a point $x_0 \in X$ such that $-c(x_0, y) \le 1$ for all $y \in Y$, then for any set $A \subset X$, one has $A^{\circ, c} = (A \cup \{x_0\})^{\circ, c} = (\operatorname{conv}(A \cup \{x_0\}))^{\circ, c}$. Similarly, if there exists $y_0 \in Y$ such that $-c(x, y_0) \le 1$ for all $x \in X$, then $B^{\star, c} = (B \cup \{y_0\})^{\star, c} = (\operatorname{conv}(B \cup \{y_0\}))^{\star, c}$.
    \end{enumerate}
    Finally, suppose $X = M_{n,m}(\R)$, $Y = \R^n$, and the cost is given by $c(A,v) = -h_Q(A^t v)$ for some $Q \in \conbodo[m]$. Then:
    \begin{enumerate}
        \setcounter{enumi}{6}
        \item For every $K \in \conbode$, one has $(K^{\star, c})^{\circ, c} = K$.
    \end{enumerate}
\end{proposition}

\begin{proof}
    Properties (1) and (2) follow immediately from the definitions of the polar sets. For (3), let $x \in A$. By the definition of $A^{\circ, c}$, we have $-c(x,y) \le 1$ for all $y \in A^{\circ, c}$. This precisely means $x \in (A^{\circ, c})^{\star, c}$, establishing the inclusion. The inclusion for $B$ is likewise.
    
    For (4), applying (2) to the inclusion $B \subseteq (B^{\star, c})^{\circ, c}$ yields $B^{\star, c} \supseteq ((B^{\star, c})^{\circ, c})^{\star, c}$. Conversely, substituting the set $B^{\star, c}$ into property (3) yields $B^{\star, c} \subseteq ((B^{\star, c})^{\circ, c})^{\star, c}$, establishing equality. The proof for $A^{\circ, c}$ is identical. Property (5) is an immediate consequence of (4).
    
    For (6), since $A \subseteq \operatorname{conv}(A)$, property (1) immediately yields $A^{\circ, c} \supseteq (\operatorname{conv}(A))^{\circ, c}$. For the reverse inclusion, let $y \in A^{\circ, c}$. By definition, this means $-c(x,y) \le 1$ for all $x \in A$, which is equivalent to saying $A \subseteq \{x \in X : -c(x,y) \le 1\}$. By our hypothesis, this sublevel set is geodesically convex in $X$. Because $\operatorname{conv}(A)$ is the smallest geodesically convex set containing $A$, it follows that $\operatorname{conv}(A) \subseteq \{x \in X : -c(x,y) \le 1\}$. This implies that for all $x \in \operatorname{conv}(A)$, $-c(x,y) \le 1$, meaning $y \in (\operatorname{conv}(A))^{\circ, c}$. This establishes equality. The proof for $B^{\star, c}$ follows symmetrically. 

    For the ``moreover'' part, the hypothesis that $-c(x_0, y) \le 1$ for all $y \in Y$ means, for any set $A \subset X$, the condition $-c(x,y) \le 1$ for all $x \in A \cup \{x_0\}$ is the same statement as $-c(x,y) \le 1$ for all $x \in A$. This immediately gives $A^{\circ, c} = (A \cup \{x_0\})^{\circ, c}$. Applying the first part to the set $A \cup \{x_0\}$ yields the second equality $(A \cup \{x_0\})^{\circ, c} = (\operatorname{conv}(A \cup \{x_0\}))^{\circ, c}$. The proof for $B^{\star, c}$ follows in the same manner.

    Finally for (7), write from \eqref{eq:lowered_polarity_set} that
    \[
    (K^{\star,Q})_{\circ,Q}=\bigcap_{A\in K^{\star,Q}}(AQ)^\circ.
    \]
    From item (3), we only need to establish $(K^{\star,Q})_{\circ,Q} \subseteq K$. By taking the (usual) polarity, it suffices to show 
    \[
    K^\circ\subseteq ((K^{\star,Q})_{\circ,Q})^\circ =\operatorname{conv}\{AQ:A\in K^{\star,Q}\}.
    \]
    Inserting the definition of $K^{\star,Q}$ from \eqref{eq:raised_polarity_set}, we obtain
    \[
    ((K^{\star,Q})_{\circ,Q})^\circ =\operatorname{conv}\{AQ:AQ\subseteq K^\circ,A\in M_{n,m}(\R)\}.
    \]
    Now, we consider only those $A$ that are of rank $1$. That is, 
    \[
    ((K^{\star,Q})_{\circ,Q})^\circ \supseteq \operatorname{conv}\{v\otimes u Q: v\otimes u Q\subseteq K^\circ, u\in\R^m, v\in\R^n\}:=C(K).
    \]
    The set $v\otimes u Q$ has a nice structure:
    \[
    v\otimes u Q = \left\{\langle u,x\rangle v:x\in Q\right\} = \left\{\langle u,x\rangle:x\in Q\right\} v = [-h_Q(-u),h_Q(u)]v.
    \]
   By definition of polarity, $[-h_Q(-u),h_Q(u)]v \subset K^\circ$ if and only if $h_Q(u)h_K(v)\leq 1$ and $h_Q(-u)h_K(-v)\leq 1$. This is immediate if $Q$ is also origin-symmetric. To show $C(K) \supseteq K^\circ$ in the general case, it suffices to show that for any boundary point $v \in \partial K^\circ$, there exists some $u \in \R^m$ such that $v \in v \otimes u Q \subseteq K^\circ$. 
    
    Since $v \in \partial K^\circ$ and $K^\circ$ is origin-symmetric, $h_K(v) = h_K(-v) = 1$, meaning $[-1,1]v \subseteq K^\circ$. Thus, we only need to find a vector $u \in \R^m$ satisfying $h_Q(u) = 1$ and $h_Q(-u) \le 1$. 

    Because $Q$ is a convex body with non-empty interior, there exists some direction $w \in \R^m$ such that $h_Q(w) > 0$. We compare $h_Q(w)$ and $h_Q(-w)$:
    \begin{itemize}
        \item If $h_Q(w) \ge h_Q(-w)$, we define $u = w / h_Q(w)$. By homogeneity, $h_Q(u) = 1$ and $h_Q(-u) = h_Q(-w)/h_Q(w) \le 1$.
        \item If $h_Q(-w) > h_Q(w)$, we define $u = -w / h_Q(-w)$, which similarly yields $h_Q(u) = 1$ and $h_Q(-u) < 1$. 
    \end{itemize}

    In either case, we obtain a direction $u$ such that $v \otimes u Q = [-h_Q(-u), 1]v \subseteq [-1, 1]v \subseteq K^\circ$. Since this set contains $v$, it follows that every boundary point of $K^\circ$ is contained in $C(K)$. Because $C(K)$ is convex and contains the origin, it must contain the convex hull of $\partial K^\circ$. Therefore, $C(K) \supseteq K^\circ$, completing the proof.
\end{proof}

\noindent {\bf Remark.} In \cite{AASW23,AASW23_2}, the authors of those works considered the case of a symmetric cost $c:X\times X\to (-\infty,\infty)$ and defined the $c$-dual of a set $K\subset X$ at level $t\in (-\infty,\infty)$ as 
\[
K^c_t=\left\{y\in X\bigg|\, \inf_{x\in K}c(x,y)\geq t\right\}.
\]
However, it is easy to see that such dualities are equivalent to \eqref{eq:abstract_lowered_polar} and \eqref{eq:abstract_raised_polar} for the case $X=Y$ for all $t\in \R$, as $\inf_{x\in K}c(x,y)\geq t$ if and only if $\sup_{x\in K}e^{-c(x,y)+t}\leq 1$; this is of the correct form. Nevertheless, their considerations were orthogonal to ours.

Firstly, in \cite{AASW23}, they were interested in generalizing some known results concerning finite costs in the theory of optimal transport, primarily the existence of optimal transport plans, to the case when the cost can take the value $+\infty$. The closest connection between their results and ours is that they introduce a condition (see \cite[Definition 3.2]{AASW23}) which is equivalent to (see \cite[Lemma 3.12]{AASW23}): if $c:X\times X\to (-\infty,\infty)$ is an upper semi-continuous cost function, then two probability measures $\mu, \nu \in \mathcal{P}\left(X\right)$ are $c$-compatible if and only if, for every set $K\subset X$ satisfying $K=(K^c_\infty)^c_{\infty}$ it holds
\[
\nu(K^c_\infty) \leq 1-\mu(K),
\]
and they are strongly $c$-compatible if and only if the inequality is strict when $\nu(K^c_\infty)\notin\{0,1\}$. Then, they showed in \cite[Theorem 1.1]{AASW23}, as an illustrative example, the existence of an optimal transport plan between $\mu$ and $\nu$ with certain properties.

Secondly, in \cite{AASW23_2}, they considered \textit{order-reversing quasi-involutions} $T:2^X\to 2^X$, which satisfy $K\subset TTK$ and, if $K\subset L$, then $TL \subset TK$. Of course, Proposition~\ref{p:dualities} shows that the dualities \eqref{eq:abstract_lowered_polar} and \eqref{eq:abstract_raised_polar} are examples. In this symmetric situation, items (1)-(5) are classical anyway (see, e.g., \cite[Lemma 3.10]{AASW23}). They showed \cite[Theorem 1.3]{AASW23_2} that, if $T:2^X\to 2^X$ is an order-reversing quasi-involution, then there exists a symmetric cost $c:X\times X\to \{\pm 1\}$ such that, for every $K\subset X$, $TK=K_0^c$. Along the way, they combine a particular duality they discovered, the so-called dual polarity, with the Blaschke-Santal\'o inequality for origin-symmetric convex bodies in $\R^n$ and the Gaussian measure by Cordero-Erausquin \cite{CE02} to obtain a Blaschke-Santal\'o-type inequality for the dual polarity and a collection of cone-like convex sets using the Gaussian measure. 

While these applications beautifully demonstrate the utility of their framework, it is a distinct structural direction from the functional Santal\'o-type inequalities via the transference principle developed in the present work. \hfill \qedsymbol

\subsection{Duality for Functions}
\label{sec:properties} 
In this section, we establish some important properties about the polarities ${}^{\star, c}$ and ${}^{\circ, c}$. First, we have that they are in duality with each other (a restatement of \eqref{eq:duality}).
\begin{proposition}
\label{p:leg_dual}
Let $X$ and $Y$ be two Polish spaces, and let $c: X \times Y \to \R$ be a continuous cost function. For any functions $\psi: X \to \R$ and $\varphi: Y \to \R$, $$\Lu c\Ld c \psi \leq \psi \quad \text{and} \quad \Ld c\Lu c \varphi \leq \varphi.$$
    It follows that, if $f: X\to\R_+$ and $g:Y\to\R_+$, then,
    $$(f^{\circ, c})^{\star, c} \geq f \quad \text{and} \quad (g^{\star, c})^{\circ, c}\geq g.$$
\end{proposition}

\begin{proof}
We will show the first inequality; the second is identical by symmetry. For any $x_0 \in X$, we have
\begin{align*}
\Lu c\Ld c \psi (x_0) &= \sup_{y \in Y} \left( -c(x_0, y) - \Ld c \psi(y) \right) \\
&= \sup_{y \in Y} \inf_{x \in X} \left(-c(x_0, y) + c(x, y) + \psi(x) \right) \leq \psi(x_0),
\end{align*}
where the final inequality follows simply by choosing $x = x_0$ in the infimum.
\end{proof}

In the next two propositions, we show that $\Ld c$ and $\Lu c$ are indeed functional versions of the geometric polarities $L\to L^{\circ, c}$ and $K\to K^{\star, c}$, respectively. We say a function $c: X \times Y \to \R$ is $1$-homogeneous in each variable if $c(sx, ty) = st c(x,y)$ for all $s,t > 0$, $x\in X,y\in Y$.

\begin{proposition}
\label{p:polar_gauge}
Let $X$ and $Y$ be geodesic vector spaces, and let $c: X \times Y \to \R$ be a continuous cost function that is $1$-homogeneous in each variable. Let $1 < p, q < \infty$ such that $\frac{1}{p} + \frac{1}{q} = 1$. Let $K \subset X$ and $L \subset Y$ be compact, geodesically convex sets containing the origin in their respective interiors. Then,
\[
\Ld c\left(\frac{\|\cdot\|_K^p}{p}\right)(y)=\frac{\|y\|_{K^{\circ, c}}^q}{q} \quad \text{and} \quad \Lu c\left(\frac{\|\cdot\|_L^q}{q}\right)(x)=\frac{\|x\|_{L^{\star, c}}^p}{p}.
\]
\end{proposition}

\begin{proof}
We show only the first equality; the second is symmetric. First, observe that $\Ld c\left(\frac{\|\cdot\|_K^p}{p}\right)$ is positively homogeneous of degree $q$. Indeed, for $\lambda > 0$, by substituting $x = \lambda^{q-1}z$ and noting that $p(q-1) = q$, we have:
\begin{align*}
\Ld c\left(\frac{\|\cdot\|_K^p}{p}\right)(\lambda y)
&= \sup_{x \in X}\left\{ -c(x, \lambda y) - \frac{\|x\|_K^p}{p} \right\} \\
&= \sup_{z \in X}\left\{ -c(\lambda^{q-1}z, \lambda y) - \frac{\|\lambda^{q-1}z\|_K^p}{p} \right\} \\
&= \lambda^q \sup_{z \in X}\left\{ -c(z, y) - \frac{\|z\|_K^p}{p} \right\} \\
&= \lambda^q \Ld c\left(\frac{\|\cdot\|_K^p}{p}\right)(y).
\end{align*}

Since a positively homogeneous function is determined by one level set, it suffices to show that $\Ld c\left(\frac{\|\cdot\|_K^p}{p}\right)(y) \leq 1/q$ if and only if $y \in K^{\circ, c}$. 

Take $y \in K^{\circ, c}$. By the definition of the $c$-polar \eqref{eq:abstract_lowered_polar}, we have $-c(x, y) \le 1$ for all $x \in K$. By the separate $1$-homogeneity of $-c$, for any $x \in X$, 
\[
\|x\|_K^{-1} (-c(x, y)) = -c\left(\frac{x}{\|x\|_K}, y \right) \leq 1,
\]
which implies $-c(x, y) \le \|x\|_K$. Therefore,
\begin{align*}
\Ld c\left(\frac{\|\cdot\|_K^p}{p}\right)(y)
&\leq \sup_{x \in X}\left\{\|x\|_K - \frac{\|x\|_K^p}{p}\right\} \leq \sup_{\alpha \geq 0} \left\{\alpha - \frac{\alpha^p}{p}\right\} = \frac{1}{q}.
\end{align*}

On the other hand, if $y \notin K^{\circ, c}$, then there exists $x_0 \in X$ for which
\[
-c(x_0, y) > 1 \geq \|x_0\|_K.
\]
By homogeneity, we obtain $-c\left(\frac{x_0}{\|x_0\|_K}, y \right) > 1$, so that evaluating the supremum at the normalized vector gives:
\[
\Ld c\left(\frac{\|\cdot\|_K^p}{p}\right)(y) \geq -c\left(\frac{x_0}{\|x_0\|_K}, y \right) - \frac{\left\|\frac{x_0}{\|x_0\|_K}\right\|_K^p}{p} > 1 - \frac{1}{p} = \frac{1}{q},
\]
and the result follows.
\end{proof}

 Recall that the indicator function of a set $A$ is given by
\begin{equation}
    I_A(x)= \begin{cases}
0, \quad &x\in A,\\
+\infty, \quad &x \notin A.\\
\end{cases}
\end{equation}
We denote by $o_X$ and $o_Y$ the origins in the vector spaces $X$ and $Y$, respectively.

\begin{proposition}
\label{p:indicator_gauge_general}
Let $X$ and $Y$ be geodesic vector spaces, and let $c: X \times Y \to \R$ be a cost function that is $1$-homogeneous in each variable. Let $K \subset X$ and $L \subset Y$ be compact, geodesically convex sets containing the origin in their respective interiors. Then, 
\[
\Ld c(\|\cdot\|_K)(y)=I_{K^{\circ, c}}(y) \quad \text{and} \quad \Lu c(\|\cdot\|_L)(x)=I_{L^{\star, c}}(x).
\]
\end{proposition}

\begin{proof}
    We show only the first equality; the second is similar. Observe by \eqref{eq:abstract_lowered_polar} that, if $y\in K^{\circ, c}$, then $-c(x,y) \leq 1$ for every $x\in K$. Thus, for every $y\in K^{\circ, c}$ we obtain from the $1$-homogeneity of $c$ in each variable that $-c(z,y) \leq \|z\|_K$ by applying this fact to $x=\frac{z}{\|z\|_K}\in \partial K$ for any $z\in X\setminus \{o_X\}$. Thus, 
    \[
    \Ld c(\|\cdot\|_K)(y) = \sup_{z\in X}\left( -c(z, y) -\|z\|_K \right) \leq 0.
    \]
    But, by the definition of the supremum, evaluating at the origin yields $\Ld c(\|\cdot\|_K)(y) \geq -c(o_X, y)-\|o_X\|_K = 0$ (since $1$-homogeneity implies $-c(o_X, y) = 0$). We deduce that $\Ld c(\|\cdot\|_K)(y)=0$ for every $y\in K^{\circ, c}$.
    
    On the other hand, if $y\notin K^{\circ, c}$, then $-c(z,y) > \|z\|_K$ for some $z \in X$. From homogeneity, this implies $-c\left(\frac{z}{\|z\|_K}, y\right) > 1$. From the strictness of this inequality, we know that there actually exists some $\epsilon>0$ such that $-c\left(\frac{z}{\|z\|_K}, y\right) > 1+\epsilon$. We thus deduce that $-c(z,y) > (1+\epsilon)\|z\|_K$. Observe then that by evaluating the supremum over the ray $\lambda z$ for $\lambda > 0$,
    \begin{align*}
    \Ld c(\|\cdot\|_K)(y) &= \sup_{x\in X}\left(-c(x,y)-\|x\|_K\right) 
    \\
    &\geq \sup_{\lambda > 0}\left(-c(\lambda z, y)-\|\lambda z\|_K\right) 
    \\
    &\geq \sup_{\lambda > 0}\left(\lambda(1+\epsilon)\|z\|_K - \lambda\|z\|_K\right) 
    \\
    &= \epsilon \|z\|_K \sup_{\lambda > 0} \lambda = +\infty.
    \end{align*}
    Thus, for $y\notin K^{\circ, c}$, we have that $\Ld c(\|\cdot\|_K)(y) = +\infty$. The result follows.
\end{proof}

Using the concept of polarity induced by a cost function, we can state a special case of the functional cost-Santal\'o inequality in Definition~\ref{def:cost_sant} when $m=2$ and $\lambda=\frac{1}{2}$. A common situation in this case is when the functions on $X_1$ belong to some prescribed collection of non-negative, measurable functions $\mathcal{F}$ and the functions on $X_2$ are all non-negative, measurable functions. Then, we may take $f_1=f\in \mathcal{F}$ and $f_2=f^{\circ, c}$ and $\Omega(t)=e^{-\frac{t}{2}}$. The cost-Santal\'o inequality therefore implies that, for every $f\in \mathcal{F}$,
\[
\left(\int_{X_1}f(x)d\mu_1(x)\right)\cdot \left(\int_{X_2}f^{\circ, c}(y)d\mu_2(y)\right) \leq \left(\int_{a}^b e^{-\frac{t}{2}} dF(t)\right)^2.
\]
We specialize further to the symmetric case: $X_1\!=\!X_2\!=X\!$ and $d\mu_1\!=\!d\mu_2\!=\! \left(\int_{a}^b \!e^{-\frac{t}{2}} dF(t)\right)d\mu$. Then, our definition implies the inequality
\[
\left(\int_{X}f(x)d\mu(x)\right)\cdot \left(\int_{X}f^{\circ, c}(y)d\mu(y)\right) \leq 1,
\]
which is precisely, up to presentation, the cost-Santal\'o inequality introduced in \cite[Definition 1.2.9]{AGA2}. Thus, our Definition~\ref{def:cost_sant} completely absorbs this case.

\section{The transference principle and consequences}
\label{sec:main}
\subsection{The Principle}
As mentioned in the introduction, we will utilize generalized sup-convolutions and the framework from \cite{MMRR26}. Let $m\in \N$ and take $u=(u_1,\dots,u_m)$ and $t=(t_1,\dots,t_m)$ with $u_i,t_i>0$ and $\sum_{i=1}^mt_i=1$. We need the following means, where $p\in \R$:
\begin{equation}
\label{eq:means}
    \M_p^{(t)}(u) =\M_p^{(t)}(u_1,\dots,u_m) : =\begin{cases}
        \min_{i}\{u_i\}, & p=-\infty,
        \\
        \left(\sum_{i=1}^m t_iu_i^p\right)^{\frac{1}{p}}, & p \in (-\infty,0)\cup (0,\infty),
        \\
        \prod_{i=1}^m u_i^{t_i}, & p=0,
        \\
        \max_{i}\{u_i\}, & p=\infty.
    \end{cases}
\end{equation}

Generalized sup-convolutions were defined for arbitrary Polish measure spaces. Let $\mathcal{C}$ be a collection of analytic sets in a Polish space $(X,d,\mu)$ containing the empty set, with $X=X_1 \times \cdots \times X_m$, for Polish measure spaces $(X_i,\mu_i)$ and $\mu = \mu_1 \otimes \cdots \otimes \mu_m$.

Now, suppose we have $m$ collections of analytic sets $\mathcal{A}_i$ and define the $m$-tuple $\mathcal{A}=\left(\mathcal{A}_1,\dots,\mathcal{A}_m\right)$. Take functions $f_i,g_i\in \mathcal{F}\left(\mathcal{A}_i\right)$ and define the $m$-tuple of functions $$f=(f_1,\dots,f_m) \quad \text{and}  \quad g=(g_1,\dots,g_m).$$ Then, we say $\square f:Y\to [0,\infty]$ is the generalized sup-convolution of the tuple $f$ (where $(Y,d_Y)$ is another Polish space) if
\begin{enumerate}
    \item Monotonicity: $f_i\leq g_i$ pointwise for all $i$ implies $\square f \leq \square g$;
    \item Super-additivity: suppose the functions $g_i$ satisfy: if there exists $x_i\in X_i$ such that $g_i(x_i)>0$, then $f_i(x_i)\geq f_i(y_i)$ for all $y_i\in X_i$. Then, one has $\square\left(f+g\right) \geq \square f + \square g$ pointwise on $Y$.
    \item Preservation of measurability: let $A_i \in \mathcal{A}_i$, let $a_i > 0$ and set $A=\left(A_1,\dots,A_m\right)$ and $a=(a_1,\dots,a_m)$. Define the $m$-tuple of functions $f_{a,A}=\left(a_1\chi_{A_1},\dots,a_m\chi_{A_m}\right)$. Then, we require $\square f_{a,A}$ to be measurable.
\end{enumerate}
The classical sup-convolution \eqref{eq:sup_con} is the choice $m=2$ and $\mathcal{A}_1=\mathcal{A}_2=$ all analytic subsets of $\R^n$ (i.e., projections of Borel sets) and $Y$ is the Lebesgue space $\R^n$. We say that \eqref{eq:PL} is the functional inequality associated with this sup-convolution.

\begin{proposition}
\label{p:gen_sup}
    Let $m\in \N$ and let, for $i=1,\dots,m$, $(X_i,d_i,\mu_i)$ be Polish measure spaces and consider collections of analytic sets $\mathcal{A}_i\subset 2^{X_i}$. Let $-\infty \leq a<b \leq +\infty$ and let $\Phi:X_1\times\cdots \times X_m\to (a,b)$ be a continuous function.

    Then, the operator $\square f: (a,b) \to [0,\infty]$ given by
\begin{equation*}
\square f(r) = \sup_{\{x= (x_1,\dots,x_m) \colon \Phi(x) \geq r\} }\mathcal{M}_p^{(t)}(f), \quad -\infty < p \leq 1,
\end{equation*}
where $f\in \mathcal{F}\left(\mathcal{A}_1\right)\times \cdots\times \mathcal{F}\left(\mathcal{A}_m\right)$, is a generalized supremal convolution.
\end{proposition}
\begin{proof}[Proof Sketch]
    We merely sketch the argument, as it is a minor adaptation of the proof of \cite[Theorem 2.3]{MMRR26} (which concerns the operator $\square f$, but with $\Phi(x)=r$ in place of $\Phi(x) \geq r$ in its definition). Monotonicity is free, while super-additivity follows from the definition of the $p$-mean (this is where the requirement $p \leq 1$ comes into play). 
    
    The non-trivial observation is measurability. We briefly comment on this. Let $f = (\alpha_1\chi_{A_1}, \cdots, \alpha_m \chi_{A_m})$ for $A_i \in \mathcal{A}_i$, with $\alpha = (\alpha_1,\cdots, \alpha_m)$, $\alpha_i > 0$, producing
\[
\square f(r)= \mathcal{M}_p^{(t)}(\alpha) \chi_{(a,\sup\Phi(A))}(r),
\]
where $A =A_1 \times \cdots \times A_m$. 
In view of the analyticity of the $A_i$, there are continuous $\psi_i$ and Polish spaces $Z_i$ for which $\psi_i(Z_i) = A_i$. Then the map $\psi:= \Phi(\psi_1(\cdot),\cdots, \psi_m(\cdot))$ is continuous and satisfies $\psi(Z_1 \times \cdots \times Z_m) = \Phi(A)$.  Consequently, $\Phi(A)$ is analytic and universally measurable.  It follows that $(a,\sup \Phi(A))$ is measurable, as required. 
\end{proof}

The following theorem provides an equivalence between the functional inequality associated with a sup-convolution and its geometric counterpart. Since the supremum of Borel measurable functions may not be Borel measurable, we denote by $ \int_{Y}^{\star} \square f(x) \,d\nu(x)$ the lower integral of $\square f(x)$ against $\nu$.

\begin{theorem}[\cite{MMRR26}, Theorem 2.4]\label{thm:MMRR_equivalence}
Let $(X_1,d_1,\mu_1),\dots, (X_m,d_m,\mu_m), (Y,d,\nu)$ be Polish measure spaces. For $p \leq 1$ and $t=(t_1,\dots,t_m)$ with $t_i > 0$ and $\sum_{i=1}^mt_i=1$, the following are equivalent: under the above notations,
\begin{enumerate}
    \item[(i)] For any choice of non-empty sets $A_i \in \mathcal{A}_i$ and constants $\alpha_i > 0$, define the $m$-tuple of functions $\alpha = (\alpha_1\chi_{A_1},\dots, \alpha_m\chi_{A_m})$. One has the geometric inequality
    \begin{equation}\label{eq:geometric_hypothesis}
        \int_{Y} \square \alpha(x) \,d\nu(x) \geq \M_p^{(t)}\left(\alpha_1\mu_1(A_1),\dots, \alpha_m\mu_m(A_m)\right).
    \end{equation}
    \item[(ii)] For any choice of functions $f_i \in \mathcal{F}(\mathcal{A}_i)$, one has the functional inequality
    \begin{equation}\label{eq:functional_conclusion}
        \int_{Y}^{\star} \square f(x) \,d\nu(x) \geq \M_p^{(t)}\left(\int_{X_1} f_1(x) \,d\mu_1(x),\dots, \int_{X_m} f_m(x) \,d\mu_m(x)\right).
    \end{equation}
\end{enumerate}
\end{theorem}
 We pause to remark that, although \cite[Theorem 2.4]{MMRR26} was originally stated when \textbf{every} nontrivial superlevel set of $f_i$ belongs to the corresponding collection $\mathcal{A}_i$, the version we present here in Theorem~\ref{thm:MMRR_equivalence}, with ``every superlevel set'' changed to ``almost every superlevel set'' in the definition of $\mathcal{F}\left(\mathcal{A}_i\right)$, follows from the same proof, since the exceptional levels form a set of $\mu_i$-measure zero and therefore do not contribute to the layer-cake integrals. With this framework in hand, we establish Lemma~\ref{l:transfer}.
\begin{proof}[Proof of Lemma~\ref{l:transfer}]
    We will prove the equivalence by showing (1) implies (2) via the transference principle above, and then showing (2) implies (1) directly.

    To show (1) implies (2), we utilize the equivalence framework of Theorem \ref{thm:MMRR_equivalence}. The target space is $Y = (a,b)$. We choose the geometric mean in \eqref{eq:means} by setting $p=0$, and we assign the weights $t=(\lambda_1,\dots, \lambda_m)$. In what follows, we may replace each $\mathcal C_i$ by
\[
\widehat{\mathcal C}_i:=\mathcal C_i\cup\{\emptyset\}.
\]
Indeed,
\[
\mathcal F(\widehat{\mathcal C}_i)=\mathcal F(\mathcal C_i),
\]
since $\{f\geq r\}\neq\emptyset$ whenever $0<r<\|f\|_\infty$. Moreover, the geometric inequality still holds, as, if at least one $A_i$ is empty, then both sides of \eqref{eq:geometric_hypothesis} vanish.

    For functions $f_i \in \mathcal{F}(\mathcal{C}_{X_i})$, we define the operator:
    \begin{equation}\label{eq:general_sup_conv}
        \square f(r) := \sup_{\{(x_1,\dots,x_m) : -c(x_1,\dots,x_m) \geq r\}} \prod_{i=1}^m f_i(x_i)^{\lambda_i}.
    \end{equation}
    By Proposition~\ref{p:gen_sup}, the operation given by \eqref{eq:general_sup_conv} is a generalized sup-convolution mapping into $Y$. By hypothesis, $Y$ is endowed with the Lebesgue-Stieltjes measure $\nu$ given by $d\nu(r) = dF(r)$, which satisfies $\nu((a, M)) = F(M)$ for $M \in [a,b]$.

    To utilize the equivalence in Theorem~\ref{thm:MMRR_equivalence}, we first establish the geometric hypothesis \eqref{eq:geometric_hypothesis}. Next, we insert $\alpha = (\alpha_1\chi_{A_1},\dots, \alpha_m\chi_{A_m})$, with indicator functions $\alpha_i \chi_{A_i}$ for non-empty $A_i \in \mathcal{C}_{X_i}$ and $\alpha_i>0$, into our sup-convolution \eqref{eq:general_sup_conv}, whereby, for almost all $r\in (a,b)$,
    \[
    \square \alpha (r) = \left(\prod_{i=1}^m \alpha_i^{\lambda_i}\right) \chi_{\left(a, \sup_{x_i \in A_i} -c(x_1,\dots,x_m)\right)}(r).
    \]
    Integrating this against $\nu$ yields:
    \begin{align}
        \int_{a}^b \square\alpha(r) \,d\nu(r) &= \left(\prod_{i=1}^m \alpha_i^{\lambda_i}\right) F\left(\sup_{x_i \in A_i} -c(x_1,\dots,x_m)\right).
    \end{align}
    Applying the geometric hypothesis (1), we obtain
    \begin{align*}
        \int_{a}^b \square\alpha(r) \,d\nu(r) &\geq \prod_{i=1}^m \left(\alpha_i\mu_{X_i}(A_i)\right)^{\lambda_i} 
        \\
        &= \M_0^{(\lambda)}\left(\alpha_i\mu_{X_i}(A_i),\dots, \alpha_m\mu_{X_m}(A_m)\right),
    \end{align*}
    where $\lambda = (\lambda_1,\dots, \lambda_m)$

    Therefore, because we have satisfied the geometric condition, Theorem~\ref{thm:MMRR_equivalence} guarantees that the functional inequality holds for any measurable functions $f_i\in \mathcal{F}\left(\mathcal{A}_i\right)$:
    \begin{equation}\label{eq:general_lifted_functional}
        \int_{(a,b)}^\star \left( \sup_{\{-c(x_1,\dots,x_m) \geq r\}} \prod_{i=1}^m f_i(x_i)^{\lambda_i} \right) \,d\nu(r) \geq \prod_{i=1}^m \left(\int_{X_i} f_i(x_i)\,d\mu_{X_i}(x_i)\right)^{\lambda_i}.
    \end{equation}
    Now suppose that $f_1,\dots,f_m$ satisfy the functional hypothesis 
    \begin{equation}
    \label{eq:pointwise_hypo}
    \prod_{i=1}^m f_i(x_i)^{\lambda_i}\leq \Omega\bigl(-c(x_1,\dots,x_m)\bigr)\quad \text{whenever} -c(x_1,\ldots,x_m)\in(a,b].
    \end{equation}
    If $\prod_{i=1}^m f_i(x_i)^{\lambda_i}=0,$ the above inequality is always true. We claim that, for every tuple $(x_1,\ldots,x_m)$ contributing to the supremum in \eqref{eq:general_lifted_functional} and satisfying
    \[
    \prod_{i=1}^m f_i(x_i)^{\lambda_i}>0, \quad \text{one has } \,-c(x_1,\ldots,x_m)\in(a,b].
    \]
    Indeed, on one hand, since $r\in(a,b)$, and the tuple satisfies $-c(x_1,\ldots,x_m)\geq r,$ we have
    \[
    -c(x_1,\ldots,x_m)>a.
    \]
     On the other hand, from the fact that $f_i\in\mathcal F(\mathcal A_i)$, we may choose $t_i\in(0,f_i(x_i))$ such that $\{f_i\geq t_i\}\in\mathcal A_i$ for every $i$. Since $x_i\in\{f_i\geq t_i\}$, the hypothesis on the collections $\mathcal A_i$ implies that
    \[-c(x_1,\ldots,x_m)\leq \sup_{\substack{y_i\in\{f_i\geq t_i\}\\1\leq i\leq m}} -c(y_1,\ldots,y_m)\leq b.\]
    Hence, $-c(x_1,\ldots,x_m)\in(a,b]$, and therefore, by \eqref{eq:pointwise_hypo} and the monotonicity of $\Omega$, \[\prod_{i=1}^m f_i(x_i)^{\lambda_i} \leq \Omega\bigl(-c(x_1,\ldots,x_m)\bigr)\leq \Omega(r).\]
    Plugging this upper bound into the left side of \eqref{eq:general_lifted_functional} yields the functional conclusion (2):
    \begin{equation}
    \label{eq:trans_proved}
        \prod_{i=1}^m \left(\int_{X_i} f_i(x_i)\,d\mu_{X_i}(x_i)\right)^{\lambda_i} \leq \int_{a}^b \Omega(r) \,d\nu(r).
    \end{equation}

    Conversely, to show (2) implies (1), let $A_i \in \mathcal{C}_{X_i}$ for $i=1,\dots,m$. Define $M = \sup_{x_i \in A_i} -c(x_1,\dots,x_m)$. Choose $f_i = \chi_{A_i}$, and define $\Omega(r) = 1$ for $r \leq M$ and $\Omega(r) = 0$ for $r > M$. It trivially follows that, whenever $(x_1,\dots,x_m)\in X_1\times\cdots\times X_m$, are such that $-c(x_1,\dots,x_m)\in (a,b]$,
    \[
    \prod_{i=1}^m f_i(x_i)^{\lambda_i} = \prod_{i=1}^m \chi_{A_i}(x_i) \leq \Omega(-c(x_1,\dots,x_m)).
    \]
    Applying the functional hypothesis (2), we  recover the geometric inequality (1):
    \[
    \prod_{i=1}^m \mu_{X_i}(A_i)^{\lambda_i} \leq \int_{a}^b \Omega(r) \,d\nu(r) = \nu((a, M)) = F(M) = F\left(\sup_{x_i \in A_i} -c(x_1,\dots,x_m)\right). \qedhere
    \]
\end{proof}

\noindent {\bf Remark.} By repeating the same proof line by line, Lemma~\ref{l:transfer} can be extended to the means $\mathcal{M}_p^{(\cdot)}(\cdot)$ whenever $p \leq 1$. We omit the details for the sake of presentation; the curious reader can fill them in. 

\subsection{Helpful Consequence}
We conclude this section by obtaining an consequence of Lemma~\ref{l:transfer}, which we will use for the results stemming from geometric Blaschke-Santal\'o inequalities. First, we need the following extension of Proposition~\ref{p:max_cost}. We will use the polarity from \eqref{eq:abstract_lowered_polar}. The non-degeneracy condition in \eqref{eq:non_degen} below is an extension of the idea that, for a convex body $K$, $h_K >0$ if and only if $K^\circ$ is bounded.
\begin{proposition}
\label{p:max_cost_generalized}
    Let $(X,\|\cdot\|_X)$ and $(Y,\|\cdot\|_Y)$ be finite dimensional normed spaces. Let $c: X \times Y \to \R$ be a continuous cost function that is positively 1-homogeneous in the second variable, meaning $c(x, ty) = tc(x,y)$ for all $t > 0$.
    
    Suppose a bounded set $K\subset X$ satisfies the non-degeneracy condition:
    \begin{equation}
        \label{eq:non_degen}
        \inf_{u \in \mathbb{S}_Y} \sup_{x \in K} -c(x,u) = \epsilon > 0,
    \end{equation}
    where $\mathbb{S}_Y = \{y \in Y \mid \|y\|_Y = 1\}$. Let $L \subset Y$ be a non-empty set and suppose that $-c$ is non-negative on $K\times L$. Set $M(K,L) = \sup_{x\in K, y\in L} -c(x, y)$. Then,
    \begin{enumerate}
        \item If $L$ is unbounded, then $M(K,L)=\infty$.
        \item If $L$ is bounded, then $0\leq M(K,L)<\infty$ and $L\subseteq M(K,L)K^{\circ, c}$. Moreover, if $-c(x,y)=0$ for all $x \in K$ implies $y=o$, then $M(K,L)=0$ implies $L=\{o_Y\}$.
    \end{enumerate}
\end{proposition}

\begin{proof}
    We first consider the case when $L$ is unbounded. There exists a sequence $\{y_k\}_{k=1}^{\infty} \subset L$ such that $\lim_{k \to \infty} \|y_k\|_Y = \infty$. For each $k$, we write $y_k = \|y_k\|_Y u_k$, where $u_k \in \mathbb{S}_Y$. By the non-degeneracy condition \eqref{eq:non_degen}, for each direction $u_k$, there exists an $x_k \in K$ such that $-c(x_k, u_k) \geq \epsilon$. Utilizing the positive 1-homogeneity of $c$ in the second variable, we have:
    \[
    M(K,L) = \sup_{x \in K, y \in L} -c(x, y) \geq -c(x_k, y_k) = \|y_k\|_Y (-c(x_k, u_k)) \geq \|y_k\|_Y \epsilon.
    \]
    Taking the limit as $k \to \infty$, the right-hand side diverges to infinity. Thus, $M(K,L) = \infty$.
    
    Next, assume $L$ is bounded. Because $K$ and $L$ are both bounded sets and $c$ is a continuous function, the supremum of $-c$ on $K \times L$ is finite, so $M(K,L) < \infty$. Since $-c$ is non-negative on $K\times L$, it is immediate that $M(K,L) \geq 0$. 
    
    If $M(K,L) = 0$, then $-c(x,y) = 0$ for all $x \in K$ and $y \in L$. By the homogeneity on the cost function and the non-degeneracy condition \eqref{eq:non_degen}, this implies $y = o$ for any $y \in L$, so $L = \{o\}$.
    
    Now assume $M(K,L) > 0$. Let $y \in L$. By the definition of the supremum, for all $x \in K$, we have $-c(x, y) \leq M(K,L)$. Because $M(K,L) > 0$ and $-c$ is positively 1-homogeneous in the second variable, we can divide both sides by $M(K,L)$:
    \[
    -c\left(x, \frac{y}{M(K,L)}\right) = \frac{1}{M(K,L)} (-c(x,y)) \leq 1, \qquad \forall x \in K.
    \]
    By the definition of the cost-polar, this implies $\frac{y}{M(K,L)} \in K^{\circ, c}$, or $y \in M(K,L)K^{\circ, c}$, completing the set inclusion.
\end{proof}

The next lemma, which is the sought-out consequence, is an extension of Proposition~\ref{p:BS_two}. 

\begin{lemma}
\label{l:general_homogeneous_santalo}
Let $\nu$ be a Borel measure on $\R^n$ and let $\mu$ be a Borel measure on
$\R^d$ which is homogeneous of degree $\alpha>0$. Suppose there
exists a continuous cost function $c:\R^n\times\R^d\to\R$ and a collection
$\mathcal C\subset 2^{\R^n}$ of bounded, measurable sets such that
\begin{enumerate}
    \item $c$ is positively $1$-homogeneous in the second variable;
    \item $-c$ is a non-negative mapping on
    $\mathcal C\times\mathcal B(\R^d)$;
    \item any $K\in\mathcal C$ failing the non-degeneracy condition \eqref{eq:non_degen} has $\nu$-measure zero;
    \item there exists $\lambda\in(0,1)$ and a set $B\in\mathcal C$ such that,
    for every $K\in\mathcal C$,
    \begin{equation}
        \label{eq:gen_geom_hypo}
        \nu(K)^\lambda \mu(K^{\circ,c})^{1-\lambda}
        \leq
        \nu(B)^\lambda \mu(B^{\circ,c})^{1-\lambda}
        :=
        C_{\nu,\mu}.
    \end{equation}
\end{enumerate}

Define the function $F(r)=C_{\nu,\mu}\,r^{\alpha(1-\lambda)}$ on $\R_+$. Then, the polish spaces $(\R^n,|\cdot|,\nu)$ and $(\R^d,|\cdot|,\mu)$ satisfy the cost-Santal\'o inequality associated with
$(\lambda,c,F)$ with domain $(\mathcal C,\mathcal B(\R^d))$. In the implied geometric
inequality,
\begin{equation}
\label{eq:implied_sant}
    \nu(K)^\lambda \mu(L)^{1-\lambda}
    \leq
    C_{\nu,\mu}\,M(K,L)^{\alpha(1-\lambda)},
\end{equation}
where
\[ K\in\mathcal C, \,\,L\in\mathcal B(\R^d), \qquad \text{and} \quad M(K,L)=\sup_{x\in K,y\in L}-c(x,y),
\]
we have an equality characterization: if $L$ is bounded with $\nu(K)\mu(L)>0$, then there is equality in
\eqref{eq:implied_sant} if and only if, up to removal of null sets, one has $L=M(K,L)K^{\circ,c}$ and $K$ is a set for which equality holds in \eqref{eq:gen_geom_hypo}.

Consequently, for every $f\in\mathcal F(\mathcal C)$ and non-negative
measurable function $g$ satisfying the pointwise bound
\[
    f(x)^\lambda g(y)^{1-\lambda}
    \leq
    \Omega(-c(x,y)), \qquad \forall \, x\in \R^n, \, y\in \R^d, \text{such that } (-c(x,y))>0,
\]
 for a non-increasing function $\Omega:\R_+\to\R_+$, one has the functional inequality
\begin{equation}
\label{eq:gen_func_concl}
\begin{split}
    \left(\int_{\R^n} f(x)\,d\nu(x)\right)^\lambda
    \left(\int_{\R^d} g(y)\,d\mu(y)\right)^{1-\lambda}
    &\leq
    C_{\nu,\mu}\,\alpha(1-\lambda)
    \int_0^\infty
    \Omega(r)r^{\alpha(1-\lambda)-1}\,dr
    \\
    &=\frac{C_{\nu,\mu}}{\mu(B_2^d)}
    \int_{\R^d}
    \Omega\left(|z|^{\frac{1}{1-\lambda}}\right)
    d\mu(z).
\end{split}
\end{equation}
\end{lemma}

\begin{proof}
We first establish the geometric inequality required for
Lemma~\ref{l:transfer}. Let $K\in\mathcal C$ and let
$L\in\mathcal B(\R^d)$ be non-empty. If $\mu(L)=0$ or $\nu(K)=0$, the inequality
holds trivially. Thus, we may assume $\nu(K)\mu(L)>0$. By
condition (3), the set $K$ satisfies the non-degeneracy condition \eqref{eq:non_degen}.

If $L$ is unbounded, Proposition~\ref{p:max_cost_generalized} yields $M(K,L)=\infty,$ and the inequality again holds. Assume now that $L$ is bounded. By Proposition~\ref{p:max_cost_generalized}, we have
\[
    0\leq M(K,L)<\infty \qquad \text{and} \qquad
    L\subseteq M(K,L)K^{\circ,c}.
\]
Since $\mu(L)>0$, this also implies $M(K,L)>0$. Hence, by the
$\alpha$-homogeneity of $\mu$,
\[
    \mu(L)
    \leq
    \mu\bigl(M(K,L)K^{\circ,c}\bigr)
    =
    M(K,L)^\alpha\mu(K^{\circ,c}).
\]
Raising this inequality to the power $1-\lambda$, multiplying by
$\nu(K)^\lambda$, and applying \eqref{eq:gen_geom_hypo}, we obtain
\[
    \nu(K)^\lambda \mu(L)^{1-\lambda}
    \leq
    C_{\nu,\mu}\,M(K,L)^{\alpha(1-\lambda)}.
\]
This proves \eqref{eq:implied_sant}.

We now characterize equality in \eqref{eq:implied_sant} when $L$ is bounded and
$\nu(K)\mu(L)>0$. In the preceding chain, equality holds if and only if equality
holds both in
\[
    \mu(L)
    \leq
    \mu\bigl(M(K,L)K^{\circ,c}\bigr) \qquad \text{and in} \qquad
    \nu(K)^\lambda\mu(K^{\circ,c})^{1-\lambda}
    \leq
    C_{\nu,\mu}.
\]
Therefore, we must have $L=M(K,L)K^{\circ,c}$ up to removal of null sets. The second equality holds if and only if $K$ is a
set for which equality holds in \eqref{eq:gen_geom_hypo}. This proves the
equality characterization.

Thus condition (1) of Lemma~\ref{l:transfer} is satisfied with $F(r)=C_{\nu,\mu}\,r^{\alpha(1-\lambda)}.$ By Definition~\ref{def:cost_sant}, the pair satisfies the associated
cost-Santal\'o inequality on the domain $(\mathcal C,\mathcal B(\R^d))$. Since \[
    dF(r)
    =
    C_{\nu,\mu}\,\alpha(1-\lambda)
    r^{\alpha(1-\lambda)-1}\,dr,
\]
condition (2) of Lemma~\ref{l:transfer} yields the first inequality in \eqref{eq:gen_func_concl}.

We now obtain the equality in \eqref{eq:gen_func_concl}. Since $\mu$ is homogeneous of degree $\alpha$, we have $\mu(rB_2^d)=r^\alpha\mu(B_2^d).$ Set
\[
    p=\frac{1}{1-\lambda}.
\]
Then
\[
    \int_{\R^d}\Omega(|z|^p)\,d\mu(z)
    =
    \int_0^\infty \Omega(r^p)\,d\mu(rB_2^d)
    =
    \alpha\mu(B_2^d)
    \int_0^\infty \Omega(r^p)r^{\alpha-1}\,dr.
\]
Using the change of variables $s=r^p$, we get
\[
    \int_{\R^d}\Omega(|z|^p)\,d\mu(z)
    =
    \frac{\alpha}{p}\mu(B_2^d)
    \int_0^\infty \Omega(s)s^{\frac{\alpha}{p}-1}\,ds.
\]
Since $1/p=1-\lambda$, this becomes
\[
    \int_{\R^d}\Omega(|z|^p)\,d\mu(z)
    =
    \alpha(1-\lambda)\mu(B_2^d)
    \int_0^\infty \Omega(s)s^{\alpha(1-\lambda)-1}\,ds,
\]
as claimed.
\end{proof}

The above lemma also makes sense for Polish spaces admitting scalar multiplication.

\section{Results from Blaschke-Santal\'o inequalities on Euclidean spaces}
\label{sec:euclidean}
We now list new functional Santal\'o inequalities on Euclidean spaces. To demonstrate our approach, let $q\in (0,n)$ and let $a\in L^\frac{n}{q}(\s)$ be non-negative. It was observed in \cite{CKLR26}, with the case $a\equiv 1$, that one may use H\"older's inequality and an appeal to the classical Blaschke-Santal\'o inequality \eqref{eq:BS} to get:
\begin{align*}
        &\vol(K)^\frac{n-q}{n}\int_{K^\circ}a\left(\frac{y}{|y|}\right)|y|^{-q}dy = \frac{ \vol(K)^\frac{n-q}{n}}{n-q}\int_{\mathbb{S}^{n-1}}a(\theta)\|\theta\|_{K^\circ}^{-(n-q)}d\theta
        \\
        &\leq  \frac{ \vol(K)^\frac{n-q}{n}}{n-q}\left(\int_{\mathbb{S}^{n-1}}\|\theta\|_{K^\circ}^{-n}d\theta\right)^{\frac{n-q}{n}}\cdot \left(\int_{\s}a(\theta)^\frac{n}{q}d\theta\right)^{\frac{q}{n}}
        \\
        &=\frac{n}{n-q} \left(\vol(K^\circ) \vol(K)\right)^{\frac{n-q}{n}} \vol(\B)^\frac{q}{n} \|a\|_{L^\frac{n}{q}(\s,\sigma_{n-1})},
        \\
        &\leq \frac{n}{n-q}\vol(\B)^\frac{2n-q}{n}\|a\|_{L^\frac{n}{q}(\s,\sigma_{n-1})}.
    \end{align*}
    There is equality if and only if $K^\circ$ is a centered ellipsoid and $a=c\|\cdot\|_{K^\circ}^{-q}$ for some $c>0$ almost everywhere. In particular, if $a\equiv 1$, then $K^\circ$ must be a centered Euclidean ball. Recall that $\|a\|_{L^\frac{n}{q}(\s,\sigma_{n-1})}$ denotes the $L^\frac{n}{q}$ norm of $a$ with respect to the Haar probability measure $\sigma_{n-1}$ on $\s$.
    
    We now invoke Lemma~\ref{l:general_homogeneous_santalo}, and first obtain that, if $K\subset \R^n$ is a compact set of positive volume such that $K$ or $K^\circ$ has its center of mass at the origin and $L\subset \R^n$ is a non-empty, measurable set, one has for $q\in (0,n)$ and non-negative $a\in L^\frac{n}{q}\left(\s\right)$,
    \[
    \vol(K)^\frac{n-q}{2n-q}\left(\int_{L}a\left(\frac{y}{|y|}\right)|y|^{-q}dy\right)^\frac{n}{2n-q} \leq \left(\frac{nM(K,L)^{n-q}}{n-q}\|a\|_{L^\frac{n}{q}(\s,\sigma_{n-1})}\right)^{\frac{n}{2n-q}}\vol(\B).
    \]
    If $0<\vol(L)<\infty$, there is equality if and only if, up to removal of null sets, $K$ is a centered ellipsoid, $L=M(K,L)K^\circ$, and $a=c\|\cdot\|_{K^\circ}^{-q}$ for some $c>0$ almost everywhere. Secondly, we get the following functional inequality, which, as far as we are aware, is new, even if $a\equiv 1$.

\begin{corollary}[The Weighted Functional Santal\'o inequality]
    Let $q \in (0, n)$, let $a\in L^\frac{n}{q}(\s)$ be non-negative and define $$C_{n,q}(a) = \frac{n-q}{n\|a\|_{L^1(\s,\sigma_{n-1})}} \left(\frac{n}{n-q}\|a\|_{L^\frac{n}{q}(\s,\sigma_{n-1})}\right)^{\frac{n}{2n-q}}.$$ Then, for every function $f \in \mathcal{F}_n$ and any non-negative measurable function $g$ on $\R^n$ that satisfy the pointwise bound
    \[
        f(x)^{\frac{n-q}{2n-q}} g(y)^{\frac{n}{2n-q}} \leq \Omega(\langle x, y \rangle) \quad\forall \,x, y \in \R^n \text{ with } \langle x,y\rangle>0,
    \]
    where $\Omega: \R_+ \to \R_+$ is a measurable, non-increasing function, one has:
    \begin{equation}
        \label{eq:functional_q}
        \left(\int_{\R^n} \!\!\!f(x) dx \right)^{\frac{n-q}{2n-q}} \left(\int_{\R^n} \!\!\!\!g(y) |y|^{-q}a\left(\frac{y}{|y|}\right) dy \right)^{\frac{n}{2n-q}} \leq C_{n,q}(a) \int_{\R^n}\!\!\! \Omega(|z|^\frac{2n-q}{n}) |z|^{-q}a\left(\frac{z}{|z|}\right) dz.
    \end{equation}
\end{corollary}

\subsection{Unconditional Log-concave Measures}
\label{sec:uncon}
Recall that a function $f:\R^n\to \R_+$ is said to be \textit{unconditional} if there exists a basis $e_1,\dots,e_n$ of $\R^n$ such that, for every choice of signs $(\varepsilon_1,\dots,\varepsilon_n)\in \{-1,1\}^n$, and $x=(x_1,\dots,x_n)\in \R^n$, one has
\[
f\left(\sum_{i=1}^nx_ie_i\right)=f\left(\sum_{i=1}^nx_i\varepsilon_ie_i\right).
\]
It was shown in \cite[Theorem 7]{FMZ23} that, if $\mu$ is a Borel measure on $\R^n$ with an unconditional, log-concave density with respect to Lebesgue measure, e.g. Gaussian, then, for every $K\in\mathcal{K}_e^n$,
\begin{equation}
\label{eq:unconditional_BS}
    \mu(K)\mu(K^\circ)\leq \mu(B_2^n)^2.
\end{equation}
We obtain the following functional counterpart.

\begin{corollary}[Functional Santal\'o inequality for unconditional log-concave measures]
\label{c:unconditional_functional_BS}
Let $\mu$ be a probability measure on $\R^n$ with an unconditional, log-concave density with respect to Lebesgue measure. Let $f\in\mathcal{F}(\mathcal{K}_e^n)$ and let $g:\R^n\to\R_+$ be measurable. Suppose that $\Omega:\R_+\to\R_+$ is measurable and non-increasing and that
\[
    f(x)^{\frac12}g(y)^{\frac12}
    \leq
    \Omega(\langle x,y\rangle), \qquad \forall \,\, x,y\in\R^n \text{such that} \, \langle x,y\rangle>0.
\]
Then,
\begin{equation}
\label{eq:unconditional_functional_BS}
    \left(\int_{\R^n}f(x)\,d\mu(x)\right)^{\frac12}
    \left(\int_{\R^n}g(y)\,d\mu(y)\right)^{\frac12}
    \leq
    \int_{\R^n}\Omega(|x|^2)\,d\mu(x).
\end{equation}
\end{corollary}

\begin{proof}
Write $d\mu(x)=\phi(x)\,dx$. For $r>0$, define the measure $\mu_r$ on $\R^n$ by
\[
    d\mu_r(x)=r^{\frac n2}\phi(\sqrt r\,x)\,dx.
\]
Then $\mu_r$ also has an unconditional, log-concave density and, for every measurable set $A\subset\R^n$,
\[
    \mu_r(A)=\mu(\sqrt r\,A).
\]
Consequently, applying \eqref{eq:unconditional_BS} with respect to the measure $\mu_r$ and the body $\frac{1}{\sqrt r}K$, we obtain
\begin{align*}
    \mu(K)\mu(rK^\circ)
    &=
    \mu_r\left(\frac{1}{\sqrt r}K\right)
    \mu_r\left(\sqrt r\,K^\circ\right)
    \leq
    \mu_r(B_2^n)^2
    =
    \mu(\sqrt r\,B_2^n)^2.
\end{align*}

Now let $K\in\mathcal{K}_e^n$ and let $L\subset\R^n$ be a non-empty measurable set. Set, as usual, $M(K,L)=\sup_{x\in K,\,y\in L}\langle x,y\rangle.$ Since $L\subseteq M(K,L)K^\circ,$ the preceding inequality with $r=M(K,L)$, yields
\[
    \mu(K)^{\frac12}\mu(L)^{\frac12}
    \leq
    \mu\left(\sqrt{M(K,L)}\,B_2^n\right).
\]
Thus, Lemma~\ref{l:transfer} applies with both Polish measure spaces being $(\R^n,|\cdot|,\mu)$, the parameters $\lambda_1=\lambda_2=\frac12$, the continuous cost function $c(x,y)=-\langle x,y\rangle,$ the domain $\left(\mathcal{K}_e^n,\mathcal{B}(\R^n)\right),$ and the profile function
\[
    F(r)=\mu(\sqrt r\,B_2^n).
\]
We therefore obtain
\[
    \left(\int_{\R^n}f(x)\,d\mu(x)\right)^{\frac12}
    \left(\int_{\R^n}g(y)\,d\mu(y)\right)^{\frac12}
    \leq
    \int_0^\infty\Omega(r)\,dF(r).
\]
Finally, since we have by polar coordinates
\[
    F(r)
    =
    \mu\left(\left\{x\in\R^n:|x|^2\leq r\right\}\right)=\int_{\s}\int_0^{\sqrt{r}}t^{n-1}\phi(t\theta)dt\,d\theta,
\]
we see that
\[
dF(r)=\left(\frac{1}{2}r^{\frac{n}{2}-1}\int_{\s}\phi\left(\sqrt{r}\theta\right)d\theta\right)dr.
\]
The differentiation underneath the integral sign was justified from the fact that integrable, log-concave functions are bounded by an exponential.
Therefore, Fubini's theorem, a variable substitution $r=u^2$ and polar coordinates again yields
\begin{align*}
    \int_0^\infty\Omega(r)\,dF(r)
    &=
    \frac{1}{2}\int_0^\infty\Omega(r)r^{\frac{n}{2}-1}\int_{\s}\phi\left(\sqrt{r}\theta\right)d\theta \,dr,
    \\
    &=\frac{1}{2}\int_{\s}\int_0^\infty\Omega(r)r^{\frac{n}{2}-1}\phi\left(\sqrt{r}\theta\right) dr\,d\theta
    \\
    &=\int_{\s}\int_0^\infty\Omega(u^2)r^{n-1}\phi\left(u\theta\right)du\,d\theta
    \\
    &=\int_{\R^n}\Omega(|x|^2)\phi(x)dx,
\end{align*}
which proves \eqref{eq:unconditional_functional_BS}.
\end{proof}

\subsection{The Ball-Santal\'o inequality}
\label{sec:joint}
We now invoke our framework to deduce the functional version of the isotropic Ball-Santal\'o inequality.
\begin{proof}[Proof of Theorem~\ref{thm:func_ball_san}]

If $u=o$, then the inequality reads $0\leq 0$, and is thus true. Henceforth, we assume $u\neq o$. For $u\in \R^n$, we define the Borel measure $\mu_u$ by $d\mu_u(x) = \langle x, u \rangle^2 dx$. Notice that $\mu_u$ is $(n+2)$-homogeneous. Then, the isotropic Ball-Santal\'o inequality \eqref{eq:better_ball_san} for isotropic convex bodies is precisely that $$\mu_u(K)^\frac{1}{2}\mu_u(K^\circ)^\frac{1}{2} \leq \mu_u(B_2^n), \qquad K\in \conbodI.$$ 
Thus, by Lemma~\ref{l:general_homogeneous_santalo}, for every $K \in \conbodI$ and every non-empty $L \in \mathcal{B}\left(\R^n\right)$, we have the inequality:
     \begin{equation}
     \label{eq:BSiso}
        \mu_u(K)^{\frac{1}{2}} \mu_u(L)^{\frac{1}{2}} \leq \mu_u(B_2^n) M(K,L)^{\frac{n+2}{2}}.
    \end{equation}
Finally, our claimed functional inequality follows from \eqref{eq:gen_func_concl} with $\mu_1=\mu_2=\mu_u$ and $\lambda=\frac{1}{2}$. Notice that $C_{B,\mu_1,\mu_2}=\mu_2(\B)=\mu_u(\B)$ in this case.
\end{proof}

\noindent {\bf Remark.} We recover the inequality \eqref{eq:fun_ball} from Theorem~\ref{thm:func_ball_san} for $f, g \in \mathcal{F}(\conbodI)$. Indeed, applying Theorem~\ref{thm:func_ball_san} to $u=e_i$ and adding up gives the following:
\[
\sum_{i=1}^n \left(\int_{\R^n} f(x) x_i^2 dx\right) \left(\int_{\R^n} g(y) y_i^2 dy\right) \leq \sum_{i=1}^n\left(\int_{\R^n} \Omega(|z|^2) z_i^2 dz \right)^2.
\]
On the other hand, expanding the quadratic form $\langle x,y\rangle^2 = \sum_{i,j} x_i x_j y_i y_j$ and applying Fubini's theorem yields:
\[
\int_{\R^n}\int_{\R^n} f(x)g(y) \langle x,y\rangle^2 dx dy = \sum_{i,j=1}^n \left(\int_{\R^n} f(x) x_i x_j dx\right) \left(\int_{\R^n} g(y) y_i y_j dy\right).
\]
But, since $f, g \in \mathcal{F}(\conbodI)$, their level sets are isotropic, which implies that $\int_{\R^n} x_i x_j f(x) dx = 0$ whenever $i \neq j$. We deduce
\[
\int_{\R^n}\int_{\R^n} f(x)g(y) \langle x,y\rangle^2 dx dy \leq \sum_{i=1}^n\left(\int_{\R^n} \Omega(|z|^2) z_i^2 dz \right)^2.
\]

Finally, because the superlevel sets of $\Omega(|z|^2)$ are Euclidean balls, $\Omega(|z|^2)$ itself has isotropic superlevel sets. One can verify that \eqref{eq:iso_def} holds for this function:
\[
I = \frac{1}{n} \int_{\R^n} \Omega(|z|^2) |z|^2 dz= \int_{\R^n} \Omega(|z|^2) z_i^2 dz.
\]
Consequently,
\[
\sum_{i=1}^n\left(\int_{\R^n} \Omega(|z|^2) z_i^2 dz \right)^2=n I^2 = n^{-1}\left(\int_{\R^n} \Omega(|z|^2) |z|^2 dz \right)^2,
\]
completing the recovery of \eqref{eq:fun_ball}. \hfill \qedsymbol

\subsection{Sine Duality}
\label{sec:sine}

In this section, fix $n\geq 2$. Recall that
$$[x,y]=\begin{cases}
|x||y|\sqrt{1-\left\langle \frac{x}{|x|},\frac{y}{|y|} \right\rangle^2},& \text{ if }x,y \in \R^n\setminus\{o\},
\\
0,& \text{otherwise}.
\end{cases}
$$
Then, the sine polar of a convex body $K$ is given by
\[
K^{\diamond}=\left\{x \in \mathbb{R}^n:[x, y] \leq 1 \text { for all } y \in K\right\} = K^{\star, c}, \qquad c(x,y)=-[x,y],
\]
where we identify the sine polarity with the general duality from \eqref{eq:abstract_raised_polar}.  Note that even if $K$ contains the origin in its interior, the inclusion $K\subseteq K^{\diamond\diamond}$ obtained from Proposition~\ref{p:dualities} is generally strict. 

We elaborate. The following details are all from \cite{HLXY22}, although some of them follow from Proposition~\ref{p:dualities}. Let $\operatorname{Cyl} (A)$ denote the \textit{cylindrical hull} of a measurable set $A \subset \R^n$, which is the intersection of all origin-symmetric closed solid cylinders containing $A$. Set 
$$D_n\!=\!\{A\subset \R^n| \, A \text{ is bounded, measurable and not contained in a $1$-dimensional subspace}\}.$$
Then, for $A\in D_n,$ $A^{\diamond\diamond}=\operatorname{Cyl} (A)$ and $A^{\diamond}=\operatorname{Cyl} (A)^{\diamond}$. 

Moreover, the sine-polar sets $\mathcal{S}_n^s$, i.e., those subsets of $\R^n$ satisfying $A^{\diamond\diamond}=A$, are precisely those origin-symmetric convex sets obtained as the intersection of such cylinders. The description of this set was provided in \cite[Proposition 3.5]{HLXY22}. What is even more interesting is that $D_n^{\diamond\diamond}=\mathcal{S}_n^s\subset \conbode$. That is, even if the original set $A\in D_n$ had an empty interior, its double polar $A^{\diamond\diamond}$ is an origin-symmetric convex body. The condition that $A$ has dimension at least two is used to show $A^{\diamond\diamond}$ is bounded. Additionally, it was shown in \cite[Proposition 3.3]{HLXY22} that, for $K\in \mathcal{S}_n^s$,
\[
\vol(K^\diamond)\leq \vol(K^\circ).
\]
Since, for $A\in D_n$, $A\subseteq A^{\diamond\diamond}=K \in \mathcal{S}_n^s$, it immediately follows from \eqref{eq:BS} that
\begin{equation}
\label{eq:BS_sine}
\vol(A)\vol(A^\diamond) \leq \vol(A^{\diamond\diamond})\vol(A^{\diamond}) = \vol(A^{\diamond\diamond}) \vol((A^{\diamond\diamond})^\circ) \leq \vol(\B)^2,
\end{equation}
with equality if and only if $A$ is, up to removal of null sets, a centered ellipsoid ($n=2$) or a centered Euclidean ball ($n\geq 3$). This sine Blaschke-Santal\'o inequality first appeared in \cite[Theorem 5.2]{HLXY22}. We note that, in the proof above, while the Blaschke-Santal\'o inequality forces ellipsoids for the equality characterization, it was shown in the aforementioned work that, for $n\geq 3$, if a centered ellipsoid $E$ satisfies $E=E^{\diamond\diamond}$, then $E$ is a centered Euclidean ball.

A functional version was shown in \cite{HuLi25}. Let
\[
f^\diamond(x) = \inf_{y\in \R^n}\left[\frac{e^{-[x,y]}}{f(y)}\right]=f^{\circ, c}(x), \qquad c(x,y)=-[x,y].
\]
Then, it was shown in \cite[Proposition 2.2]{HuLi25} that $f^\diamond$ is log-concave, and in \cite[Lemma 3.1]{HuLi25} that if $f$ is even and integrable, then $\int_{\R^n}f^\diamond(x) dx \leq \int_{\R^n}f^\circ(x) dx$. Moreover, equality occurs for certain functions, e.g., a Gaussian. Therefore, Ball's functional Santal\'o \eqref{eq:ball_bS} for even, integrable functions immediately implies the sharp inequality
\[
\left(\int_{\R^n}f(x)dx\right)\left(\int_{\R^n}f^\diamond(x)dx\right) \leq (2\pi)^n.
\]
We now use the transference principle to upgrade this result to the form of Proposition~\ref{prop:Ball_Fradelizi_Meyer}.
\begin{theorem}
\label{thm:func_sine_BS}
Let $n \geq 2$. Consider the space $\R^n$ equipped with the Lebesgue measure and the cost function $c(x,y) = -[x,y]$. Define the function $F: \R_+\to \R_+$ by $F(r) = \vol(\B)r^{\frac{n}{2}}.$
Then, $\R^n$ satisfies the cost-Santal\'o inequality associated with $\left(\frac{1}{2}, c, F \right)$ with domain $(\mathcal{B}_b(\R^n), \mathcal{B}(\R^n))$.
\end{theorem}
Here, $\mathcal{B}_b(\R^n)\subset \mathcal{B}(\R^n)$ are all bounded, measurable sets. This theorem follows from Lemma~\ref{l:general_homogeneous_santalo} with $n=d$ and $\mu$ and $\nu$ both being the Lebesgue measure. Note that the condition on \eqref{eq:BS_sine} that the sets must have dimension at least 2 is not an obstacle, as this means they have volume zero. Corollary~\ref{cor:sine_fun} is the functional inequality implied by Theorem~\ref{thm:func_sine_BS}, after using approximation and the monotone convergence theorem (details omitted) to drop the requirement of bounded level sets inherited from $\mathcal{B}_b(\R^n)$.

\subsection{Weighted Euclidean Santal\'o inequalities}
Colesanti, Livshyts, Kolesnikov, and Rotem \cite{CKLR26} recently considered weighted versions of the functional Santal\'o inequality. A particular case of interest in their work involves measures of the form $e^{-V}$, where $V$ is a smooth, $p$-homogeneous and convex function. Using the transference principle, we can generalize an equivalence between conjectured results appearing in their work. To recover their specific setup, simply restrict to symmetric convex bodies, select $\Omega(r) = e^{-r}$, and set $g = f^\circ$. 

\begin{proposition}
\label{prop:CKLR_general}
    Let $V$ be a strictly convex, $p$-homogeneous, $C^2$ function on $\R^n$, $p>1$, satisfying $V(o)=0$. Then, the geometric inequality
    \begin{equation}
    \label{eq:CLKR_equiv}
        \vol(K)\vol\left(\nabla \mathcal{L}V(K^\circ)\right)^{p-1} \leq \vol\left(\left\{V\leq \frac{1}{p}\right\}\right)^p, \qquad \forall \,K \in \mathcal{C}^n,
    \end{equation}
    holds if and only if for every $f \in \mathcal{F}_n$ and all measurable functions $g:\R^n\to\R_+$ and $\Omega:\R_+\to\R_+$, where $\Omega$ is non-increasing, that satisfy the pointwise bound
    \[
        f(x)^{\frac{1}{p}}g(y)^{\frac{p-1}{p}} \leq \Omega\left(\frac{\langle x,y \rangle}{p}\right), \qquad \forall \,\,x, y \in \R^n \text{ with } \langle x,y \rangle>0,
    \]
    one has the functional inequality:
    \begin{equation}
    \label{eq:CLKR_func_general}
        \left(\int_{\R^n}f(x)dx\right)\left(\int_{\R^n}g(\nabla V(z))dz\right)^{p-1} \leq \left(\int_{\R^n}\Omega(V(z))dz\right)^p.
    \end{equation}
\end{proposition}
We omit the proof; it follows from Lemma~\ref{l:transfer} in the same manner as the other Euclidean examples. In \cite{CKLR26}, they proved and disproved \eqref{eq:CLKR_equiv} in various situations.

\section{Matrix space}
\label{sec:matrix}
This section is dedicated to our results concerning Santal\'o-type inequalities on the setting of matrix spaces, primarily Theorems~\ref{t:hiBS}, \ref{t:hiBS_2}, and \ref{t:main}.

\subsection{The Geometric Polarities}
We first continue the investigation of Section~\ref{sec:set_prop} of the polarity on sets in the matrix case. We note that the set inclusions in property (3) of Proposition~\ref{p:dualities} are strict in general. This is seen by taking $Q$ to be origin-symmetric and $K$ not origin-symmetric. Then $(K^{\star,Q})_{\circ,Q}$ is origin-symmetric, but $K$ is not. In the next proposition, we show that, when $L$ is star-shaped, we may exchange intersections over an entire star-shaped $L$ in $M_{n,m}(\R)$ with just intersections over its boundary $\partial L$ in the definition of the polarity.

\begin{proposition}
\label{p:boundary}
    Let $L\subset M_{n,m}(\R)$ be a compact, star-shaped set. Then,
    \[
    \bigcap_{A\in L}(AQ)^{\circ} = \bigcap_{Z\in \partial L}(ZQ)^{\circ}.
    \]
    In particular, $\|\cdot\|_{L_{\circ,Q}} = \sup_{A\in L}\|\cdot\|_{(AQ)^\circ}=\sup_{Z\in\partial L}\|\cdot\|_{(ZQ)^\circ}.$
\end{proposition}

\begin{proof}
    Clearly,
    \[
     \bigcap_{A\in L}(AQ)^{\circ} \subseteq \bigcap_{Z\in \partial L}(ZQ)^{\circ},
    \]
    so it remains to show the opposite set inclusion. Indeed, let $x \in \bigcap_{Z\in \partial L}(Z Q)^{\circ}$. Then, $x \in (ZQ)^{\circ}$ for all $Z\in \partial L$, i.e., by definition of polarity, $h_{ZQ}(x) \leq 1,\; \forall Z\in \partial L$. Since $L$ is star-shaped with respect to the origin, for every $A\in L,$ there exists $\lambda \in [0,1]$ and $Z\in \partial L$ such that $A=\lambda Z$. Consequently, we have 
    \[
    h_{AQ}(x) = h_{(\lambda Z)Q}(x) \leq \lambda \leq 1,
    \]
    i.e., $x \in (AQ)^\circ$ for all $ A\in L$. Therefore, $x \in \bigcap_{A\in L}(AQ)^{\circ}$.
\end{proof}

    While not strictly necessary for our investigations, it is interesting to note that the gauge $\|\cdot\|_{K^{\star,Q}}$ is actually a (pseudo)-operator norm when $Q$ contains the origin in its interior. Indeed, in this instance, we can write $h_Q = \|\cdot\|_{Q^\circ}$. Then, for every $A\in M_{n,m}(\R)$, denote its columns by $a_i$, in which case $A^tv = \left(\langle a_i,v\rangle\right)_{i=1}^m\in \R^m$. We deduce
\begin{align*}
    \|A\|_{K^{\star,Q}}&= \sup_{\|v\|_K \leq 1}\big\|\left(\langle a_i,v\rangle\right)_{i=1}^m\big\|_{Q^\circ}=\sup_{\|v\|_K \leq 1}\sup_{\mathcal{E}\in Q}\sum_{i=1}^m\varepsilon_i\langle a_i,v \rangle
    \\
    &=\sup_{\mathcal{E}\in Q}\sup_{\|v\|_K \leq 1}\left\langle\sum_{i=1}^m \varepsilon_i a_i,v\right\rangle = \sup_{\mathcal{E}\in Q}\left\|\sum_{i=1}^m \varepsilon_i a_i \right\|_{K^\circ},
\end{align*}
where $\mathcal{E}=(\varepsilon_1,\dots,\varepsilon_m)$.

    Finally, we would like to highlight a particular choice of $Q$ which has appeared implicitly in the literature. Fix $p\geq 1, q=\frac{p}{p-1}$ and choose, for $i=1,\dots,m$, $c_i>0$ and define the $n\times n$ matrix $T=\operatorname{diag}\left(c_i^\frac{1}{p}\right)$. Then,
    \[
    h_{TB_q^n}(y)=\left(\sum_{i=1}^mc_i|y_i|^p\right)^\frac{1}{p}, \quad y \in \R^m.
    \]
This observation yields
    \begin{equation}
    \label{eq:generalized_lp}
    h_{TB_q^n}(A^tv)=\left(\sum_{i=1}^mc_i|\langle a_i,v\rangle|^p\right)^\frac{1}{p}, \qquad A=(a_i)_{i=1}^m,\,\, a_i,v\in \R^n.
    \end{equation}
    In \cite{ABC26}, they considered, for $A\in M_{n,m}(\R)$ fixed with columns $a_i\in \mathbb{S}^{n-1}$, the generalized $\ell_p$-balls given as the unit ball of the norm $\|v\|=h_{TB_q^n}(A^tv)$. Such sets are also realized as $L_p$ projection bodies of polytopes.

\subsection{Geometric Inequalities} The following fact compares the size of $K^{\star,Q}$ and $K^\circ$. We emphasize that no assumptions about the center of mass are needed on any of the bodies.

\begin{proposition}\label{p:PPIinfty}
Let $n,m \in \N$. Then, for any $Q \in \conbodo[m]$, and any compact set $K\subset \R^n$ with positive volume such that $o\in \operatorname{int}\operatorname{conv}(K)$,
\[\Vol(K^{\star,Q}) \vol(K^\circ)^{-m} \leq \Vol((\B)^{\star,Q}) \vol(\B)^{-m}.\]

 For a fixed $Q$, there is equality when $K=E$ for some centered ellipsoid $E\subset\R^n$ up to removal of null sets. For a fixed $K$, there is equality when $Q=[0,1]^m$.
\end{proposition}

\begin{proof}
This inequality is trivial when $m=1$. Indeed, let $R(Q)=\max_{q\in Q}|q|$ be the outer-radius of $Q$. There exists $q_0\in Q$ where $|q_0|=R(Q)$, since $Q$ is compact. If $x\in K^{\star,Q}$, then $q_0x\in K^\circ$. Hence, $K^{\star,Q}\subseteq q_0^{-1}K^\circ$. Therefore, $\vol(K^{\star,Q})\leq R(Q)^{-n}\vol(K^\circ)$. But, $(\B)^{\star,Q}=R(Q)^{-1}\B$, and we conclude by taking volumes of this equality.

    For $m\geq 2$, the case when $K$ is convex body containing the origin in its interior was proven via a limiting argument in \cite[Corollary 4.5]{HLPRY25_2} from an $L^p$ framework. The case for general $K$ follows from Proposition~\ref{p:dualities}, item (6), by replacing $K$ with $\operatorname{conv} \left(K\right)$.

    The fact that $K=E$ for a centered ellipsoid $E$ (up to removal of null sets) yields equality is due to the affine invariance of the inequality. Finally, notice that $K^{\circ,[0,1]^m}=(K^\circ)^m$, and therefore, this choice of $Q$ gives equality, namely that both sides are identically $1$.
\end{proof}

As a consequence of Proposition~\ref{p:PPIinfty}, we obtain a direct proof of Theorem~\ref{t:hiBS_2}

\begin{proof}[Proof of Theorem~\ref{t:hiBS_2}]
Observe we can write Proposition~\ref{p:PPIinfty} as
\[
\Vol(K^{\star,Q})\vol(K)^{m}  \leq \Vol((\B)^{\star,Q}) \vol(\B)^{m}\left(\frac{\vol(K)\vol(K^\circ)}{\vol(\B)^2}\right)^{m},
\]
with equality when $K$ is a centered ellipsoid. We conclude by using \eqref{eq:BS}.
\end{proof}

In the next lemma, we prove Theorem~\ref{t:hiBS} by showing that it is equivalent to Theorem~\ref{t:hiBS_2}.

\begin{lemma}
\label{l:equiv}
    Let $L\subset M_{n,m}(\R)$ and $K\subset \R^n$ be compact sets of positive volume whose convex hulls contain the (appropriate) origin in their interiors. Then, 
    \[
    \vol[nm](K^{\star,Q})\vol[n](K)^m \leq \vol[nm](K^{\star,Q})\vol[n]((K^{\star,Q})_{\circ,Q})^m 
    \] 
    and 
    \[ 
    \vol[nm](L)\vol[n](L_{\circ,Q})^m \leq \vol[nm]((L_{\circ,Q})^{\star,Q})\vol[n](L_{\circ,Q})^m,
    \]
    with equality only for invariant sets (up to sets of Lebesgue measure zero). 
    
    In particular, the following two statements are equivalent:
    \begin{enumerate}
        \item if $K$ or $K^\circ$, has center of mass at the origin, then
    \[
    \vol[nm](K^{\star,Q})\vol[n](K)^m \leq \vol[nm]((\B)^{\star,Q})\vol[n](\B)^m 
    \]
    with equality if and only if $K=E$ for a centered ellipsoid $E\subset\R^n$ up to removal of null sets; 
    \item  if $L_{\circ,Q}$ or $(L_{\circ,Q})^\circ$ has center of mass at the origin, then
    \[
    \vol[nm](L)\vol[n](L_{\circ,Q})^m \leq \vol[nm]((\B)^{\star,Q})\vol[n](\B)^m
    \]
    with equality if and only if $L=E^{\star,Q}$ for a centered ellipsoid $E\subset \R^n$ up to removal of null sets.
    \end{enumerate}
\end{lemma}

\begin{proof}
First, if $K$ is not convex, then use that $K\subset K^{\circ\circ}= \operatorname{conv}(K)$ and Proposition~\ref{p:dualities} item (6) to reduce to convex $K$. Then, the claim follows from again using Proposition~\ref{p:dualities} by taking the volume throughout the inclusion in item (3) and utilizing items (4) and (5) for the equality characterization.
\end{proof}

Finally, we have the following corollary of Theorem~\ref{t:hiBS_2}, which maximizes over $Q$.
\begin{proposition}
    Fix $L\in\conbodo[n]$ and let $Q\in\conbodo[m]$ be so that $\vol[m](Q)=\omega_m$. Suppose either $Q$ or $Q^\circ$ has center of mass at the origin. Then,
    \[
    \vol[nm](L^{\star,Q}) \leq \vol[nm](L^{\star,B_2^m}).
    \]
\end{proposition}
\begin{proof}
The assertion follows from Theorem~\ref{t:hiBS_2} by interchanging the roles of the two dimensions. Indeed, for every $A\in M_{n,m}(\mathbb R))$, we have
\begin{align*}
A\in L^{\star,Q}&\Longleftrightarrow AQ\subset L^\circ\\
&\Longleftrightarrow \langle Aq,\xi\rangle\leq 1 \quad\text{for every }q\in Q,\ \xi\in L,\\
&\Longleftrightarrow \langle A^t\xi,q\rangle\leq 1 \quad\text{for every }\xi\in L,\ q\in Q,\\
&\Longleftrightarrow A^tL\subset Q^\circ\\
&\Longleftrightarrow A^t\in Q^{\star,L}.
\end{align*}

Consequently, $(L^{\star,Q})^t=Q^{\star,L}.$ Since transposition is an isometry between $M_{n,m}(\mathbb R)$ and $M_{m,n}(\mathbb R)$, it follows that
\[
\vol[nm](L^{\star,Q})=\vol[mn](Q^{\star,L}).
\]

We now apply Theorem~\ref{t:hiBS_2}, with $n$ and $m$ interchanged, to the body $Q\subset\mathbb R^m$ and the auxiliary body $L\subset\mathbb R^n$. We deduce
\[
\vol[mn](Q^{\star,L})\vol[m](Q)^n
\leq\vol[m](B_2^m)^n\vol[mn]\bigl((B_2^m)^{\star,L}\bigr).
\]
Using $\vol[m](Q)=\omega_m$, we obtain $\vol[mn](Q^{\star,L})\leq\vol[mn]\bigl((B_2^m)^{\star,L}\bigr).$ Finally, applying the transpose relation once more gives $\vol[nm](L^{\star,Q})\leq\vol[nm](L^{\star,B_2^m}),$ as desired.
\end{proof}

\subsection{Functional Polarity}
We now move onto considering the associated functional Satanl\'o-type inequality. 

\begin{proof}[Proof of Theorem~\ref{t:main}]
    We apply Lemma~\ref{l:general_homogeneous_santalo} with $\nu = \vol$ and $d=nm$, so that $\R^d = M_{n,m}(\R)$, and we choose $\mu = \Vol$. The measure $\mu$ is homogeneous of degree $\alpha = nm$. We choose the weight $\lambda = \frac{m}{m+1}$. 

    We define the continuous cost function $c: \R^n \times M_{n,m}(\R) \to \R$ by
    \[
        c(v,A) = -h_Q(A^t v).
    \]
    Because $Q \in \conbodo[m]$, $-c(v,A) = h_Q(A^t v) \geq 0$, satisfying Property (2) of Lemma~\ref{l:general_homogeneous_santalo}. Furthermore, $c$ is positively 1-homogeneous in the second variable $A$.

    For Property (3), suppose a bounded set $K \subset \R^n$ fails the non-degeneracy condition. By the compactness of the unit sphere in $M_{n,m}(\R)$, there must exist a non-zero matrix $A$ such that $h_Q(A^tv) = 0$ for all $v \in K$. Now if $Q$ contains the origin in its \textit{interior}, this implies $A^t v = 0$ for all $v \in K$. Thus, $K$ is contained in $\ker(A^t)$, which is a proper subspace of $\R^n$ and therefore has volume zero.

    In the case $Q$ has the origin on its boundary, we have to work a bit harder. Using the same non-zero matrix $A$, we have from the definition of the support function that $\langle y, A^tv \rangle = \langle Ay, v \rangle \leq 0$ for all $y \in Q$ and $v \in K$. Because $A \neq o$ and $Q$ is full-dimensional, there exists some $y \in Q$ such that $u = Ay \neq o$. This forces $\langle u, v \rangle \leq 0$ for all $v \in K$, meaning $K$ is contained in a half-space passing through the origin. However, $K \in \mathcal{C}^n$ has positive volume and its center of mass at the origin, which means it cannot be contained in a half-space. Thus, $K$ must have volume zero.

    By Theorem~\ref{t:hiBS_2}, every $K \in \mathcal{C}^n$ satisfies the geometric condition from \eqref{eq:gen_geom_hypo}:
    \[
        \vol(K)^\frac{m}{m+1} \Vol(K^{\star,Q})^\frac{1}{m+1} \leq \vol(\B)^\frac{m}{m+1} \Vol((\B)^{\star,Q})^\frac{1}{m+1} := C_{\nu,\mu}.
    \]

    We may assume the integrals are finite and non-zero. Let $F(v) = f(v)^{\frac{1}{m}}$. Applying Lemma~\ref{l:general_homogeneous_santalo}, we obtain
    \begin{equation}
    \label{eq:matrix_1D_bound}
        \left(\int_{\R^n} F(v) dv \right)^{\frac{m}{m+1}} \left(\int_{M_{n,m}(\R)} g(A) dA \right)^{\frac{1}{m+1}} \leq C_{\nu,\mu} \frac{nm}{m+1} \int_0^\infty \Omega(r) r^{\frac{nm}{m+1}-1} dr.
    \end{equation}

    To achieve the final formula, we evaluate the integral on the right-hand side using polar coordinate integration. In $\R^n$, substituting $r = t^{\frac{m+1}{m}}$ yields:
    \[
        \int_{\R^n} \!\!\!\Omega\left(|v|^{\frac{m+1}{m}}\right) dv = n\vol(\B)\int_0^\infty \!\!\Omega\left(t^{\frac{m+1}{m}}\right) t^{n-1} dt = \frac{nm}{m+1}\vol(\B) \int_0^\infty\!\!\! \Omega(r) r^{\frac{nm}{m+1}-1} dr.
    \]
    Similarly, integrating in $M_{n,m}(\R)$ and substituting $r = t^{m+1}$ yields:
    \begin{align*}
        \int_{M_{n,m}(\R)} \Omega\left(\|A\|_{(\B)^{\star,Q}}^{m+1}\right) dA &= nm\Vol((\B)^{\star,Q})\int_0^\infty \Omega\left(t^{m+1}\right) t^{nm-1} dt \\
        &= \frac{nm}{m+1}\Vol((\B)^{\star,Q}) \int_0^\infty \Omega(r) r^{\frac{nm}{m+1}-1} dr.
    \end{align*}
    Raising \eqref{eq:matrix_1D_bound} to the $(m+1)$ power therefore yields
    \begin{align*}
        &\left(\int_{\R^n} f(v)^{\frac{1}{m}} dv \right)^m \left(\int_{M_{n,m}(\R)} g(A) dA \right) \leq \left( C_{\nu,\mu}\frac{nm}{m+1} \int_0^\infty \Omega(r) r^{\frac{nm}{m+1}-1} dr \right)^{m+1} \\
        &= \left( \vol(\B)^m \Vol((\B)^{\star,Q}) \right) \left( \frac{\int_{\R^n} \Omega\left(|v|^{\frac{m+1}{m}}\right) dv}{\vol(\B)} \right)^m \left( \frac{\int_{M_{n,m}(\R)} \Omega\left(\|A\|_{(\B)^{\star,Q}}^{m+1}\right) dA}{\Vol((\B)^{\star,Q})} \right) \\
        &= \left(\int_{\R^n} \Omega\left(|v|^{\frac{m+1}{m}}\right) dv\right)^m \left(\int_{M_{n,m}(\R)} \Omega\left(\|A\|_{(\B)^{\star,Q}}^{m+1}\right) dA\right).
    \end{align*}
    The equality conditions follow immediately. We conclude.
\end{proof}

We take a moment to list several corollaries of Theorem~\ref{t:main}, highlighting relations between \eqref{eq:lower_Q_polar} and the classical polarity \eqref{eq:classical_polarity}. For $A\in M_{n,m}(\R)$, let $a_i$ be its column vectors, $i=1,\dots,m$, so that for $v\in\R^n$, $A^tv=(\langle a_1,v\rangle,\ldots,\langle a_m,v\rangle)\in\R^m$. Next, make the choice of $Q=B_\infty^m$, so that $h_Q(\cdot)=\|\cdot\|_{B_1^m}$. In fact, we will consider the more general bodies from \eqref{eq:generalized_lp}, choosing $Q$ a linear image of $B_\infty^m$, so that
\[
h_Q(v)=h_{TB_\infty^m}(v)=\sum_{i=1}^mc_i|v_i|, \qquad T=\operatorname{diag}(c_i),\, c_i>0.
\]
Then, we see that $$e^{-h_Q(A^tv)}=\prod_{i=1}^me^{-c_i|\langle a_i,v \rangle|}.$$

With this observation in mind, we consider the following immediate corollary to Theorem~\ref{t:main}, when $g$ is a product of $m$ functions on $\R^n$ and $Q=TB_\infty^m$.
\begin{corollary}
\label{c:BS_multiple}
    Fix $m\in \N$ and a measurable, non-increasing function $\Omega:\R_+\to\R_+$. For $i=1,\dots,m$, let $c_i>0$. Suppose $f_1,\dots,f_{m+1}:\R^n\to\R_+$ are measurable, even functions that satisfy
    $$\left(\prod_{i=1}^{m}f_i(a_i)\right)\cdot f_{m+1}(v) \leq \Omega^{m+1}\left(\sum_{i=1}^mc_i|\langle a_i,v\rangle|\right) \quad \text{for all } a_1,\dots,a_m,v\in \R^n.$$ 
    Then, we have
    \begin{equation}
  \begin{split}
    &\prod_{i=1}^m\left(\int_{\R^n}f_i(x)dx\right)\Bigg(\int_{\R^n} f_{m+1}(v)^\frac{1}{m}dv\Bigg)^m
    \\
    &\leq \left(\int_{\R^n}\cdots\int_{\R^n}\Omega\left(\left(\max_{|\xi| \leq 1} \sum_{i=1}^m c_i|\langle a_i, \xi \rangle|\right)^{m+1}\right)da_1\cdots da_m\right)\left(\int_{\R^n}\Omega\left(|v|^\frac{m+1}{m}\right)dv\right)^{m}.
\end{split}
    \end{equation}
    When $\Omega$ is log-concave and integrable over $\R_+$, equality holds when there exists $C>0$ such that $$\prod_{i=1}^mf_i(a_i) = C\cdot \Omega\left(\left(\max_{|\xi| \leq 1} \sum_{i=1}^m c_i|\langle a_i, \xi \rangle|\right)^{m+1}\right) \text{ and }f_{m+1} = \frac{1}{C} \cdot \Omega\left( |\cdot|^{\frac{m+1}{m}}\right)^m.$$
\end{corollary}
\noindent Here, we used the fact that $(B_2^n)^{\star, B_\infty^m}$ is the unit ball of the operator norm from $\ell_\infty^m$ to $\ell_2^n$. 

In turn, let $f_1,\dots,f_m:\R^n\to \R_+$ be even, measurable functions, and define $\prod_{i=1}^mf_i:M_{n,m}(\R)\to \R$ by $\left(\prod_{i=1}^mf_i\right)(A)=\prod_{i=1}^mf_i(a_i)$. Then, 
\begin{align*}
    \left(\prod_{i=1}^mf_i\right)^{\circ,{TB_\infty^m}}(v)&= \inf_{A\in M_{n,m}(\R)}\left[\frac{\prod_{i=1}^me^{-c_i|\langle a_i,v \rangle|}}{\prod_{i=1}^mf_i(a_i)}\right]
    \\
    &=  \inf_{\substack{a_i\in \R^n, 
    \\
    i=1,\dots,m}} \prod_{i=1}^m\left(\frac{e^{-|\langle a_i,v \rangle|}}{f_i(a_i)^\frac{1}{c_i}}\right)^{c_i}
    = \prod_{i=1}^m\left(\left(f_i^\frac{1}{c_i}\right)^\circ(v)\right)^{c_i}.
\end{align*}
Therefore, by taking $\Omega(t)=e^{-t/(m+1)}$, we may set
\[
f_{m+1}(v)
=
\prod_{i=1}^m
\left(
\left(f_i^{1/c_i}\right)^\circ(v)
\right)^{c_i}
\]
in Corollary~\ref{c:BS_multiple}.

Similarly, we can replace $Q=B_\infty^m$ with $Q=[0,1]^m$, in which case $h_Q(y)=\sum_{i=1}^m(y_i)_+$ and
\begin{align*}
    \left(\prod_{i=1}^mf_i\right)^{\circ,{[0,1]^m}}\!\!\!\!(v)&= \inf_{A\in M_{n,m}(\R)}\left[\frac{\prod_{i=1}^me^{-\langle a_i,v \rangle_+}}{\prod_{i=1}^mf_i(a_i)}\right] 
    =  \inf_{\substack{a_i\in \R^n, 
    \\
    i=1,\dots,m}} \prod_{i=1}^m\frac{e^{-\langle a_i,v \rangle_+}}{f_i(a_i)}
    = \prod_{i=1}^mf_i^\circ(v),
\end{align*}
for any measurable, even functions $f_i:\R^n\to \R_+$. The corresponding result is as follows; this should be seen as the ``asymmetric'' version of the above.
\begin{corollary}
\label{c:BS_multiple_aysm}
    Fix $m\in \N$ and a measurable, non-increasing function $\Omega:\R_+\to\R_+$. Let $f_1,\dots,f_{m+1}:\R^n\to\R_+$ be measurable, even functions that satisfy
    $$\left(\prod_{i=1}^{m}f_i(a_i)\right)\cdot f_{m+1}(v) \leq \Omega^{m+1}\left(\sum_{i=1}^m\langle a_i,v\rangle_+\right) \quad \text{for all } a_1,\dots,a_m,v\in \R^n.$$ 
    Then, we have
    \begin{equation}
  \begin{split}
    &\prod_{i=1}^m\left(\int_{\R^n}f_i(x)dx\right)\Bigg(\int_{\R^n} f_{m+1}(v)^\frac{1}{m}dv\Bigg)^m
    \\
    &\leq \left(\int_{\R^n}\cdots\int_{\R^n}\Omega\left(\left(\max_{|\xi| \leq 1} \sum_{i=1}^m \langle a_i, \xi \rangle_+\right)^{m+1}\right)da_1\cdots da_m\right)\left(\int_{\R^n}\Omega\left(|v|^\frac{m+1}{m}\right)dv\right)^{m}.
\end{split}
    \end{equation}
    When $\Omega$ is log-concave and integrable over $\R_+$, there is equality when there exists $C>0$ such that $$\prod_{i=1}^mf_i(a_i) = C\cdot \Omega\left(\left(\max_{|\xi| \leq 1} \sum_{i=1}^m \langle a_i, \xi \rangle_+\right)^{m+1}\right) \text{ and }f_{m+1} = \frac{1}{C} \cdot \Omega\left( |\cdot|^{\frac{m+1}{m}}\right)^m.$$
\end{corollary}
\noindent In Corollaries~\ref{c:BS_multiple} and \ref{c:BS_multiple_aysm}, we can replace the assumption that our functions are even by requiring $f_{m+1}\in \mathcal{F}_n$.

There have been other proofs of the Blaschke-Santal\'o inequality, which could be adapted to the matrix setting. For example, if one mimics Ball's proof of \eqref{eq:BS_fun} in the case of even functions, one arrives at
\begin{align*}\left(\frac{\omega_n}{\omega_{nm}^\frac{1}{m}}\int_{0}^\infty \!\!\!\!\vol[nm](\{f > \beta\})^\frac{1}{m}d\beta\right)\left(\int_{\R^{n}}\!\!\!\!f^{\circ,Q}(v)dv\right)
\leq  \left(\int_{M_{n,m}(\R)}\!\!\!\!\!\!\!\!\!\!\!\!\!\!\!\!e^{-\frac{\|A\|^2_{(\B)^{\star, Q}}}{2}}dA\right)^\frac{1}{m}\left(\int_{\R^n}\!\!\!\!\!e^{-\frac{|v|^2}{2}}dv\right),
\end{align*}
and one can consider a variant for the other polarity as well. These inequalities are sharp but less interesting.

\section{Results from Enlargement Inequalities}
\label{sec:iso}

Let $(X,d,\mu)$ be a Polish probability space.  For a Borel set $A\subset X$ and $s\ge0$, set
\[
A_s:=\{x\in X:\ d(x,A)\le s\} \quad \text{and} \quad A_{<s}:=\{x\in X:\ d(x,A)< s\}.
\]

\begin{definition}
\label{d:enlargement}
Consider a family of functions $(R_s)_{s\geq 0}$, where each $R_s:[0,1]\to[0,1]$ is a function that satisfies
\begin{equation}
R_0(v)=v, \   R_s(v)\ge v,
\end{equation}
and the family has the property that $s\mapsto R_s(v)$ is non-decreasing and continuous for every $v \in [0,1]$. 

We say the family $(R_s)_{s\geq 0}$ satisfies an enlargement inequality for $(X,d,\mu)$ if
\begin{equation}
\label{eq:enlarge_ineq}
\mu(A_s)\ge R_s(\mu(A)).
\end{equation}
We call each $R_s$ an enlargement function. Moreover, we associate with $(R_s)_{s\geq 0}$ the function $\Theta: \R_+ \to [0,1/2]$ given by
\begin{equation}
\label{eq:theta_def}
\Theta(s):=\sup_{0\le v\le1}\sqrt{v\bigl(1-R_s(v)\bigr)}.
\end{equation}
Finally, the function $\Theta$ is non-increasing in $s$.
\end{definition}

Throughout this section, we need the following special case of Lemma~\ref{l:transfer} and Definition~\ref{def:cost_sant} suited to this setting. 
\begin{lemma}[Metric form of the Transference Principle]
\label{l:transfer_distance}
Let $(X,d,\mu)$ be a Polish probability space. Let $G:\R_+\to\R_+$ be a continuous, non-increasing function decaying to zero at infinity and let $\mathcal A,\mathcal D\subset \mathcal{B}(X)$. Then, the following are equivalent:
\begin{enumerate}
    \item For all
$A\in\mathcal A$ and $B\in\mathcal D$, one has
\[
    \mu(A)^{\frac12}\mu(B)^{\frac12}
    \leq
    G(d(A,B));
\]
\item For every pair of measurable functions $f \in \mathcal{F}(\mathcal{A})$ and $ g\in \mathcal{F}\left(\mathcal{D}\right)$ and measurable, non-decreasing function $W \colon \R_+  \to \R_+$ that satisfy \begin{align*}
f(x)^\frac{1}{2}g(y)^\frac{1}{2} &\leq W(d(x,y)) \quad \text{ for all } 
\quad(x,y)\in X\times  X,
\end{align*}
one has
\[
\left(\int_{X} f(x) d\mu(x) \right)^{\frac{1}{2}} \left(\int_{X} g(y) d\mu(y) \right)^{\frac{1}{2}}\leq \int_{\R_+} W(r) d(-G)(r).
\]
\end{enumerate}
We say that the space $(X,d,\mu)$ satisfies the metric-Santal\'o inequality induced by $G$ on the domain $(\mathcal A,\mathcal D)$.
\end{lemma}
\begin{proof}
Define $F:(-\infty,0]\to\R_+$ by $F(t)=G(-t).$ Since $G$ is continuous, non-increasing, and decays to zero at infinity, the
function $F$ is continuous, non-decreasing, and satisfies $F((-\infty)^+)=0$.
Then, the claim follows from Lemma~\ref{l:transfer} with $m=2$, $X_1=X_2=X$,
$\lambda_1=\lambda_2=\frac12$, and $c(x,y)=d(x,y)$. 

Indeed, for $A\in\mathcal A$ and $B\in\mathcal D$,
\[
F\left(\sup_{x\in A,y\in B}-c(x,y)\right)=F\left(-\inf_{x\in A,y\in B}d(x,y)\right)
=
F\left(-d(A,B)\right)
=
G(d(A,B)),
\]
and so the first condition of the present lemma is of the form of the first condition of Lemma~\ref{l:transfer}.

Which means we have the second condition of Lemma~\ref{l:transfer}; we now rewrite this in the form listed in the present lemma. Let $\Omega(t)=W(-t)$, which is non-increasing on $(-\infty,0]$. Then, in the functional formulation, replacing $t\leq0$
by $r=-t\geq0$ gives
\[
    \int_{(-\infty,0]}\Omega(t)\,dF(t)
    =
    \int_{\R_+}W(r)\,d(-G)(r).
\]
This is exactly the stated metric form.
\end{proof}

We first show that such enlargement profiles imply geometric Santal\'o inequalities.

\begin{lemma}
\label{l:general_enlargement_santalo}
Let $(X,d,\mu)$ be a Polish probability space with enlargement functions
$(R_s)_{s\geq 0}$ and associated function $\Theta$. 

Assume $A,B\in \mathcal{B}(X)$ are non-empty sets with
$0<\mu(A)\mu(B)\leq 1$.

Then,
\begin{equation}
\label{eq:enlargement_general}
    \mu(A)^{\frac12}\mu(B)^{\frac12}
    \leq
    2\Theta(d(A,B)).
\end{equation}
Moreover, if $A$ and $B$ are essentially disjoint, that is,
$\mu(A\cap B)=0$, then one has the sharp bound
\begin{equation}
\label{eq:enlargement_sharp}
    \mu(A)^{\frac12}\mu(B)^{\frac12}
    \leq
    \Theta(d(A,B)).
\end{equation}

Furthermore, assume that, for each $s>0$, the supremum defining $\Theta$ in
\eqref{eq:theta_def} is uniquely achieved at some $v_s\in(0,1)$ and that
$R_s(v_s)<1$. Then equality holds in \eqref{eq:enlargement_sharp} with
$0<\mu(A),\mu(B)<1$ if and only if the following conditions hold:
\begin{enumerate}
    \item If $d(A,B)=0$, then $\mu(A)=\mu(B)=\frac12$ and $B=X\setminus A$ up to a null set.

    \item If $d(A,B)>0$, then
    \begin{enumerate}
        \item $\mu(A)=v_{d(A,B)}$ and $\mu(B)=1-R_{d(A,B)}(v_{d(A,B)})$;
        \item
        \[
            \mu(A_{<d(A,B)})=R_{d(A,B)}(\mu(A));
        \]
        \item
        \[
            \mu\bigl(B\Delta(X\setminus A_{<d(A,B)})\bigr)=0.
        \]
    \end{enumerate}
\end{enumerate}
\end{lemma}

\begin{proof}
Let
\[
\delta=d(A,B)
\]
and let $\sigma \in (0,\delta]$. Then,
\[
    B\subseteq X\setminus A_{<\sigma}.
\]
Moreover, for every $0<t<\sigma$, one has $A_t\subseteq A_{<\sigma}$, and hence,
by the enlargement inequality,
\[
    \mu(A_{<\sigma})
    \geq
    \mu(A_t)
    \geq
    R_t(\mu(A)).
\]
Letting $t\uparrow\sigma$ and using the continuity of
$t\mapsto R_t(\mu(A))$, we obtain
\[
    \mu(A_{<\sigma})\geq R_\sigma(\mu(A)).
\]
Therefore,
\begin{equation}
\label{eq:key_theta_chain}
\begin{split}
    \mu(A)^{\frac12}\mu(B)^{\frac12}
    &\leq
    \sqrt{\mu(A)\bigl(1-\mu(A_{<\sigma})\bigr)}
    \leq
    \sqrt{\mu(A)\bigl(1-R_\sigma(\mu(A))\bigr)}
    \leq
    \Theta(\sigma).
\end{split}
\end{equation}

We now prove the two inequalities. 

If $\delta>0$, then we may take $\sigma=\delta$ in
\eqref{eq:key_theta_chain}. This gives
\[
    \mu(A)^{\frac12}\mu(B)^{\frac12}\leq \Theta(\delta) \leq 2\Theta(\delta),
\]
showing both \eqref{eq:enlargement_sharp} (since $\delta>0$ implies $A\cap B= \emptyset$) and \eqref{eq:enlargement_general} hold in
this case.

It remains only to consider $\delta=0$. For arbitrary Borel sets, the trivial bound gives
\[
    \mu(A)^{\frac12}\mu(B)^{\frac12}\leq 1.
\]
Since $R_0(v)=v$, we have
\[
    \Theta(0)
    =
    \sup_{0\leq v\leq 1}\sqrt{v(1-v)}
    =
    \frac12.
\]
Hence
\[
    \mu(A)^{\frac12}\mu(B)^{\frac12}
    \leq
    2\Theta(0)
    =
    2\Theta(\delta),
\]
which proves \eqref{eq:enlargement_general}. If, in addition, $A$ and $B$ are
essentially disjoint, then
\[
    \mu(A)+\mu(B)=\mu(A\cup B)\leq 1,
\]
and the inequality for arithmetic and geometric means gives
\begin{equation}
\label{eq:using_am_gm}
    \mu(A)^{\frac12}\mu(B)^{\frac12}
    \leq \frac{\mu(A)+\mu(B)}{2} \leq
    \frac12
    =
    \Theta(0)
    =
    \Theta(\delta).
\end{equation}
This proves \eqref{eq:enlargement_sharp}.

It remains to characterize equality in \eqref{eq:enlargement_sharp}. Suppose first
that $\delta=0$. In the preceding argument, equality holds if and only if equality holds \eqref{eq:using_am_gm}, which yields
\[
    \mu(A)=\mu(B)=\frac12 \qquad \text{and} \qquad
    \mu(A\cup B)=1.
\]
Since $\mu(A\cap B)=0$, this is equivalent to
\[
    B=X\setminus A
\]
up to a null set.

Now suppose that $\delta>0$. Equality in \eqref{eq:enlargement_sharp} is precisely equality
in the chain \eqref{eq:key_theta_chain} with $\sigma=\delta$. By the uniqueness assumption for the supremum defining $\Theta$, equality in the last step of
\eqref{eq:key_theta_chain} forces
\[
    \mu(A)=v_\delta.
\]
Equality in the middle step forces
\[
    \mu(A_{<\delta})=R_\delta(\mu(A)),
\]
and equality in the first step forces
\[
    \mu(B)=1-\mu(A_{<\delta}).
\]
Since $B\subseteq X\setminus A_{<\delta}$, this is equivalent to
\[
    \mu\bigl(B\Delta(X\setminus A_{<\delta})\bigr)=0.
\]
Thus
\[
    \mu(B)=1-R_\delta(v_\delta),
\]
and the stated conditions are necessary. Conversely, these conditions give
\[
    \mu(A)^{\frac12}\mu(B)^{\frac12}
    =
    \sqrt{v_\delta(1-R_\delta(v_\delta))}
    =
    \Theta(\delta),
\]
so they are also sufficient. This completes the proof.
\end{proof}

The above lemma shows that, in the setting of a Polish probability space $(X,d,\mu)$, the ``dual'' of a measurable set $A$ with $0<\mu(A)<1$ is nothing but $X\setminus A_{<\delta}$ for any $\delta$ such that $\mu(A_{<\delta})<1$. We can now deduce a metric-Santal\'o inequality for such spaces as a direct consequence of Lemma~\ref{l:transfer}.

\begin{theorem}
\label{thm:general_enlargement_theorem}
Let $(X,d,\mu)$ be a Polish probability space with enlargement functions
$(R_s)_{s\geq 0}$ and associated function $\Theta$. Suppose that $\Theta$ is continuous and decays to zero at infinity. Then:
\begin{enumerate}
    \item The space $(X,\mu)$ satisfies the metric-Santal\'o inequality induced by $2\cdot\Theta$ on the domain $(\mathcal{B}(X),\mathcal{B}(X))$.

    \item Let $\mathcal A,\mathcal D\subset \mathcal{B}(X)$ be sub-collections
    satisfying $\mu(A\cap D)=0$ for every $A\in\mathcal A$ and $D\in\mathcal D$. Then $(X,\mu)$ satisfies
    the sharper metric-Santal\'o inequality induced by $\Theta$ on the domain $(\mathcal A,\mathcal D)$.
\end{enumerate}
\end{theorem}

\begin{proof}
Disregarding trivial null cases, Lemma~\ref{l:general_enlargement_santalo} gives, with $G=2\Theta$,
\[
    \mu(A)^{\frac12}\mu(B)^{\frac12}
    \leq
    2\Theta(d(A,B)).
\]
Thus condition (1) of Lemma~\ref{l:transfer_distance} is satisfied, proving the first
claim.

For the second claim, let $A\in\mathcal A$ and $B\in\mathcal D$. By assumption,
$\mu(A\cap B)=0$. Hence Lemma~\ref{l:general_enlargement_santalo} gives, with $G=\Theta$,
\[
    \mu(A)^{\frac12}\mu(B)^{\frac12}
    \leq
    \Theta(d(A,B)).
\]
This verifies condition (1) of Lemma~\ref{l:transfer_distance} on the domain
$(\mathcal A,\mathcal D)$ and proves the sharper claim.
\end{proof}

\subsection{Isoperimetry}
\label{sec:iso_gen}
The preceding framework includes, in particular, enlargement inequalities obtained by integrating an isoperimetric profile. Recall that, for a Borel set $A\subset X$, its outer Minkowski boundary measure is defined by
\begin{equation}
\label{eq boundary measure}
    \mu^+(A):=\liminf_{h\to0^+}
    \frac{\mu(A_h)-\mu(A)}{h}.
\end{equation}
A continuous, strictly positive function $I:(0,1)\to(0,\infty)$ is called an
isoperimetric function for $(X,d,\mu)$ if
\begin{equation}
\label{eq:extremal_iso}
    \mu^+(A)\geq I(\mu(A))
\end{equation}
for every Borel set $A\subset X$ such that $0<\mu(A)<1$. We call any set obtaining equality in \eqref{eq:extremal_iso} an extremal set.

The following standard lemma, see, for example, \cite{BH97_2}, shows that isoperimetry and enlargement are, in some sense, equivalent.

\begin{lemma}
\label{l:isop_to_enl}
Let $(X,d,\mu)$ be a Polish probability space. Suppose it admits a continuous isoperimetric function $I:(0,1)\to(0,\infty)$. Define
\begin{equation}
\label{eq:phi_def}
\Phi(v):=
\begin{cases}
    \int_{\frac{1}{2}}^v\frac{dr}{I(r)},
    & v\in(\frac{1}{2},1),
    \\
    0, & v=\frac{1}{2},
    \\
    -\int_{v}^\frac{1}{2}\frac{dr}{I(r)}, & v \in (0,\frac{1}{2}),
\end{cases}
\end{equation}
and set $a:=\inf \Phi((0,1))$ and $b:=\sup \Phi((0,1)),$ where $-\infty \leq a \leq 0 <b\leq \infty$. Define the distribution function $\Psi:\mathbb R\to[0,1]$ by
\begin{equation}
\label{eq:Psi_def}
\Psi(t):=
\begin{cases}
    0, & t\leq a,\\
    \Phi^{-1}(t), & t\in(a,b),\\
    1, & t\geq b.
\end{cases}
\end{equation}
Then $\Psi$ is continuous and non-decreasing with $\Psi(0)=\frac{1}{2}$. For each $s\geq0$, define
$R_s^I:[0,1]\to[0,1]$ by
\begin{equation}
\label{eq:canonical_enlargement}
    R_s^I(0):=0,
    \qquad
    R_s^I(1):=1,
\quad \text{and }
    R_s^I(v):=\Psi\bigl(\Phi(v)+s\bigr),
    \qquad v\in(0,1).
\end{equation}
Then $R_s^I$ is an enlargement function for $(X,d,\mu)$ in the
sense of Definition~\ref{d:enlargement}. Conversely, if $(X,d,\mu)$ has $R_s^I$ of the above form as an enlargement function, then it satisfies isoperimetry with isoperimetric function $I$.

Moreover, if $A\subset X$ is a Borel set with $0<\mu(A)<1$ and, for some
$s>0$ such that $R_s^I(\mu(A))<1$, equality holds in
\eqref{eq:enlarge_ineq}, then
\begin{equation}
\label{eq:enlargement_equality_path}
    \mu(A_t)=R_t^I(\mu(A)),
    \qquad 0\leq t\leq s.
\end{equation}
In particular, $\mu(A_0)=\mu(A)$ and $A$ is an extremal set for the
isoperimetric inequality \eqref{eq:extremal_iso}.

\end{lemma}

\begin{proof}[Proof Sketch]
We sketch the proof. Fix a Borel set $A\subset X$, and set
\[
    u(s):=\mu(A_s),
    \qquad s\geq0.
\]
Since $(A_s)_h\subseteq A_{s+h},$ one has
\[
    \liminf_{h\to0^+}
    \frac{u(s+h)-u(s)}{h}
    \geq \mu^+(A_s)
    \geq I(u(s))
\]
whenever $u(s)\in(0,1)$. It can be shown that this implies
\[
    \Phi(u(s))\geq \Phi(u(0))+s
\]
as long as $u(s)<1$. Since $u(0)=\mu(A_0)\geq\mu(A)$, the monotonicity of
$\Phi$ gives
\[
    \Phi(u(s))\geq \Phi(\mu(A))+s.
\]
Applying $\Psi$, we obtain
\[
    \mu(A_s)=u(s)
    \geq \Psi\bigl(\Phi(\mu(A))+s\bigr)
    =R_s^I(\mu(A)).
\]
The cases $\mu(A)\in\{0,1\}$ are immediate. The remaining properties in
Definition~\ref{d:enlargement} follow directly from the definition of
$R_s^I$. The converse direction holds by differentiation.

Finally, suppose that, for some $s>0$ such that
$R_s^I(\mu(A))<1$, equality holds in \eqref{eq:enlarge_ineq}. Set
$v:=\mu(A)$ and $u(t):=\mu(A_t)$. For every $t\in[0,s]$, we have shown
\[
    u(t)\geq R_t^I(v).
\]
On the other hand, since $(A_t)_{s-t}\subseteq A_s,$ another application of the enlargement inequality yields
\[
    u(s)\geq R_{s-t}^I(u(t)).
\]
It can be shown that the family $(R_t^I)_{t\geq0}$ satisfies the semigroup relation
\[
    R_{s-t}^I\bigl(R_t^I(v)\bigr)=R_s^I(v).
\]
Moreover, because $R_s^I(v)<1$, the map
$w\mapsto R_{s-t}^I(w)$ is strictly increasing on the relevant interval.
Thus, if $u(t)>R_t^I(v)$ for some $t\in[0,s]$, then
\[
    u(s)\geq R_{s-t}^I(u(t))
    >R_{s-t}^I\bigl(R_t^I(v)\bigr)
    =R_s^I(v),
\]
a contradiction. Therefore,
\[
    \mu(A_t)=u(t)=R_t^I(v),
    \qquad 0\leq t\leq s.
\]
Taking $t=0$ gives $\mu(A_0)=\mu(A)$. Furthermore,
\[
    \mu^+(A)
    =\liminf_{h\to0^+}
      \frac{\mu(A_h)-\mu(A)}{h}=\lim_{h\to0^+}
      \frac{R_h^I(v)-v}{h}
     =I(v).
\]
Hence $A$ is an extremal set for \eqref{eq:extremal_iso}.
\end{proof}
To provide an example, the Gaussian isoperimetric inequality mentioned in \eqref{eq:gamma_iso} is usually expressed as, for $A\subset \R^n$ Borel with $0<\gamma_n(A)<1$,
\[
\gamma_n^+(A)\geq I_\gamma\left(\gamma_n(A)\right), \qquad I_\gamma=\Psi^\prime_\gamma\circ \Psi_\gamma^{-1}.
\]

The above lemma, in conjunction with Theorem~\ref{thm:general_enlargement_theorem}, allows an immediate passage from isoperimetric information directly to metric-Santal\'o inequalities.

\begin{corollary}[Metric-Santal\'o from Differential Isoperimetry]
\label{cor:differential_isoperimetry_santalo}
Let $(X,d,\mu)$ be a Polish probability space which admits a continuous isoperimetric function $I: (0,1) \to (0,\infty)$. Let $(R_s^I)_{s\geq  0}$ be the family of enlargement functions induced by $I$ via Lemma~\ref{l:isop_to_enl}, and let 
\begin{equation}
\label{eq:theta_I_def}
\Theta^I(s) = \sup_{0 \le v \le 1} \sqrt{v\bigl(1 - R_s^I(v)\bigr)}.\end{equation}
Suppose $\Theta$ is continuous and decaying to zero at infinity.
Then:
\begin{enumerate}
    \item The space $(X,\mu)$ satisfies the metric-Santal\'o inequality induced by $2\cdot \Theta^I$ on the domain $(\mathcal{B}(X), \mathcal{B}(X))$.
    \item For any sub-collections $\mathcal{A}, \mathcal{D} \subset \mathcal{B}(X)$ of essentially disjoint sets ($\mu(A \cap D) = 0$ for all $A \in \mathcal{A}$ and $D \in \mathcal{D}$), $(X,\mu)$ satisfies the sharper metric-Santal\'o inequality induced by $\Theta^I$ on the domain $(\mathcal{A}, \mathcal{D})$.
\end{enumerate}
\end{corollary}

When the isoperimetric function $I$ is symmetric, the enlargement
function $R_s^I$ can be expressed in terms of a symmetric distribution function.

\begin{proposition}
\label{p:symmetric_isop_enlargement}
Let $(X, d, \mu)$ be a Polish probability space with continuous isoperimetric function $I$. Then,
\begin{equation}
\label{eq:sym_iso}
    I(v)=I(1-v),
    \qquad v\in(0,1).
\end{equation}
holds if and only if the function $\Psi$ given by \eqref{eq:Psi_def} satisfies
\begin{equation}
\label{eq:Psi_reflexive}
    1-\Psi(t)=\Psi(-t),
    \qquad t\in\R.
\end{equation}
Moreover, the function $\Psi$ is log-concave (resp. strictly log-concave) on $\Phi((0,1))$ if and only if $\frac{I(v)}{v}$ is non-increasing (resp. decreasing) on $(0,1)$.
\end{proposition}

\begin{proof}
If $v=\frac{1}{2}$, the identity $\Phi(1-v)=-\Phi(v)$ is immediate.
Let $v\in(0,\frac{1}{2})$. Then, the change of variables $r\mapsto 1-r$ and the symmetry of $I$ give
\[
    \Phi(1-v)
    =
    \int_{\frac{1}{2}}^{1-v}\frac{dr}{I(r)}
    =
    \int_v^{\frac{1}{2}}\frac{dr}{I(1-r)}
    =
    \int_v^{\frac{1}{2}}\frac{dr}{I(r)}
    =
    -\Phi(v).
\]
Similarly, if $v\in\left(\frac{1}{2},1\right)$, then
\[
    \Phi(1-v)
    =
    -\int_{1-v}^{\frac{1}{2}}\frac{dr}{I(r)}
    =
    -\int_{\frac{1}{2}}^v\frac{dr}{I(1-r)}
    =
    -\int_{\frac{1}{2}}^v\frac{dr}{I(r)}
    =
    -\Phi(v).
\]
Therefore,
\[
    \Phi(1-v)=-\Phi(v),
    \qquad v\in(0,1).
\]
Consequently, for every $t\in\Phi((0,1))$,
\[
    \Psi(-t)=1-\Psi(t).
\]
We extend this identity to $t\in \R$ by continuously extending $\Psi$ as constants beyond $\Phi((0,1))$, which yields the claim. The converse direction follows by differentiation. 

Finally, suppose that the map \[ v\mapsto \frac{I(v)}{v} \] is non-increasing (resp. decreasing) on $(0,1)$. On the interval $\Phi\left((0,1)\right)$, we have $\Psi=\Phi^{-1}$ and hence \[ \Psi'(t)=I(\Psi(t)). \] Therefore, \[ \frac{d}{dt}\log\Psi(t) =\frac{I(\Psi(t))}{\Psi(t)}. \] Since $\Psi$ is non-decreasing (resp. increasing), the right-hand side is non-increasing (resp. decreasing) in $t$. Thus $\log\Psi$ is concave on $\Phi\left((0,1)\right)$.
\end{proof}

We now determine the function $\Theta^I$ explicitly in this symmetric case.

\begin{proposition}
\label{p:realization_of_theta}
Let $(X, d, \mu)$ be a Polish probability space. Suppose that either of the following equivalent conditions holds (by Lemma~\ref{l:isop_to_enl} and Proposition~\ref{p:symmetric_isop_enlargement}):
\begin{enumerate}
    \item it has a continuous isoperimetric function $I$ such that \eqref{eq:sym_iso} holds and that $v\mapsto \frac{I(v)}{v}$ is non-increasing;
    \item it has $R_s^I(v) = \Psi\left(\Psi^{-1}(v) +  s\right)$ as an enlargement function, where the distribution function $\Psi$ defined in \eqref{eq:Psi_def} is log-concave and is such that \eqref{eq:Psi_reflexive} holds.
\end{enumerate}
Then, we have 
\begin{equation}
\label{eq:theta_evaluated}
    \Theta^I(s)=\Psi\left(-\frac{s}{2}\right),
\end{equation}
where $\Theta^I$ is defined in \eqref{eq:theta_I_def}.
\end{proposition}

\begin{proof}
We first compute the
function $\Theta^I$ defined in \eqref{eq:theta_I_def}.

Let $s\geq0$ and let $v\in(0,1)$. Write
\[
    v=\Psi(a)
\]
for some $a\in\R$. Using \eqref{eq:Psi_reflexive}, we evaluate the expression
inside the supremum defining $\Theta(s)$:
\[
    v\bigl(1-R_s^I(v)\bigr)
    =
    \Psi(a)\Bigl(1-\Psi(a+s)\Bigr)
    =
    \Psi(a)\Psi(-a-s).
\]
By the log-concavity of $\Psi$,
\begin{equation}
\label{eq:log_concave_midpoint}
    \sqrt{v\bigl(1-R_s^I(v)\bigr)}
    =
    \sqrt{\Psi(a)\Psi(-a-s)}
    \leq
    \Psi\left(\frac{a+(-a-s)}{2}\right)
    =
    \Psi\left(-\frac{s}{2}\right).
\end{equation}
The right-hand side is independent of $a$, and the choice $a=-\frac{s}{2}$ gives equality in \eqref{eq:log_concave_midpoint}. Therefore, $\Theta^I(s)=\Psi\left(-\frac{s}{2}\right),$ as claimed.
\end{proof}

We explicitly mention that we will combine Proposition~\ref{p:realization_of_theta} with Corollary~\ref{cor:differential_isoperimetry_santalo}. Moreover, we can combine Proposition~\ref{p:realization_of_theta} with Lemma~\ref{l:general_enlargement_santalo} to obtain a sharp geometric Santal\'o inequality in terms of $\Psi$; we leave these details to the reader.

    We can now immediately prove the a functional Santal\'o for even log-concave probability measures on the real line.
\begin{corollary}
\label{cor:functional_1d_symmetric}
Let $\mu$ be an even, log-concave probability measure on $\R$ with density. 

For any pair of measurable functions $f, g:\R \to [0,\infty)$ and any non-decreasing measurable function $W: \R_+ \to (0, \infty)$ satisfying the pointwise relation
\begin{equation}
\label{eq:pointwise_cond_1d_sym}
f(x)^{\frac{1}{2}} g(y)^{\frac{1}{2}} \leq W\left(\frac{|x-y|}{2}\right) \quad \text{for all } x, y \in \R,
\end{equation}
one has the functional inequality
\begin{equation}
\label{eq:functional_ineq_1d_sym}
\left( \int_{\R} f(x) \, d\mu(x) \right)^{\frac{1}{2}} \left( \int_{\R} g(y) \, d\mu(y) \right)^{\frac{1}{2}} \leq  \int_{\R} W(|r|) \,d\mu(r).
\end{equation}
\end{corollary}
    \begin{proof}
    We apply Corollary~\ref{cor:differential_isoperimetry_santalo} to the Polish probability space $(\R, |\cdot|, \mu)$. Denote by
    \[
    \Psi(t) = \int_{-\infty}^t \varphi(x) \, dx, \quad \text{for } t \in \R,
    \]
    the cumulative distribution function of $\mu$. Here, $\varphi(x)=\frac{d\mu}{dx}$ is the Radon-Nikodym derivative of $\mu$, which exists and is log-concave on its support, by Borell's classification of log-concave measures \cite{Bor75}. Moreover, $\Psi$ is log-concave. We also note that, because $\mu$ is even, we have $\varphi(x) = \varphi(-x)$, and, therefore:
    \[
    1 - \Psi(t) = \int_t^\infty \varphi(x) \, dx = \int_{-\infty}^{-t} \varphi(x) \, dx = \Psi(-t).
    \]

    By a classic result of Bobkov \cite{SB96} , $\mu$ satisfies isoperimetry with its isoperimetric function $I=\varphi \circ \Psi^{-1}$; the isoperimetric extremal sets are half-lines. By Lemma~\ref{l:isop_to_enl}, $\mu$ satisfies the enlargement inequality \eqref{eq:enlarge_ineq} with $R_s(v)=\Psi\left(\Psi^{-1}(v)+s\right)$ for all $s$.

    Let $\widetilde W(r)=W\left(\frac r2\right)$. Then the pointwise
assumption becomes
\[
    f(x)^{\frac12}g(y)^{\frac12}
    \leq
    \widetilde W(|x-y|).
\]
 By Corollary~\ref{cor:differential_isoperimetry_santalo} and Proposition~\ref{p:realization_of_theta}, applied with the metric profile
\[
    G(s)=2\Theta^I(s)=2\Psi\left(-\frac{s}{2}\right),
\]
we obtain
\[
\left( \int_{\R} f \,d\mu \right)^{\frac{1}{2}}
\left( \int_{\R} g \,d\mu \right)^{\frac{1}{2}}
\leq
\int_0^\infty \widetilde W(r)\,d(-G)(r).
\]
Since
\[
    d(-G)(r)
    =
    -\frac{d}{dr}\left(2\Psi\left(-\frac r2\right)\right)dr
    =
    \varphi\left(\frac r2\right)dr,
\]
we get
\[
\left( \int_{\R} f \,d\mu \right)^{\frac{1}{2}}
\left( \int_{\R} g \,d\mu \right)^{\frac{1}{2}}
\leq
\int_0^\infty
W\left(\frac r2\right)
\varphi\left(\frac r2\right)\,dr.
\]
Changing variables $u=r/2$ gives
\[
    \int_0^\infty
    W\left(\frac r2\right)
    \varphi\left(\frac r2\right)\,dr
    =
    2\int_0^\infty W(u)\varphi(u)\,du.
\]
Since $\varphi$ is even,
\[
    2\int_0^\infty W(u)\varphi(u)\,du
    =
    \int_{\R}W(|u|)\,d\mu(u).
\]
This proves \eqref{eq:functional_ineq_1d_sym}.
\end{proof}

\subsection{RCD Spaces} 
\label{sec:RCD}
The theory of $\operatorname{RCD}$ spaces is a rich topic. We will follow the work \cite{AM16}; please see the references therein. We will denote by 
$$\mathcal{P}_2(X):=\left\{\mu \in \mathcal{P}(X): \int_X d^2\left(x_0, x\right) \mathrm{d} \mu(x)<\infty\text{ for some (and hence all) }x_0 \in X\right\},$$
the subspace of $\mathcal{P}(X)$ consisting of all the probability measures with finite second moment. For $\nu_0, \nu_1 \in \mathcal{P}_2(X)$ the Wasserstein-2 distance $W_2\left(\nu_0, \nu_1\right)$ is defined by
\begin{equation}
\label{eq:wasser}
W_2^2\left(\nu_0, \nu_1\right)\!=\!\!\!\inf_{\pi \in \mathcal{P}(X\times X)}\left\{\int_{X\times X} \!\!\!\!\!\!\!d^2(x,y) \, d\pi(x,y): \pi(\cdot\times X)= \nu_0(\cdot), \pi(X\times \cdot)=\nu_1(\cdot)\right\},
\end{equation}
where the infimum is taken over all $\pi$ with $\nu_0$ as the first marginal and $\nu_1$ as the second marginal. We will denote by $\operatorname{Geo}(X)$ the space of all constant speed geodesics $\gamma:[0,1] \rightarrow X$. Finally, denote by $\operatorname{dom}_\mu\left(H\right)$ the elements $\nu \in \mathcal{P}_2(X)$ for which the relative entropy $H(\nu|\mu)<\infty$.

\begin{definition}[$\mathrm{CD}(K, \infty)$ condition] Given $K \in \mathbb{R}$, we say that a Polish probability space $(X, d,\mu)$ is a $\operatorname{CD}(K, \infty)$-space if for any two measures $\nu_0, \nu_1 \in \operatorname{dom}_\mu\left(H\right)$ there exists a geodesic $\left(\nu_t\right) \in \operatorname{Geo}\left(\mathcal{P}_2(X)\right)$ which satisfies the convexity inequality
\begin{equation*}
H(\nu_t|\mu) \leq(1-t) H(\nu_0|\mu)+t H(\nu_1|\mu)-\frac{K}{2} t(1-t) W_2^2\left(\nu_0, \nu_1\right), \quad \text { for all } t \in[0,1] .
\end{equation*}
\end{definition}

Next, if $\operatorname{LIP}(X)$ denotes the set of Lipschitz functions on $X$, then if $f \in\operatorname{LIP}(X)$, its metric slope is
\[
\operatorname{lip}f(x):=\limsup_{y\to x}\frac{|f(x)-f(y)|}{d(x,y)},
\]
with the convention that $\operatorname{lip}f(x)=0$ for an isolated point $x$. Using this notion, the Cheeger energy of $f\in L^2(X,\mu)$ is
\[
\operatorname{Ch}_\mu(f):=\inf\left\{\liminf_{n\to \infty}\frac{1}{2}\int_X \operatorname{lip}^2f_nd\mu\,\bigg|\, \{f_n\}\subset\operatorname{LIP}(X),\,\| f_n- f\|_{L^2(X,\mu)}\xrightarrow{n\to \infty} 0 \right\}.
\]
\begin{definition}[$\mathrm{RCD}(K, \infty)$ condition] Let $(X,d,\mu)$ be a Polish measure space. We say that it is an $\operatorname{RCD}(K, \infty)$-space if it satisfies the $\mathrm{CD}(K, \infty)$-condition and if its Cheeger energy is quadratic, i.e., for all $f,g\in L^2(X,\mu)$ such that $0<\operatorname{Ch}_\mu(f)\operatorname{Ch}_\mu(g)<\infty$, one has
\[
\operatorname{Ch}_\mu(f+g) + \operatorname{Ch}_\mu(f-g)=2\operatorname{Ch}_\mu(f) + 2\operatorname{Ch}_\mu(g).
\]
\end{definition}

The result from \cite[Theorem 4.2]{AM16} implies that, if $(X, d,\mu)$ is an $\operatorname{RCD}(K, \infty)$ probability space for $K>0$, then Gaussian isoperimetry holds: for every non-empty, measurable $A\subset X$ such that $0<\mu(A)<1$,
\begin{equation}
\label{eq:RCD_iso}
\mu(A_r)\geq \Psi_\gamma\left(a+\sqrt{K}r\right), \quad a =\Psi_\gamma^{-1}(\mu(A)),
\end{equation}
where, again, $\Psi_\gamma$ is the cumulative distribution function of the one-dimensional standard Gaussian measure. Formally speaking, the reference provided establishes the version of \eqref{eq:RCD_iso} in the form of \eqref{eq:extremal_iso}, with $I=\sqrt{K}I_\gamma$; this equivalent form follows by Lemma~\ref{l:isop_to_enl}.

\begin{corollary}[Functional Santal\'o inequalities on $\operatorname{RCD}(K, \infty)$-spaces]
    \label{cor:RCD}
    Let $(X,d,\mu)$ be an $\operatorname{RCD}(K,\infty)$ probability space for some $K > 0$. For any pair of measurable functions $f, g: X \to [0,\infty)$ and any non-decreasing measurable function $W: \R_+ \to (0, \infty)$ satisfying the pointwise relation
    \begin{equation}
    \label{eq:pointwise_cond_rcd}
    f(x)^{\frac{1}{2}} g(y)^{\frac{1}{2}} \leq W\left(\frac{d(x,y)}{2}\right) \quad \text{for all } x, y \in X,
    \end{equation}
    one has the functional inequality
    \begin{equation}
    \label{eq:functional_ineq_rcd}
    \left( \int_X f(x) \, d\mu(x) \right)^{\frac{1}{2}} \left( \int_X g(y) \, d\mu(y) \right)^{\frac{1}{2}} \leq  \int_{\R} W\left(\frac{|r|}{\sqrt{K}}\right) \, d\gamma(r).
    \end{equation}
\end{corollary}
\begin{proof}
Let $\Psi_\gamma$ denote the distribution function of the standard Gaussian
measure on $\R$. By the $\operatorname{RCD}(K,\infty)$ isoperimetric inequality
\eqref{eq:RCD_iso}, we may invoke Corollary~\ref{cor:differential_isoperimetry_santalo} and
Proposition~\ref{p:realization_of_theta} with
\[
    G(s)=2\Theta^I(s)    =
    2\Psi_\gamma\left(-\frac{\sqrt K\,s}{2}\right).
\]

Therefore, if we set
\[
    \widetilde W(s)=W\left(\frac{s}{2}\right),
\]
the pointwise assumption \eqref{eq:pointwise_cond_rcd} becomes
\[
    f(x)^{\frac12}g(y)^{\frac12}
    \leq
    \widetilde W(d(x,y))
\]
and implies the functional metric-Santal\'o inequality
\[
    \left( \int_X f \,d\mu \right)^{\frac12}
    \left( \int_X g \,d\mu \right)^{\frac12}
    \leq
    \int_0^\infty \widetilde W(s)\,d(-G)(s).
\]
Since
\[
    d(-G)(s)
    =
    -\frac{d}{ds}
    \left[
        2\Psi_\gamma\left(-\frac{\sqrt K\,s}{2}\right)
    \right]ds
    =
    \sqrt K\,\Psi_\gamma'\left(\frac{\sqrt K\,s}{2}\right)ds,
\]
we obtain
\[
    \left( \int_X f \,d\mu \right)^{\frac12}
    \left( \int_X g \,d\mu \right)^{\frac12}
    \leq
    \int_0^\infty
    W\left(\frac{s}{2}\right)
    \sqrt K\,\Psi_\gamma'\left(\frac{\sqrt K\,s}{2}\right)ds.
\]
With the change of variables $r=\frac{\sqrt K\,s}{2}$, this becomes
\[
    2\int_0^\infty
    W\left(\frac{r}{\sqrt K}\right)d\gamma(r).
\]
By the evenness of $\gamma$ and the fact that $r\mapsto W(|r|/\sqrt K)$ is
even, this equals
\[
    \int_{\R}W\left(\frac{|r|}{\sqrt K}\right)d\gamma(r),
\]
which proves \eqref{eq:functional_ineq_rcd}.
\end{proof}

\subsection{Spherical Blaschke-Santal\'o}
\label{sec:sbs}
In this section, we aim to derive an analogue of the functional Blaschke-Santal\'o inequality on the sphere, completing the example from the introduction. First, some notation: for a non-empty, measurable set $A\subset \mathbb{S}^n$, we set
\[
r_A:=\Sigma_n^{-1}\left(\sigma_n(A)\right)
\]
and let $C_A$ be the spherical cap of radius $r_A$ with center at the north pole $e_{n+1}$ (without loss of generality); that is,
\begin{equation}
\label{eq:caps}
\sigma_n(C_A)=\Sigma_n(r_A)=\sigma_n(A).
\end{equation}
We show that $\Sigma_n$ is log-concave.

\begin{proposition}
\label{p:log_check}
    For every $n \geq 2$, the function $\Sigma_n$ is strictly log-concave on $(0,\pi)$.
\end{proposition}
\begin{proof}
    First, we show that $\sin^{n-1}$ is strictly log-concave. Let $$h(\theta)=-\log \sin^{n-1}(\theta)=-(n-1)\log\sin(\theta).$$ We want to show that this function is strictly convex. Differentiating twice, we see that $h^{\prime\prime}(\theta)=(n-1)\left(\frac{1}{\sin(\theta)}\right)^2$, which is strictly positive on $(0,\pi)$ for $n \geq 2$.

    Let $c=\int_0^\pi \sin^{n-1}(\theta) \, d\theta$ and define the log-concave function $f(\theta)= \frac{1}{c}\sin^{n-1}(\theta)$, so that $\Sigma_n(r)=\int_0^r f(\theta) \,d\theta$. For $-\log V$ to be convex, we must have, by differentiating twice,
    \begin{equation}
    \label{eq:log_test}
        f^{\prime}(r)\cdot \Sigma_n(r) < f^2(r).
    \end{equation}
    Notice that $f^2$ is strictly positive on $(0,\pi)$. However, $$f^\prime(r)=\frac{1}{c}(n-1)\sin^{n-2}(r)\cos(r),$$ is non-positive on $[\frac{\pi}{2},\pi)$. Therefore, on this interval, \eqref{eq:log_test} holds trivially. For the interval $(0,\frac{\pi}{2})$ we use that $f$ is log-concave to deduce the usual tangent line bound (of the concave function $\log f$)
    \[
    f(\theta) < f(r)\exp\left(\frac{f^\prime(r)}{f(r)}\left(\theta-r\right)\right).
    \]
    Here, we used that, since $f$ is strictly log-concave, the inequality is strict. Integrating in $\theta$ yields
    \[
    \Sigma_n(r) \leq f(r)\int_0^r\exp\left(\frac{f^\prime(r)}{f(r)}\left(\theta-r\right)\right) \,d\theta = \frac{f(r)^2}{f^\prime(r)}\left(1-e^{-\frac{f^\prime(r)}{f(r)}r}\right)<\frac{f(r)^2}{f^\prime(r)},
    \]
    which completes the proof. We mention that the inequality is strict, since $f^\prime>0$.
\end{proof}

We now use the transference principle to establish a functional version of this inequality.

\begin{proof}[Proof of Corollary~\ref{cor:spherical_functional}]
Let
\[
    d(x,y)=\arccos\langle x,y\rangle
\]
be the geodesic distance on $\mathbb S^n$. By the spherical isoperimetric
inequality \eqref{eq:sph_iso}, Corollary~\ref{cor:differential_isoperimetry_santalo},
and Proposition~\ref{p:realization_of_theta}, if we define the metric profile,
\[
    H(s)
    =
    \Sigma_n\left(\frac{\pi-s}{2}\right),
    \qquad 0\leq s\leq \pi,
\]
and $H(s)=0$ for $s>\pi$, then the Polish probability space $(\mathbb S^n, d,\sigma_n)$ satisfies the metric-Santal\'o inequality with
profile $2H$.

Since
\[
    \langle x,y\rangle=\cos d(x,y),
\]
we reparametrize the profile by setting
\[
    G(r)=H(\arccos r)
    =
    \Sigma_n\left(\frac{\pi-\arccos r}{2}\right),
    \qquad -1\leq r\leq 1.
\]
Equivalently, the Lebesgue-Stieltjes measure $dG$ is the pushforward of $d(-H)$ under the
map $s\mapsto \cos s$. Hence, for every measurable $\Omega$,
\[
    \int_0^\pi \Omega(\cos s)\,d(-H)(s)
    =
    \int_{-1}^1 \Omega(r)\,dG(r).
\]

Now the pointwise assumption gives
\[
    f(x)^{\frac12}g(y)^{\frac12}
    \leq
    \Omega(\langle x,y\rangle)
    =
    \Omega(\cos d(x,y)).
\]
Notice that $\Omega\circ \cos$ is non-decreasing. Applying the metric-Santal\'o inequality with profile $2H$ and then using the
change of variables above yields
\[
    \left(\int_{\mathbb S^n} f\,d\sigma_n\right)^{\frac12}
    \left(\int_{\mathbb S^n} g\,d\sigma_n\right)^{\frac12}
    \leq
    2\int_0^\pi \Omega(\cos s)\,d(-H)(s)
    =
    2\int_{-1}^1 \Omega(r)\,dG(r),
\]
which proves \eqref{eq:sph_func_BS_general}.

Finally, assume that $f(x)g(y)=0$ whenever $\langle x,y\rangle>0$. Then the
positive superlevel sets of $f$ and $g$ are essentially disjoint: indeed, if a
point belonged to both, then taking $x=y$ would give $\langle x,y\rangle=1>0$,
contradicting the assumption. Therefore the sharp metric-Santal\'o inequality
applies with profile $H$ instead of $2H$. Repeating the same change of variables
gives
\[
    \left(\int_{\mathbb S^n} f\,d\sigma_n\right)^{\frac12}
    \left(\int_{\mathbb S^n} g\,d\sigma_n\right)^{\frac12}
    \leq
    \int_0^\pi \Omega(\cos s)\,d(-H)(s)
    =
    \int_{-1}^1 \Omega(r)\,dG(r),
\]
which proves \eqref{eq:sph_func_polar_bound}.
\end{proof}

In \cite{GHS01}, Gao, Hug and Schneider introduced the polarity 
\begin{equation}
    A^s=\left\{u\in\mathbb{S}^n:\langle u,a\rangle \leq 0 \,\,\, \forall \,\, a\in A\right\}, \qquad A\subset \mathbb{S}^n.
\end{equation}
This polarity fits with the framework utilized in Proposition~\ref{p:dualities}, choosing the cost $c(x,y)=-e^{\langle x,y\rangle}$. Moreover, in the terminology of Section~\ref{sec:set_prop}, this is a convex duality, as the sublevel sets of this cost function are geodesically convex on the sphere. In \cite[Corollary]{GHS01}, Gao, Hug and Schneider proved the following spherical Blaschke-Santal\'o inequality.

\begin{proposition}
\label{p:cap_BS}
   Let $A\subset \mathbb{S}^n$ be a measurable set with positive measure. Then,
\begin{equation}
\sigma_{n}(A^s) \leq \sigma_{n}(C_A^s).
\label{eq:sph_BS}
\end{equation}
If $0<\sigma_n(A) < \frac{1}{2}$, there is equality in \eqref{eq:sph_BS} if and only if $A$ is a spherical cap up to removal of null sets. If $\sigma_n(A)\geq \frac{1}{2}$, there is always equality, namely, both sides of \eqref{eq:sph_BS} are zero.
\end{proposition}
We now show that the spherical isoperimetric inequality \eqref{eq:sph_iso} itself implies \eqref{eq:sph_BS}.

\begin{proof}[Proof of Proposition~\ref{p:cap_BS}]
    Recall that $C_A$ is a spherical cap with radius $r_A$ satisfying $$\sigma_n(A) = \sigma_n(C_A).$$ 
    
    Suppose $\sigma_n(A) < \frac{1}{2}$, which implies $r_A \leq \pi/2$. By the definition of spherical polarity, any point $y \in A^s$ satisfies $\langle x, y \rangle \leq 0$ for all $x \in A$. Consequently, the spherical distance, $d(x,y) = \arccos(\langle x, y\rangle)$, between any point in $A$ and any point in $A^s$ is bounded below by $\pi/2$. 
    
    Therefore, $A^s$ must be contained in the complement of the open $\pi/2$-neighborhood of $A$, denoted $A_{\pi/2} = \{z \in \mathbb{S}^n : d(z, A) < \pi/2\}$. This yields the volume bound:
    \begin{equation} \label{eq:polar_complement_bound}
        \sigma_n(A^s) \leq \sigma_n(\mathbb{S}^n \setminus A_{\pi/2}) = 1 - \sigma_n(A_{\pi/2}).
    \end{equation}

    By the spherical isoperimetric inequality \eqref{eq:sph_iso}, the volume of the $\pi/2$-neighborhood of $A$ is bounded below by the volume of the $\pi/2$-neighborhood of $C_A$. That is, $$\sigma_n(A_{\pi/2}) \geq \sigma_n((C_A)_{\pi/2}).$$ 
    Substituting this into \eqref{eq:polar_complement_bound} gives:
    \begin{equation}\label{eq:polar_isoperimetric_bound}
        \sigma_n(A^s) \leq 1 - \sigma_n((C_A)_{\pi/2}) = \sigma_n(\mathbb{S}^n \setminus (C_A)_{\pi/2}).
    \end{equation}
    
    We now analyze the right-hand side. The open neighborhood $(C_A)_{\pi/2}$ is a cap of radius $r_A + \pi/2$. Its complement $\mathbb{S}^n \setminus (C_A)_{\pi/2}$ is the closed opposing cap of radius $\pi - (r_A + \pi/2) = \pi/2 - r_A$. Thus, the volume of the complement is $\Sigma_n(\pi/2 - r_A)$.
    
    On the other hand, we claim that $C_A^s$ is exactly this complement cap. Indeed, if $C_A$ is centered at a pole $e_{n+1}$, a point $y$ belongs to $C_A^s$ if and only if its distance to every point in $C_A$ is at least $\pi/2$. Since the furthest points of $C_A$ from $e_{n+1}$ are at distance $r_A$, $y$ must be at a distance of at least $\pi/2 + r_A$ from $e_{n+1}$. This defines an opposing cap centered at $-e_{n+1}$ with radius $\pi - (\pi/2 + r_A) = \pi/2 - r_A$. 
    
    Therefore, 
    \[
    \sigma_n(\mathbb{S}^n \setminus (C_A)_{\pi/2}) = \sigma_n(C_A^s).
    \] 
    Combining this with \eqref{eq:polar_isoperimetric_bound} yields $\sigma_n(A^s) \leq \sigma_n(C_A^s)$, establishing \eqref{eq:sph_BS}.
    
    Now, we analyze the case of equality when $\sigma_n(A) < 1/2$. If $A$ is a spherical cap (up to a null set), then it is a rotation of $C_A$. Since rotations preserve volume under polarity, we have $\sigma_n(A^s) = \sigma_n(C_A^s)$. 
    
    Conversely, suppose $\sigma_n(A^s) = \sigma_n(C_A^s)$. Then equality must hold in our application of the spherical isoperimetric inequality: $\sigma_n(A_{\pi/2}) = \sigma_n((C_A)_{\pi/2})$. By the equality conditions of the spherical isoperimetric inequality \eqref{eq:sph_iso}, this implies that $A$ must be a spherical cap up to a set of measure zero.

    Finally, if $\sigma_n(A) \geq 1/2$, then the radius of the equivalent cap $C_A$ satisfies $r_A \geq \pi/2$. Consequently, the open $\pi/2$-neighborhood of $C_A$ has full Haar measure, meaning $\sigma_n((C_A)_{\pi/2}) = 1$. By the spherical isoperimetric inequality \eqref{eq:sph_iso}, we must also have $\sigma_n(A_{\pi/2}) = 1$. Since the polar sets are contained in the complements of these neighborhoods, both $\sigma_n(A^s)$ and $\sigma_n(C_A^s)$ are zero, and the inequality \eqref{eq:sph_BS} holds trivially as $0 \leq 0$. 
\end{proof}

\subsection{Hamming Cube} 
\label{sec:hcube}

Let $n \geq 2$. We denote by $Q_n=\{0,1\}^n$ the Hamming cube in $\R^n$ equipped with the Hamming distance
\[
d_{H}(x,y) = \sum_{i=1}^n|x_i-y_i|. 
\]
Notice that $d_H$ takes only the values $0,\dots,n$. The associated norm is the $\ell_1$ norm
\[
\|x\|_1=\sum_{i=1}^nx_i.
\]
We equip the Hamming cube with the uniform probability measure $\mu_n$. 

We denote by $\# A$ the cardinality of $A\subset Q_n$. Harper's vertex isoperimetric theorem (see \cite[Chapter~3]{AGA}) then states that, if $\# A=m$, one has
\begin{equation}
\label{eq:ham_iso}
\# A_k \geq \#\left( I_m\right)_k, \text{ for } k=0,\ldots,n.
\end{equation}
where $A_k=\{x\in Q_n:d_H(x,A)\leq k\}$ and, for $m=0,\ldots,2^n$ the symbol $I_m$ denotes the first $m$ elements of $Q_n$ in simplicial order: first by cardinality, and lexicographically inside each cardinality layer.

In this section, we define the continuous cost  function $$c(x,y)=-\left(n-d_H(x,y)\right)$$ and we define the profile function $F_n$ by $F_n(-1):=0$,
    \[
    F_n(j):= \max_{0\leq m\leq 2^n}\left[\frac{m}{2^n} \left(1-\frac{\#(I_m)_{n-j-1}}{2^n}\right) \right]^{\frac{1}{2}}, \ j=0,\ldots, n-1,
    \]
and $F_n(n):=1$. The following result is a sharp Boolean geometric Santal\'o inequality.
\begin{lemma}
\label{l:ham}
 Using the above conventions, set $M(A,B)=\sup_{x\in A, y\in B} -c(x,y)$ for non-empty $A,B\subset Q_n$. Then, 
    \begin{equation} \label{eq Boolean BS}
        \sqrt{\mu_n(A)\mu_n(B)}\leq F_n\left( M(A,B) \right).
    \end{equation}
    It is sharp in the sense that, if $j<n$, then 
\begin{equation}\label{eq:DiscreteFunction}
   F_n(j)=\max\{\sqrt{\mu_n(A)\mu_n(B)}: d_H(A,B)\geq n-j\}.
\end{equation}
\end{lemma}
\begin{proof}
    If $M=n$, the inequality \eqref{eq Boolean BS} is immediate. Assume $M\leq n-1$. In particular, for each $x\in A$ and $y\in B$ we have
    \[
    d_H(x,y)\geq n-M\geq 1.
    \]
    If $k=n-M-1$, then $B\cap A_k=\emptyset$. Therefore, if $m=\#A$, we have by \eqref{eq:ham_iso} \begin{align*}
        \mu_n(B)\leq 1-\mu_n(A_k)\leq 1-\frac{\#(I_m)_k}{2^n}.
    \end{align*} It follows that
    \[
    \sqrt{\mu_n(A)\mu_n(B)}\leq \left( \frac{m}{2^n} \left(1-\frac{\#(I_m)_k}{2^n}\right) \right)^{\frac{1}{2}}.
    \]

    It remains to prove sharpness. Fix $j<n$, and let $m$ be chosen so that the
maximum in the definition of $F_n(j)$ is attained. Set
\[
    A=I_m
    \qquad\text{and}\qquad
    B=Q_n\setminus (I_m)_{n-j-1}.
\]
Then, by construction,
\[
    d_H(A,B)\geq n-j.
\]
Moreover,
\[
    \mu_n(A)=\frac{m}{2^n}
    \qquad\text{and}\qquad
    \mu_n(B)=1-\frac{\#(I_m)_{n-j-1}}{2^n}.
\]
Therefore,
\[
    \sqrt{\mu_n(A)\mu_n(B)}
    =
    \left[
        \frac{m}{2^n}
        \left(
            1-\frac{\#(I_m)_{n-j-1}}{2^n}
        \right)
    \right]^{\frac12}
    =
    F_n(j).
\]
This proves \eqref{eq:DiscreteFunction}.
\end{proof}

We now establish a functional Boolean Santal\'o inequality. We set $$\Delta F_n(j):=F_n(j)-F_n(j-1),$$ where $F_n$ is as in \eqref{eq:DiscreteFunction}.

\begin{theorem} \label{t:ham}
   Using the above conventions, let $\Omega:\{0,\ldots,n\}\to [0,\infty)$ be a non-increasing function and let $f,g:Q_n\to [0,\infty)$. Then, if 
    \[
    \sqrt{f(x)g(y)}\leq \Omega (-c(x,y)), \qquad \, \forall x,y \in Q_n\]
    we have
    \[ \sqrt{\int_{Q_n} f \ d\mu_n\int_{Q_n} g \ d\mu_n}\leq \sum_{j=0 }^{n} \Omega(j) \Delta F_n(j).
    \]
\end{theorem}
\begin{proof}
    Set $\bar{F}_n:(-1,n]\to [0,\infty)$ to be the piecewise affine extension of $F_n$, and  $\bar{\Omega}:(-1,n]\to [0,\infty)$, as $\bar{\Omega}(t):=\Omega(j)$, if $t\in (j-1,j]$. Notice $\bar F_n$ is non-decreasing. Then we see that
    \[
    \sum_{j=0 }^{n} \Omega(j) \Delta F_n(j)=\int_{-1}^n\bar{\Omega}\ d \bar{F_n}.\]
    Therefore, the claim follows by Lemma~\ref{l:ham} and Lemma~\ref{l:transfer}.
    \end{proof}

Formally speaking, Lemma~\ref{l:ham} and Theorem~\ref{t:ham} also follow from Theorem~\ref{thm:general_enlargement_theorem}, by using $\Theta(k)=F_n(n-k)$. We chose to use Lemma~\ref{l:transfer} for the proof, as it seemed simpler in this case.

\section{Applications to Optimal Transport}
\label{sec:opt_trans}
\subsection{History} 
Let $c: X\times Y \to(-\infty,\infty]$ be a measurable cost function, bounded from below and lower semi-continuous. Then, the optimal transport cost between $\nu_1\in \mathcal{P}(X)$ and $\nu_2\in \mathcal{P}(Y)$ is defined as 
\begin{equation}
\label{eq:costs}
    \mathcal{T}_c(\nu_1,\nu_2) = \inf_{\pi \in \mathcal{P}(X\times Y)}\left\{\int_{X\times Y} c(x,y) \, d\pi(x,y): \pi(\cdot\times Y)= \nu_1(\cdot), \pi(X\times \cdot)=\nu_2(\cdot)\right\}.
\end{equation}

Kantorovich \cite{KL40,KL48} showed that the infimum in \eqref{eq:costs} is attained and that it can alternatively be given as follows:
\begin{equation}
\label{eq:kantor}
    \mathcal{T}_c(\nu_1,\nu_2)= \sup_{\substack{\varphi\in L^1(X,d\nu_1), \psi\in L^1(Y,d\nu_2)\\ \varphi(x)+\psi(y) \leq c(x,y), \forall x\in X, y\in Y }}\left\{\int_X \varphi(x) \, d\nu_1(x) + \int_Y \psi(y) \, d\nu_2(y) \right\}.
\end{equation}

We will need the dual formula for relative entropy: supposing that $\nu\in \mathcal{P}(X)$ and $\nu$ has density $f$ with respect to $\mu$ such that $\log f\in L^1(\nu)$, then
\begin{equation}
\label{eq:dual_rel_entropy}
H(\nu|\mu) = \sup_{g\in L^1(\nu)}\left\{\int_{X} g \, d\nu - \log \int_{X} e^g \, d\mu\right\}.
\end{equation}

We now discuss the pertinent history of inequalities for optimal transport costs in the usual case of $X=Y=\R^n$. Following the method put forth by Marton \cite{MK96}, Talagrand \cite{TM96} established the Gaussian transport inequality: if $X=Y=\R^n$ and $c(x,y)=|x-y|^2$, then
\begin{equation}
\label{eq:Talagrand}
\mathcal{T}_c(\nu,\gamma_n) \leq 2H(\nu|\gamma_n),\quad\forall \nu\in\mathcal{P}(\R^n),
\end{equation}
with equality when the density of $\nu$ is a translate of the density of $\gamma_n$. It turns out that $\mathcal{T}_c(\nu_1,\nu_2)^\frac{1}{2}$ (which is a Wasserstein-2 distance between $\nu_1$ and $\nu_2$, compare \eqref{eq:wasser}) satisfies the triangle inequality. This observation combined with \eqref{eq:Talagrand} implies
\begin{equation}
\mathcal{T}_c(\nu_1,\nu_2) \leq 4H(\nu_1|\gamma_n) + 4H(\nu_2|\gamma_n), \quad \forall\nu_1,\nu_2\in\mathcal{P}(\R^n),
\label{eq:triangle_talagrand}
\end{equation}
with equality when $\nu_1$ and $\nu_2$ are standard Gaussians with opposite means. Fathi \cite{F18}, using Proposition~\ref{prop:Ball_Fradelizi_Meyer} in the form of Lehec \cite{JL09}, established a symmetrized version of this inequality: suppose that either $\nu_1$ or $\nu_2$ has center of mass at the origin. Then, 
\[
\mathcal{T}_c(\nu_1,\nu_2) \leq 2H(\nu_1|\gamma_n) + 2H(\nu_2|\gamma_n).
\]
In \cite{FGSZ24}, a method from \cite{BG99} and Proposition~\ref{prop:Ball_Fradelizi_Meyer} were used to establish transport-entropy inequalities similar to \eqref{eq:Talagrand}. It follows from \cite[Theorem 3.1]{FGSZ24} that, if $\Omega:\R_+\to(0,\infty)$ is a continuous, non-increasing, log-concave function such that $\int_0^\infty \Omega(t^2) t^{n-1}dt<+\infty$, and if $\mu_\Omega$ is the probability measure on $\R^n$ with density proportional to $\Omega(|x|^2)$, then
\begin{equation}
\label{eq:FGSZ}
    \mathcal{T}_{c_\Omega}(\nu_1,\nu_2) \leq H(\nu_1|\mu_\Omega) + H(\nu_2|\mu_\Omega), \quad \forall\, \nu_1,\nu_2\in\mathcal{P}(\R^n),
\end{equation}
where 
\[
c_\Omega(x,y) = \log\left(\frac{\Omega(|\langle x,y \rangle|)^2}{\Omega(|x|^2)\Omega(|y|^2)}\right), \quad x,y\in\R^n.
\]

\subsection{New results on optimal transport} 
This section is dedicated to generalizing \eqref{eq:FGSZ} to arbitrary Polish spaces using the transference principle, and subsequently applying it to the space of matrices. 

We first establish a universal transport-entropy inequality for spaces satisfying the functional condition of Lemma~\ref{l:transfer}. To construct a non-negative transport cost, we utilize reference functions satisfying a $c$-duality bound, analogous to the $c$-Legendre transforms described in Section 1.4.

\begin{theorem}
\label{thm:abstract_transport}
    Let $(X_1, \mu_1)$ and $(X_2, \mu_2)$ be Polish measure spaces, and let $c$ be a continuous cost function on $X_1\times X_2$ such that $-c(X_1 \times X_2)=[a,b)$. Fix a weight $\lambda \in (0,1)$. Assume that $(X_1, \mu_1)$ and $(X_2,\mu_2)$ satisfy the Cost-Santal\'o inequality in Lemma~\ref{l:transfer} associated with $(\lambda,c,F)$ and domain $(\mathcal{A}_1, \mathcal{A}_2)$, where $F:[a,b]\to [0,\infty)$ is a continuous, non-decreasing function such that $F(a)=0$.  Assume also that the total masses 
    \[
    0<Z_1 = \int_{X_1} \Omega(\Phi_1(x_1)) d\mu_1(x_1),\, Z_2 = \int_{X_2}\Omega(\Phi_2(x_2)) d\mu_2(x_2) <+\infty.
    \]
    
    Let $\Phi_1: X_1 \to [a,b]$ and $\Phi_2: X_2 \to [a,b]$ be measurable functions satisfying
    \begin{equation}
        \label{eq:duality_bound_abstract}
        \lambda \Phi_1(x_1) + (1-\lambda) \Phi_2(x_2) \ge -c(x_1, x_2), \quad \forall x_1\in X_1, x_2\in X_2.
    \end{equation}
    Let $\Omega: [a,b] \to (0,\infty)$ be a continuous, non-increasing, log-concave function. Moreover, assume the following balance condition holds:
    \begin{equation}
    \label{eq:balance_cond}
        \int_a^b \Omega(r) dF(r) \leq Z_1^\lambda Z_2^{1-\lambda}.
    \end{equation}
    Define the probability measures
    \[
    dP_1(x_1) = Z_1^{-1} \Omega(\Phi_1(x_1))d\mu_1(x_1) \quad \text{and} \quad dP_2(x_2) = Z_2^{-1} \Omega(\Phi_2(x_2)) d\mu_2(x_2)
    \]
    and define the transport cost function on $X_1 \times X_2$ by:
    \begin{equation}
        c_{\Omega}(x_1, x_2) = \log \left( \frac{\Omega(-c(x_1, x_2))^{\frac{1}{1-\lambda}}}{\Omega(\Phi_1(x_1))^{\frac{\lambda}{1-\lambda}} \Omega(\Phi_2(x_2))} \right).
    \end{equation}
    Then $c_\Omega \ge 0$, and for any $\nu_1 \in \mathcal{P}(X_1)$ having density with respect to $P_1$, and $\nu_2 \in \mathcal{P}(X_2)$ having density with respect to $P_2$, one has the optimal transport bound:
    \begin{equation}
        \mathcal{T}_{c_{\Omega}}^{\mathcal{A}}(\nu_1, \nu_2) \leq \frac{\lambda}{1-\lambda} H(\nu_1 | P_1) + H(\nu_2 | P_2),
    \end{equation}
    where $\mathcal{T}_{c_{\Omega}}^{\mathcal{A}}(\nu_1, \nu_2)$ denotes the restricted Kantorovich cost taken as the supremum over pairs $(\varphi, \psi)$ satisfying the pointwise bound $\varphi(x_1) + \psi(x_2) \le c_\Omega(x_1, x_2)$, with the additional constraint that $e^{\frac{1-\lambda}{\lambda}\varphi} \Omega(\Phi_1(\cdot)) \in \mathcal{F}(\mathcal{A}_1)$ and $e^{\psi} \Omega(\Phi_2(\cdot)) \in \mathcal{F}(\mathcal{A}_2)$.
\end{theorem}

\begin{proof}
    We first observe that $c_\Omega$ is non-negative. By the duality bound \eqref{eq:duality_bound_abstract} and the fact that $\Omega$ is non-increasing, we have $\Omega(-c(x_1, x_2)) \ge \Omega(\lambda \Phi_1(x_1) + (1-\lambda)\Phi_2(x_2))$. Because $\Omega$ is log-concave, $\Omega(\lambda \Phi_1 + (1-\lambda)\Phi_2) \ge \Omega(\Phi_1)^\lambda \Omega(\Phi_2)^{1-\lambda}$. Raising both sides to the power of $\frac{1}{1-\lambda}$ and taking the logarithm ensures $c_\Omega(x_1, x_2) \ge 0$.

    By the duality formula for relative entropy \eqref{eq:dual_rel_entropy}, we can express the entropy terms as a supremum over all integrable measurable functions $\varphi$ and $\psi$:
    \begin{align*}
    &\frac{\lambda}{1-\lambda} H(\nu_1 | P_1) + H(\nu_2 | P_2) \\
    &= \sup_{\varphi + \psi \leq c_\Omega} \left\{ \int_{X_1} \varphi\, d\nu_1 + \int_{X_2} \psi \,d\nu_2 - \left( \frac{\lambda}{1-\lambda} \log \int_{X_1} e^{\frac{1-\lambda}{\lambda}\varphi} dP_1 + \log \int_{X_2} e^\psi dP_2 \right) \right\}.
    \end{align*}
    
    Let $\mathcal{A}_{\text{admis}}$ be the set of pairs $(\varphi, \psi)$ that satisfy both the Kantorovich pointwise bound $\varphi(x_1) + \psi(x_2) \leq c_\Omega(x_1, x_2)$ and the domain constraint $f(\cdot) = e^{\frac{1-\lambda}{\lambda}\varphi(\cdot)} \Omega(\Phi_1(\cdot)) \in \mathcal{F}(\mathcal{A}_1)$ and $g(\cdot) = e^{\psi(\cdot)} \Omega(\Phi_2(\cdot)) \in \mathcal{F}(\mathcal{A}_2)$. Restricting the supremum to this set gives:
    \begin{align*}
    &\frac{\lambda}{1-\lambda} H(\nu_1 | P_1) + H(\nu_2 | P_2) \\
    &\ge \sup_{(\varphi, \psi) \in \mathcal{A}_{\text{admis}}} \left\{ \int_{X_1} \varphi \,d\nu_1 + \int_{X_2} \psi\, d\nu_2 - \left( \frac{\lambda}{1-\lambda} \log \int_{X_1} e^{\frac{1-\lambda}{\lambda}\varphi} dP_1 + \log \int_{X_2} e^\psi dP_2 \right) \right\}.
    \end{align*}

    For any valid pair $(\varphi, \psi) \in \mathcal{A}_{\text{admis}}$, the Kantorovich constraint implies:
    \[
    f(x_1)^\lambda g(x_2)^{1-\lambda} = e^{(1-\lambda)\varphi(x_1)} e^{(1-\lambda)\psi(x_2)} \Omega(\Phi_1(x_1))^\lambda \Omega(\Phi_2(x_2))^{1-\lambda} \leq \Omega(-c(x_1, x_2)).
    \]
    Since $f \in \mathcal{F}(\mathcal{A}_1)$ and $g \in \mathcal{F}(\mathcal{A}_2)$, we can apply Condition (2) of Lemma~\ref{l:transfer} to the unnormalized measures $\mu_1$ and $\mu_2$, by setting $\Omega(b):=\lim_{r\to b^-}\Omega(r)$,
    \[
    \left( \int_{X_1} f d\mu_1 \right)^\lambda \left( \int_{X_2} g d\mu_2 \right)^{1-\lambda} \leq \int_a^b \Omega(r) dF(r).
    \]
    Inserting the normalized probability measures $P_1$ and $P_2$, and utilizing the balance condition \eqref{eq:balance_cond}, this becomes:
    \[
    \left( \int_{X_1} e^{\frac{1-\lambda}{\lambda}\varphi} dP_1 \right)^\lambda \left( \int_{X_2} e^\psi dP_2 \right)^{1-\lambda} \leq 1.
    \]
    Taking the logarithm of both sides divided by $1-\lambda$, we obtain:
    \[
    \frac{\lambda}{1-\lambda} \log \int_{X_1} e^{\frac{1-\lambda}{\lambda}\varphi} dP_1 + \log \int_{X_2} e^\psi dP_2 \leq 0.
    \]
    Since we now know the sign of this term, we may drop it in our entropy bound and continue the inequality in the correct direction:
    \begin{align*}
    \frac{\lambda}{1-\lambda} H(\nu_1 | P_1) + H(\nu_2 | P_2) &\ge \sup_{(\varphi, \psi) \in \mathcal{A}_{\text{admis}}} \left\{ \int_{X_1} \varphi d\nu_1 + \int_{X_2} \psi d\nu_2 \right\} \\
    &= \mathcal{T}_{c_\Omega}^{\mathcal{A}}(\nu_1, \nu_2).
    \end{align*}
    This establishes the desired restricted optimal transport bound.
\end{proof}

\begin{rem}
    If the cost-Santal\'o inequality holds for the full domain of measurable sets, the constraint $\mathcal{A}_{\text{admis}}$ can be dropped. In this case, the proof above directly establishes the following sharper inequality for the true optimal transport cost:
    \begin{align*}
    &\frac{\lambda}{1-\lambda}H(\nu_1|P_1) + H(\nu_2|P_2) 
    \\
    &\geq \mathcal{T}_{c_\Omega}(\nu_1,\nu_2) - \sup_{\varphi+\psi \leq c_\Omega}\left\{\frac{\lambda}{1-\lambda}\log\int_{X_1} e^{\frac{1-\lambda}{\lambda}\varphi}dP_1 + \log\int_{X_2} e^{\psi}dP_2\right\}.
    \end{align*}
\end{rem}

We now apply Theorem~\ref{thm:abstract_transport} to derive an optimal transport inequality for matrices.

\begin{proof}[Proof of Corollary~\ref{c:mass_trans}]
    We apply the framework of Theorem~\ref{thm:abstract_transport} with $X_1 = \R^n$ and $X_2 = M_{n,m}(\R)$ equipped with standard Lebesgue measures $dv$ and $dA$. We set $\lambda = \frac{m}{m+1}$, meaning $\frac{\lambda}{1-\lambda} = m$. Define the underlying continuous cost $c(v, A) = -h_Q(A^t v)$.
    
    Define the functions $\Phi_1(v) = |v|^{\frac{m+1}{m}}$ and $\Phi_2(A) = \|A\|^{m+1}_{(\B)^{\star,Q}}$. By the gauge definition \eqref{eq:polar_ball}, we have $\|A\|_{(\B)^{\star,Q}} \ge h_Q\left(A^t\frac{v}{|v|}\right)$ for any non-zero $v$; applying Young's inequality in this case with conjugate exponents $p = m+1$ and $q = \frac{m+1}{m}$ yields the duality bound \eqref{eq:duality_bound_abstract}:
    \begin{equation}
    \frac{m}{m+1}|v|^{\frac{m+1}{m}} + \frac{1}{m+1}\|A\|^{m+1}_{(\B)^{\star,Q}} \ge |v| h_Q\left(A^t\frac{v}{|v|}\right) = h_Q(A^tv) = -c(v,A).
    \end{equation}
    
    By taking the $(m+1)$th root of the hypothesis of Theorem~\ref{t:main}, the spaces satisfy the Cost-Santal\'o inequality associated with the domain $(\mathcal{C}^n,\mathcal{B}\left(M_{n,m}(\R)\right))$. Furthermore, the right-hand side of Theorem~\ref{t:main} guarantees the required balance condition is satisfied. Thus, we invoke Theorem~\ref{thm:abstract_transport} for the restricted transport cost. Inserting our definitions, the transport cost evaluates to:
    \[
    c_{\Omega}(v,A)\! = \!\log\left( \frac{\Omega(-c(v,A))^{\frac{1}{1-\lambda}}}{\Omega(\Phi_1(v))^{\frac{\lambda}{1-\lambda}} \Omega(\Phi_2(A))} \right) \!=\! \log\left( \frac{\Omega(h_Q(A^tv))^{m+1}}{\Omega(|v|^{\frac{m+1}{m}})^m \Omega(\|A\|^{m+1}_{(\B)^{\star,Q}})} \right) \!\!=\! c_{\Omega,m}(v,A).
    \]
    
    This establishes the restricted transport bound (1). To establish (2), assume $Q$ is origin-symmetric and $\nu_1, \nu_2$ are even measures. Then, $c_{\Omega,m}(-v, -A) = c_{\Omega,m}(v, A)$. We claim that the usual cost $\mathcal{T}_{c_{\Omega,m}}$ coincides with $\mathcal{T}^{\mathcal{F}}_{c_{\Omega,m}}$ given our assumptions, which yields (2) from (1). 

    Indeed, recall that, for the former, we take the supremum over all integrable pairs $(\varphi, \psi)$ such that $\varphi(v) + \psi(A) \le c_{\Omega,m}(v, A)$. We first note that, because $\nu_1, \nu_2$ and $c_{\Omega,m}$ are even, we may replace any valid pair of potentials with their even parts. Setting $\tilde{\varphi}(v) = \frac{1}{2}(\varphi(v) + \varphi(-v))$ and $\tilde{\psi}(A) = \frac{1}{2}(\psi(A) + \psi(-A))$, the evenness of the cost ensures the pointwise bound $\tilde{\varphi}(v) + \tilde{\psi}(A) \le c_{\Omega,m}(v, A)$ holds, and the evenness of the measures ensures the integrals with respect to $\nu_1$ and $\nu_2$ remain unchanged.  Thus, the Kantorovich supremum can be evaluated entirely over even potentials.
    
    Moreover, by standard approximation arguments, we may further restrict the supremum to potentials that are bounded from above. For such an admissible pair, define the upper semi-continuous envelopes
    \[
    \overline{\varphi}(v):=\limsup_{w\to v}\varphi(w)
    \qquad\text{and}\qquad
    \overline{\psi}(A):=\limsup_{B\to A}\psi(B).
    \]
    Since $c_{\Omega,m}$ is continuous, for every $v\in\mathbb R^n$ and $A\in M_{n,m}(\mathbb R)$, we have
    \[
    \overline{\varphi}(v)+\psi(A) \leq \limsup_{w\to v}c_{\Omega,m}(w,A)\leq c_{\Omega,m}(v,A).
    \]
    Iterating, we get 
    \[
    \overline{\varphi}(v)+\overline{\psi}(A)\leq c_{\Omega,m}(v,A).
    \]
    Thus, $(\overline\varphi,\overline\psi)$ remains admissible. Moreover, $\overline\varphi\geq\varphi$ and $\overline\psi\geq\psi$, so this replacement does not decrease the supremum. The envelopes remain bounded above and even. Consequently, the Kantorovich supremum may be evaluated over bounded-above, even, upper semi-continuous potentials.

    For such potentials, the functions
    \[f(v)= e^{\frac{\varphi(v)}{m}}
    \Omega\left(|v|^{\frac{m+1}{m}}\right) \quad \text{and} \quad g(A)=e^{\psi(A)}\Omega\left(\|A\|_{(B_2^n)^{\star,Q}}^{m+1}\right)\]
    are upper semicontinuous and even. Since the potentials are bounded above and $\Omega$ vanishes at infinity, both $f$ and $g$ vanish at infinity. Their nontrivial superlevel sets are therefore compact and origin-symmetric, and hence have barycenter at the origin.

    Thus, we have illustrated that the supremum defining $\mathcal{T}_{c_{\Omega,m}}$ can be taken over even $f \in \mathcal{F}_n$ and even, measurable $g:M_{n,m}(\R)\to\R_+$, and we conclude.
\end{proof}

\textbf{Acknowledgments.} We thank Dario Cordero-Erausquin, Matthieu Fradelizi, Julian Haddad, and Eli Putterman for the discussions regarding some of the geometric properties in Section~\ref{sec:properties}.

\end{document}